\documentclass[12pt]{amsart}
\usepackage[utf8]{inputenc}
\usepackage{amsmath}
\usepackage{amssymb}
\usepackage{amsthm}
\usepackage{xspace}
\usepackage{breqn}
\usepackage{mathrsfs}
\usepackage{tikz-cd}
\usepackage{enumitem}
\usepackage{blindtext}
\usepackage{numprint}
\usepackage{seqsplit}
\usepackage[all]{xy}

\newcommand{\Z}{{\mathbb Z}}
\newcommand{\Q}{{\mathbb Q}}
\newcommand{\PP}{{\mathbb P}}

\usepackage{setspace}
\usepackage{graphicx}
\graphicspath{ {./} }
\usepackage[OT2,T1]{fontenc}

\DeclareSymbolFont{cyrletters}{OT2}{wncyr}{m}{n}
\DeclareMathSymbol{\Sha}{\mathalpha}{cyrletters}{"58}
\usepackage[margin=1in]{geometry}

\newtheorem{theorem}{Theorem}[section]
\newtheorem{corollary}[theorem]{Corollary}
\newtheorem{proposition}[theorem]{Proposition}
\newtheorem{lemma}[theorem]{Lemma}

\theoremstyle{definition}

\newtheorem{definition}[theorem]{Definition}
\newtheorem{remark}[theorem]{Remark}

\newtheorem{example}[theorem]{Example}

\providecommand{\mb}[1]{\mathbb{ #1} }
\providecommand{\mc}[1]{\mathcal{ #1} }
\providecommand{\mr}[1]{\mathrm{ #1} }

\newcommand{\bu}{\bullet}

\begin{document}
\title{Capitulation discriminants of genus one curves}

\author{Lazar~Radi\v{c}evi\'{c}}
\email{lazaradicevic@gmail.com}
\maketitle

\begin{abstract}
 In this paper we study the arithmetic and invariant theory of genus one normal curves embedded in $\PP^{n-1}$. We generalize the notion of genus one model of degree $n$, introduced by Cremona, Fisher and Stoll for $n 
\leq 5$, to arbitrary odd $n$, and describe the invariant theory of a genus one curve of degree $n$ embedded in $
\PP^{n-1}$ in terms of the minimal graded free resolution of its homogeneous ideal. We prove that everywhere locally soluble genus one curves over $
\Q$ admit minimal integral models, with the same invariants as those of the minimal model of their Jacobian elliptic curve.

We then apply these results to study the capitulation problem for the Tate-Shafarevich group $\Sha(E/\Q)$ of an elliptic curve $E/\Q$. We prove that every element of $\Sha(E/\Q)[n]$ of odd index $n$ splits over a degree $n$ number field $K$, of absolute discriminant at most $c(n) H_E^{2n-2}$, where $H_E$ is the naive height of $E$ and $c(n)$ is a constant only depending on $n$. 
\end{abstract}

\tableofcontents
\section{Introduction} \label{intro}
Let $E$ be an elliptic curve defined over $\mathbb{Q}$, $n 
\geq 2$ an integer. The short exact sequence of $n$-descent relates the Mordell-Weil group $E(\Q)/nE(\Q)$, $n$-Selmer group $\mathrm{Sel}^{(n)}(E/\mb{Q})$ and the $n$-torsion of the Tate-Shafarevich group $\Sha(E/\mb{Q})$.

\[
0 \xrightarrow{} E(\mb{Q})/nE(\mb{Q}) \xrightarrow{} \mathrm{Sel}^{(n)}(E/\mb{Q}) \xrightarrow{} \Sha(E/\mb{Q})[n] \xrightarrow{} 0   .
\]
These groups have well-known geometric interpretations. As a set, $\Sha(E/\mb{Q})$ is defined to be the set of (isomorphism classes of) torsors under $E$ which have points everywhere locally. Non-trivial elements of $\Sha(E/\mb{Q})$ are counter-examples to the Hasse principle, that is, they have points everywhere locally, but no global points. 

By a result of Cassels (\cite{cassels1962arithmetic}), a torsor $C$ that represents an element of $\Sha(E/\Q)[n]$ admits a $\Q$-rational divisor $D$ of degree $n$, and so can be embedded into $\PP^{n-1}$ via the complete linear system $|D|$. We call such an embedding $C \xrightarrow{} \PP^{n-1}$ an $n$\textit{-diagram}. The $n$-Selmer group can be identified with the set of isomorphism classes of everywhere locally soluble $n$-diagrams. The image $C 
\subset \PP^{n-1}$ is a smooth, projectively normal curve of genus one and degree $n$.

A curve $C\subset \PP^{n-1}$ that represents a non-trivial element of $ \Sha(E/\Q)[n]$ does not have a $\Q$-rational point, but it has plenty of points defined over degree $n$ extensions of $\Q$, as we see by considering hyperplane sections of $C$. We say $C$ \textit{capitulates} over a number field $L$ if it  admits an $L$-rational point.

\begin{example}
Consider the following well known counterexample to the Hasse principle, due to Selmer. Let $C \subset \mb{P}^{2}$ be the plane curve defined by the equation
\[
3x^3+4y^3+5z^3=0
\]
The curve $C$ is smooth and of genus one, has points defined over every completion of $\mb{Q}$, but no points defined over $\Q$.  It capitulates over $\mb{Q}(\sqrt[3]{6})$, as we see by setting $z=0$, $y=1$, and solving for $x$ by extracting a cube root.
\end{example}

We define the \textit{capitulation discriminant} of $C$ to be the least absolute discriminant of a degree $n$ number field over which $C$ capitulates. The main result of the paper is a bound on the capitulation discriminant of $C$ in terms of its Jacobian. For an elliptic curve $E/\Q$, let $c_4(E)$ and $c_6(E)$ be the invariants of a minimal Weierstrass equation of $E$, and define the \textit{naive height} of $E$ as $H_E=\mathrm{max}(|c_4(E)|^{1/4},|c_6(E)|^{1/6})$. The following theorem is the main result of this paper. 

 \begin{theorem} \label{Main theorem I}
 	 Let $n\geq 3$ be an odd integer, and let $[C \xrightarrow{} \mb{P}^{n-1}]$ be an everywhere locally soluble $n$-diagram representing an element of the $n$-Selmer group of the elliptic curve $E$. There exists a constant $c(n)$, depending only on the integer $n$, and an order $\mc{O}$ in an $n$-dimensional commutative $\mb{Q}$-algebra $K$, such that the set of $K$-points $C(K)$ is non-empty, and we have $|\mathrm{disc}(\mc{O})| <c(n) H^{2n-2}_E$.
 \end{theorem}

The index of a smooth curve $C$ defined over a perfect field $k$, is the smallest degree a positive $k$-rational divisor $D$ on $C$ can take. The next result is a corollary of the previous one.

\begin{theorem} \label{Main theorem II}
	 Let $n\geq 3$ be an odd integer, and let $C$ be a torsor under $E$ that represents an element of $ \Sha(E/\mb{Q})[n]$. Suppose that the index of $C$ is equal to $n$. There exists a constant $c(n)$, depending only on $n$, and a degree $n$ number field $K$ of discriminant at most $c(n) H^{2n-2}_E$, such that $C$ admits a $K$-rational point.
\end{theorem}

\begin{remark}
The condition on the index of $C$ is a mild technical condition, which could probably be removed with more work. It is automatically satisfied when $n=p$ is prime and the class $c\in \Sha(E/\mb{Q})$ is non-zero---this follows from the well known fact the period and the index of an element of $\Sha(E/\mb{Q})[p]$ are equal. We also expect that the theorem is true for $n$ even as well, but the proof will require more work. However, the assumption that $C$ is everywhere locally soluble is essential, and we do not expect the theorem to hold if it is removed. 
\end{remark}

\begin{remark}	
We have computed explicit values for the constant $c(n)$ for $n=3$ and $n=5$, though for reasons of space we have not included the details here. See  Chapter 6 of \cite{LazarThesis}. These bounds imply that $\Sha(E/\Q)[n]$ is trivial for a few elliptic curves with small height. Geometry of numbers methods for proving the vanishing of $\Sha(E/\mb{Q})[3]$, similar in spirit to this one, were studied before in \cite{ncovbds}.
\end{remark}

\begin{remark}
In future work, we hope to investigate the asymptotics of the constants $c(n)$ as $n \xrightarrow{} \infty$. Here we should mention that capitulation discriminants were studied before in the work of Hriljac \cite{hriljac1987splitting}, \cite{van2015arakelov}. Hriljac shows, using Arakelov theory, that for a fixed elliptic curve $E$, the capitulation discriminant of a torsor $C$ that represents an element of $\Sha(E/\Q)[n]$ is bounded by $e^{O(n)} n^n$, where the implied constant depends on the curve $E$.   
\end{remark}

\subsection{Genus one models} The first step in the proof of Theorem \ref{Main theorem I} is to develop the invariant theory of algebraic models of genus one curves $C \subset \PP^{n-1}$. This has already been carried out when $n\leq 5$, see \cite{g1inv} and \cite{an2001jacobians}. We briefly summarise the relevant facts. For the proofs, see above references. There is a parametrization of the set of equations that define an $n$-diagram. A genus one model of degree $n \leq 5$ is

\begin{enumerate}
    \item if $n=2$ a binary quartic form,
    \item if $n=3$ a ternary cubic,
    \item if $n=4$ a pair of quadratic forms in four variables,
    \item if $n=5$ a $5\times5$ alternating matrix with entries linear forms in five variables.
    \end{enumerate}

A  generic set of equations as above defines a genus one curve $C$ equipped with a $K$-rational divisor class $[D]$ of degree $n$. In the $n=5$ case,  the equations of the curve in $\PP^4$ are the five $4 \times 4$ Pfaffians of the matrix. 

Let $X_n$ denote the (affine) space of genus one models of degree $n$. There is an action of a group $\mc{G}_n$ on the space $X_n$ such that the non-degenerate orbits in $\mc{G}_n \backslash X_n$ correspond bijectively to isomorphism classes of $n$-diagrams. The group $\mc{G}_n$ is a product of several linear groups, and acts on $X_n$ either by linear changes of coordinates on the ambient $\PP^{n-1}$, or by changing the basis of the equations that define the curve. 

By an invariant of this action we mean a homogeneous polynomial $F$ in the coordinate ring $\Q[X_n]$, such that $F(g \cdot w)=\chi(g) F(w)$ for all $g \in \mc{G}_n(\mb{C})$ and $w \in X_n(\mb{C})$, where $\chi$ is a rational character on $\mc{G}_n$.  The study of these invariants goes back to 19th century and classical invariant theory. Weil (\cite{weil1}, \cite{weil2}), observed that the invariants of a binary quartic and ternary cubic can be used to write down a Weierstrass equation for the Jacobian elliptic curve of the $n$-diagram, and this was extended to genus one models of degree $4$ and $5$ in \cite{an2001jacobians} and \cite{g1inv}. More precisely, the ring of invariants is generated by two  polynomials $c_4, c_6  \in \Z[X_n]$, with the property that the Jacobian of the curve defined by a genus one model $\Phi \in X_n$ has the Weierstrass equation:
\[
y^2=x^3-27c_4(\Phi)-54 c_6 (\Phi)
\]

Within this setup, the $n$-Selmer group of an elliptic curve $E$ can be given an invariant theoretic description, as the subset of those orbits in $\mc{G}_n(\Q) \backslash X_n(\Q)$ that define $n$-diagrams that have Jacobian $E$, and are everywhere locally soluble.

Moreover elements of the Selmer group can be represented by especially nice genus one models, with integer coefficients that are small relative to the naive height of the Jacobian. A genus one model $
\Phi$, that represents an $n$-diagram $[C \xrightarrow{} \PP^{n-1}]$, is said to be \textit{minimal} if
$\Phi$ has integer coefficients, and the invariants $c_4(\Phi)$ and $c_6(\Phi)$ are equal to invariants $c_4$ and $c_6$ of the minimal Weierstrass model of the Jacobian of $C$. In \cite{minred234} and \cite{minred5}, it is shown that if $[C \xrightarrow{} \PP^{n-1}]$ represents an element of the $n$-Selmer group of its Jacobian, i.e. $C$ is everywhere locally soluble, then $[C \xrightarrow{} \PP^{n-1}]$ admits a minimal genus one model.

The case $n=2$ of this result goes back to the work of Birch and Swinnerton-Dyer 
\cite{birch1963notes}. The existence of minimal models, and algorithms for constructing them, has applications to finding large generators of Mordell-Weil groups of elliptic curves \cite{minred234}, and in Bhargava and Shankar's work on average size of $n$-Selmer groups \cite{bhargava2013average}, \cite{bhargava2015binary}.

In this paper we generalize these results to $n>5$. The first step is to extend the definition of a genus one model of degree $n$. For $n \geq 5$, an $n$-diagram $C \subset \PP^{n-1}$ is defined by $n(n-3)/2$ quadrics, and so is not a complete intersection. A generic set of $n(n-3)/2$ quadrics does not even define a curve in $\PP^{n-1}$, and unlike the case $n \leq 5$, there is no simple parametrization of those quadrics that define an $n$-diagram. 

Nevertheless, we can generalize the invariant theory constructions to this setting, if we change our viewpoint on what a genus one model is. A common feature of the above definitions of genus one model in each of the cases $n=2,3,4,5$ is that a genus one model specifies a minimal graded free resolution of the ideal that defines the curve. Informally, the resolution contains the data of the generators of the ideal defining curve, the data of the relations that these generators satisfy, then the data of the relations that these relations satisfy, and so on.

In Section \ref{resolution model section} we define the space of resolution models as a space of chain complexes satisfying certain technical conditions. A resolution model (of degree $n$) essentially consists of a collection of matrices, with entries homogeneous polynomials, that represent the differentials of a chain complex. We extend a few well known properties of genus one models to resolution models. 

Following an idea of Fisher, we attach an alternating $n\times n$-matrix $\Omega$ to a resolution model $F_{\bu}$. The entries of $\Omega$ are quadratic forms, and the coefficients of these forms are polynomials in the coefficients of the resolution model $F_{\bu}$.  We show that there exist invariants $c_4(F_{\bu})$ and $c_6(F_{\bu})$, polynomials in the coefficients of the entries of $\Omega$. For odd $n$, we prove that the Jacobian $E$ of the curve defined by $F_{\bu}$ is given by the Weierstrass equation 

\[
y^2=x^3-27c_4(F_{\bu})-54c_6(F_{\bu}).
\]

We also construct minimal models for elements of the Selmer group. An everywhere locally soluble $n$-diagram can be represented by a resolution model $F_{\bu}$ with integer coefficients, with invariants $c_4(F_{\bu}),c_6(F_{\bu})$ equal to the invariants $c_4$, $c_6$ of the minimal Weierstrass equation of the Jacobian of $C$. 

\subsection{The discriminant form}

To bound the capitulation discriminant of $C$, we make use of the observation that the intersection of $C \subset \PP^{n-1}$ with a generic hyperplane $U$ in $\mb{P}^{n-1}$ is a set  $X_U$ of $n$ points in $\mb{P}^{n-2}$. Let $A_U$ denote the algebra of global functions on $U$. By definition $C$ capitulates over $A_U$, and the strategy for the proof of Theorem \ref{Main theorem I} will be to show that we can find a hyperplane $U$ so that the conclusions of the theorem are satisfied.

Let $x_1,\ldots,x_n$ be the coordinates on $\PP^{n-1}$ and suppose $U$ is defined by an equation $u_1x_1+\ldots+u_nx_n$. To a resolution model $F_{\bu}$ of $C$ we attach a homogeneous polynomial $D$, of degree $2n$ in $n$ variables $u_1,\ldots,u_n$, with the property that $D(u_1,\ldots,u_n)$ is the determinant of trace form on $A_U$ with respect to a certain basis. We say $D$ is the \textit{discriminant form} attached to $F_{\bu}$. Its coefficients are homogeneous polynomials in the coefficients of $F_{\bu}$. Moreover if $u_1,\ldots,u_n$ are integers, then this basis is in fact a basis of an order in $A_U$, and so $D(u_1,\ldots,u_n)$ is the discriminant of this order, and thus an upper bound on the discriminant of $A_U$.

The discriminant form is a classical object, and can also be viewed as an equation of the projective dual of $C$, as $D(u_1,\ldots,u_n)$ vanishes if and only if the hyperplane $u_1x_1+\ldots+u_nx_n$ meets some point of $C$ with multiplicity greater than one. To establish the above interpretation of the integer values of the discriminant form, a key ingredient is the main result of \cite{radfis}, which gives a construction of a multiplication table for the algebra $A_U$ from the minimal graded free resolution of the homogeneous ideal of the set $X_U$. This result is an analogue of the invariant theory for genus one models described above, and for $n \leq 5$, it recovers the well-known parametrizations of rings of rank $n$ due to Levi-Delone-Faddeev and Bhargava \cite{delone1964theory}, \cite{bhargava2004higher}, \cite{bhargava2008higher}.

\subsection{Heisenberg invariant forms and geometry of numbers}
Proving Theorem \ref{Main theorem I} now amounts to showing that there exists a constant $c(n)$ satisfying the following. Let $[C \xrightarrow{} \PP^{n-1}]$ be any everywhere locally soluble $n$-diagram, represented by a minimal resolution model $F_{\bu}$. Then there exits integers $u_1,\ldots,u_n$, not all zero, such that for the associated discriminant form $D$ we have $|D(u_1,\ldots,u_n)| \leq c(n) H^{2n-2}_E$.

To prove this, we do not work directly with the model $F_{\bu}$. Instead, we make use of the fact that the discriminant form transforms naturally under changes of coordinates on $\PP^{n-1}$, and make a $\mathrm{SL}_n(\mathbb{R})$-change of coordinates, which preserves the invariants of the resolution model that represents $C$, and put $C$ into \textit{Heisenberg normal form}. 

A curve in Heisenberg normal form is characterised by having an especially simple description of the action of the $n$-torsion of its Jacobian, and can be identified with the embedding of the Jacobian in $\PP^{n-1}$ by theta functions. This last fact allows us to view coefficients of its $\Omega$-matrix as modular forms, and by computing their $q$-expansions, we are able to bound the coefficients of the discriminant form in terms of the naive height of $E$. The final step, proving that the discriminant form takes on a small value, follows immediately from Minkowski's theorem on lattice points.

\begin{example}
Let us look at the construction of the discriminant form in the case $n=3$, where the formulas are still reasonably simple. Let $C \subset \mb{P}^2$ be a smooth curve of genus one and degree 3, defined over $\Q$, and let $F(x,y,z) \in \Q[x,y,z]$ be a ternary cubic that defines $C$. The discriminant form $D$ can be computed by the following recipe. Consider a hyperplane $H$ in $\mb{P}^2$, defined by an equation $ux+vy+wz=0$  for some $u,v,w \in K$. At least one of $u,v,w$ is non-zero---assume it is $w$. We then have the parametrization $\phi:\mb{P}^1 \xrightarrow{} H$:
\[
(x:y) \mapsto (wx:wy:-ux-vy).
\]
Composing $F$ and $\phi$, we obtain a binary cubic form $f(x,y)=F(wx,wy,-ux-vy)$. The three zeros of $f$ correspond to the three points of the intersection of $C$ and $H$. Now recall that the discriminant of a binary cubic form \[ax^3+bx^2y+cxy^2+dy^3,\] is given by \[
b^2c^2-4ac^3-4b^3d-27a^2d^2+18abcd.\]
The discriminant of the form $f$ is a homogeneous polynomial in $u,v$ and $w$, of degree 12, that is divisible by $w^6$. The discriminant form $D$ turns out to be equal to $\mathrm{disc}(f)/w^6$. If $F \in \mb{Z}[x,y,z]$ has integer coefficients, and the cubic $f$ is irreducible, it follows from the Delone-Faddeev correspondence \cite{delone1964theory} that $D(u,v,w)$ is the discriminant of an order in the number field defined by $f$. 

A general cubic form $F$ depends on 10 coefficients, and the formula for the form $D$ is fairly intimidating. However,  by making a linear change of variables via $ 
\gamma \in \mathrm{SL}_3(\mb{R})$, we may replace $F$ by the cubic $G=F \circ \gamma$, where $G$ has the simple form
\[
G=a(x^3+y^3+z^3)-3bxyz.
\]
where $a$ and $b$ are real numbers. This is known as the Hesse normal form of $F$. The discriminant form of $G$ is equal to
\begin{align*}
D_G(u,v,w)=&-27a^4(u^6+v^6+w^6) + 162a^2b^2(u^4vw+uv^4w+uvw^4)\\& + (54a^4 - 108ab^3)(u^3v^3+v^3w^3+w^3u^3) + (-324a^3b + 81b^4)u^2v^2w^2 .
\end{align*}
The discriminant form $D_F$ of $F$ is also related to $D_G$ through a change of variables corresponding to $\mathrm{SL}_3(\mb{R})$ acting on the projective space that parametrizes hyperplanes in $\PP^2$. This reduces the problem of proving that $D_F$ takes a small value at some point in $(u,v,w)$ for some $(u,v,w) \in \mb{Z}^3$, to the problem of proving that $D_G$ takes a small value at some point $(u,v,w) \in \Lambda$, where $\Lambda$ is a lattice in $\mb{R}^3$ of covolume one. This problem can now be tackled using methods from the geometry of numbers. 
\end{example}

\subsection{Organization of the paper.} In Section \ref{prelim chapter} we briefly recall basic results about minimal raded free resolutions and twists of elliptic curves we need. In Section \ref{resolution model section} we introduce the notion of \textit{resolution model} of a genus one curve $C \subset \PP^{n-1}$, generalizing the notion of \textit{genus one model}. We define the $\Omega$-quadrics associated to a resolution model, prove that they transform naturally under linear changes of coordinates on $\PP^{n-1}$, and use them to construct the discriminant form and an invariant differential on $C$. 

The main results of Section \ref{resolution model section}, Theorem \ref{formula for Jacobian} and Theorem \ref{global minimization theorem}, are proved in Section \ref{unprojection chapter} and Section \ref{analytic chapter}. In Section \ref{unprojection chapter} we explain the method of unprojection, which allows us to construct free resolutions inductively, and is the key technical tool in our proofs. The main result of Section \ref{unprojection chapter} is Theorem \ref{global minimization theorem}, generalizing the minimization theorem of \cite{minred234} to resolution models. The theorem is deduced from the local minimization result, Theorem \ref{local minimization theorem}, and Theorem \ref{formula for Jacobian}, using strong approximation. Theorem \ref{local minimization theorem} is proven by an induction, using unprojection.

 The method of unprojection was developed by Kustin and Miller in \cite{kustin1983constructing}, and used by Fisher in \cite{fisher2006higher} to compute minimal free resolutions of genus one normal curves. We refine and develop this method further in Section \ref{mapping cone section} and Section \ref{construction of the unprojection data}. Two key points for the proof of Theorem \ref{global minimization theorem} is that the method of unprojection works over base rings more general than fields, most notably $\mb{Z}$, and that it is possible to compute $\Omega$-quadrics using unprojection.

In Section \ref{analytic chapter} we recall the classical theory of embedding elliptic curves into projective space by means of theta functions, and finish the proof of Theorem \ref{formula for Jacobian}. The key point is that we are able to obtain $q$-expansions for the coefficients of the associated $\Omega$-matrix, when the $n$-diagram is in the Heisenberg normal form. This reduces the proof of the formula for Jacobian to a calculation already done in \cite{jacobians}. These $q$-expansions also play a role in Section \ref{explicit bounds chapter}, where we put everything together, and finish the proofs of Theorem \ref{Main theorem I}.

\textbf{Acknowledgements.} This paper is based on the author's PhD thesis \cite{LazarThesis}. I am deeply grateful to my PhD advisor Tom Fisher, for suggesting the problem to me, for many helpful conversations and for his patient guidance along the way. I would also like to thank Jack Thorne, Vladimir Dokchitser, Stevan Gajovic and Michele Fornea for helpful comments. Finally, I thank Trinity College and Max
Planck Institute for Mathematics for their financial support.

\section{Preliminaries} \label{prelim chapter}
\subsection{Galois cohomology and $n$-diagrams.} \label{n-diagrams}
		Let $F$ be a number field, $E/F$ an elliptic curve and $n \geq 1$ an integer. We briefly recall a few standard facts about the Galois cohomology groups $H^1(F,E)$ and $H^1(F,E[n])$, see for example \cite{descentpaper}. 
		
	A torsor under $E$ is a smooth projective curve $T/F$, together with a regular simply transitive action of $E$ on $T$.	An isomorphism of torsors $C_1$ and $C_2$ is an isomorphism of curves $C_1$ and $C_2$ that respects the action of $E$. The left action of $E$ on itself by translations makes $E$ a torsor, which we call the trivial torsor. There is a natural identification of the group $H^1(F,E)$ with the set of isomorphism classes of torsors defined over $F$, and the trivial torsor $E$ corresponds to the identity element.
		
		We will also need the following interpretation of the group $H^1(F,E[n])$. We define a diagram $[C \xrightarrow[]{} S]$ to be a morphism from a torsor $C$ to a variety $S$. An isomorphism of diagrams $[C_1 \xrightarrow{} S_1] \sim [C_2 \xrightarrow{} S_2]$ is an isomorphism
		of torsors $\phi : C_1 \cong C_2$ together with an isomorphism of varieties $\psi : S_1 \cong S_2$ making the diagram
		
		\[ \begin{tikzcd}
			C_1 \arrow{r}{} \arrow[swap]{d}{\phi} & S_1 \arrow{d}{\psi} \\%
			C_2 \arrow{r}{}& S_2
		\end{tikzcd}
		\]
		commute. We define the trivial $n$-diagram to be the diagram $[E \xrightarrow[]{} \mb{P}^{n-1}]$ where the morphism is induced by the complete linear system of the divisor $n \cdot 0_E$, and in general, we say a diagram $[C \xrightarrow[]{} S]$ is an $n$-diagram if it is defined over $F$, but isomorphic to the trivial diagram over the algebraic closure $\bar{F}$, i.e. a twist of the trivial diagram. The set of isomorphism classes is also naturally identified with $H^1(F,E[n])$.
		
        The group law on $E$ induces a summation map $\mathrm{sum} : \mathrm{Div} E \xrightarrow[]{} E$, given by $\sum n_p \cdot (P) \mapsto \sum n_pP.$ Two divisors $D$ and $D'$ of the same degree are linearly equivalent if and only $\mathrm{sum}(D)=\mathrm{sum}(D')$. The Kummer map $\delta : E(F)/nE(F) \xrightarrow{}H^1(F,E[n])$ sends a class $[P] \in E(F)/nE(F)$ to the isomorphism class of the $n$-diagram $[E \xrightarrow[]{} \mb{P}^{n-1}]$, where the map is induced by the complete linear system of any degree $n$ divisor $D$ with $\mathrm{sum}(D)=P$.
	
		In this article we consider only $n$-diagrams of the form $[C \xrightarrow[]{} \PP^{n-1}]$. When $n \geq 3$, such an $n$-diagram is a closed embedding, and its image is a smooth projectively normal curve $C$ of genus one and degree $n$. If $n=3$, $C$ is a plane cubic. For $n \geq 4$, the ideal defining $C$ is generated by $n(n-3)/2$ quadrics, see \cite[Prop.~5.3]{g1inv}. Finally, as a consequence of class field theory, under the above identification, the elements of the $n$-Selmer group of $E$ can be represented by $n$-diagrams of the form $[C \xrightarrow[]{} \mb{P}^{n-1}]$.
\subsection{Free resolutions}		
Next, we recall here some basic results from commutative algebra. The main references for this section are \cite{eisenbud}, \cite{bruns1998cohen} and \cite{kustin1983constructing}.

Let $R=k[x_0,\ldots,x_m]$, where $k$ is a field. For $M=\oplus M_d$ a graded $R$-module, let $M(c)=\oplus M_{c+d}$ be the graded $R$-module with grading shifted by $c$. We say that the direct sums of the modules $R(c)$ are the free graded $R$-modules. The notion of a free resolution of a graded module will play a key role in the rest of the paper.
 
\begin{definition}\label{def minimal free res}
	A graded free resolution of a graded $R$-module $M$ is a chain complex $F_{\bullet}$ of graded free $R$-modules
	
	\[
	F_r \xrightarrow{\phi_{r}} F_{r-1} \xrightarrow{\phi_{r-1}} \ldots \xrightarrow{\phi_2} F_1 \xrightarrow{\phi_1} F_0,
	\] 
	that is exact in degree $>0$, and has $H_0(F_{\bullet})=F_0/\phi_0(F_1) \cong M$. 
	
	Let $\mathfrak{m}=(x_1,\ldots,x_{m})$ be the maximal homogeneous ideal of $R$. We say a resolution $F_{\bullet}$ is minimal if we have $\phi_{k}(F_k) \subset \mathfrak{m} F_{k-1}$ for every $k \geq 1$. 	
\end{definition}

\subsection{Koszul resolutions.} \label{koszul def} The basic example of a minimal free resolution is the Koszul complex. We briefly recall its definition and basic properties, following \cite{eisenbud} and \cite{matsumura1989commutative}. Let $R$ be a commutative ring and let $E$ be a free $R$-module of finite rank $r$. Given an $R$-linear map $f : E \xrightarrow{} R$, the Koszul complex associated to $f$ is the chain complex of $R$-modules

\begin{align*}
	K_{\bu}(f) : 0 \xrightarrow{} \Lambda^r E \xrightarrow{d_r} \Lambda^{r-1} E \xrightarrow{} \ldots \xrightarrow{} \Lambda^{1} E \xrightarrow{d_1} R
\end{align*}
where the differential $d_k$ is given by, for any $e_1,e_2,\ldots,e_k \in E$ 
\[
d_k(e_1\wedge \ldots \wedge e_k)=\sum^{k}_{i=1} (-1)^{i+1} f(e_i) e_1 \wedge \ldots \wedge \widehat{e_i} \wedge \ldots e_k.
\]
Now let $e_i$ be a basis $E$ of and let $x_i=f(e_i)$.  
\begin{theorem} \label{Koszul theorem}
	If $x_1,\ldots,x_r$ is a regular sequence in $R$, meaning that $x_i$ is not a zero-divisor on $R/(x_1,\ldots,x_{i-1})$ for all $i$, then $K_{\bu}(f)$ is a free resolution of $R/(x_1,\ldots,x_{r})$.
\end{theorem}
\begin{proof}
	This is Theorem 16.5(i) of \cite{matsumura1989commutative} 
\end{proof}
The case of the theorem we will be using is when $R=S[x_1,\ldots,x_m]$, for an arbitrary base ring $S$ and $m \geq r$. In this case $K_{\bu}(f)$ is a graded free resolution of the ideal $(x_1,\ldots,x_r)$, with the grading given by $K_{\bu}(f)_k=(\Lambda^k E)(-k)$.

It is easy to verify that the pairings $\Lambda^i E \times \Lambda^j E\xrightarrow{} \Lambda^{i+j} E$ satisfy the Leibniz rule: for $f_i \in \Lambda^iE$ and $f_j \in \Lambda^j E$ we have
\[
d_{i+j}(f_i \wedge f_j)=d_i(f_i)\wedge f_j +(-1)^{i} f_i \wedge d_j(f_j).
\]
Furthermore $f_i \wedge f_j =(-1)^{ij} f_j \wedge f_j$. Thus the wedge product makes the Koszul complex the simplest example of differential graded commutative algebra, which we will define in Section \ref{DGCA section}.

 The Koszul complex is self-dual: we can identify $\Lambda^r E=R$, using, for example, the generator $e_1 \wedge \ldots \wedge e_r$. Then the pairings $\Lambda^i E \times \Lambda^{r-i}E \xrightarrow{} \Lambda^{r} E=R$ induces a map $\eta : K_{\bu} \xrightarrow{} K_{\bu}^{*}$, which is a map of chain complexes by the Leibniz rule, and an isomorphism as the pairings are perfect.

\subsection{Sign conventions for the dual complex.} For a graded module $M$, we  denote by $M^*$ the dual module $\mathrm{Hom}(M,R)$. Note that $R(n)^* \cong R(-n)$ as a graded module. The dual of a finite chain complex $(F_{\bullet},\phi_i)$, of length $n$, of graded $R$-modules $M$ will for us be a chain complex (as opposed to a cochain complex) of graded free $R$-modules, denoted by $F_{\bullet}^*$,  with $F^{*}_{i}=\mathrm{Hom}(F_{n-i},R)$. We adopt the sign convention that  the differential $F^{*}_{i} \xrightarrow{} F^{*}_{i+1}$ is given by $(-1)^{i}\phi^{*}_{n-i}$. 

The length of a minimal free resolution is called the projective dimension of $M$ and is denoted by $\textrm{proj dim } M$. We write $\text{codim }I$ for the codimension of a homogeneous	 ideal $I \subset R$. By the Auslander-Buchsbaum formula, we have the inequality $\textrm{codim } I \leq \textrm{proj dim } I$. We say a homogeneous ideal $I$ is perfect if $\textrm{codim } I = \textrm{proj dim } I$.

\begin{definition}
Let $I \subset R$ be a perfect ideal, and let $F_{\bullet}$ be the minimal graded free resolution of $I$. The ideal $I$ is \textit{Gorenstein} if $F_{\bullet}$ is isomorphic to its dual, up to a shift in the degree, i.e. if $F_{\bullet}\cong F^{*}(c)$, for some integer $c$.  
\end{definition}

\subsection{Algebra structures on resolutions} \label{DGCA section}
In this section we will consider graded modules over the ring $R=S[x_0,\ldots,x_m]$. In our applications, the ring of scalars $S$ will either be a field, or a ring of $p$-adic integers $\mb{Z}_p$, for some prime $p$.

Let $(C_{\bullet},d)$ be a chain complex of $R$-modules with $C_0=R$. A differential graded algebra structure on $C_{\bullet}$  is a collection of maps $C_i \otimes C_j \xrightarrow{} C_{i+j}$ that satisfies the Leibniz rule. In other words, a multiplication rule on $C_{\bullet}$ such that, if $z_i \in C_i$ and $z_j \in C_j$, then $z_i \cdot z_j \in C_{i+j}$, and
\[ \label{Leibniz rule}
d_{i+j}(z_i \cdot z_j)=d_i(z_i) \cdot z_j+(-1)^{i} z_i \cdot d_j(z_j).
\]
We also require that the map $C_0 \otimes C_i \xrightarrow{} C_i$ is just the usual multiplication $R \otimes C_i \xrightarrow{} C_i$, when we identify $C_0=R$. In other words, $1 \in C_0$ is the identity element. 

We omit the subscripts on the differential maps when there is no danger of confusion. We say the algebra is commutative if we have $z_i \cdot z_j=(-1)^{ij} z_j \cdot z_i$ for $i$ and $j$, and $z_i \cdot z_i=0$ for all $i$ odd. In this case, we say $C_{\bu}$ has a DGCA (differential graded commutative algebra) structure.

\subsection{Buchsbaum-Eisenbud construction.} Now let $F_{\bullet}$ be a finite graded free resolution of an ideal $I \subset R$, so that $F_0=R$, and $H_i(F_{\bullet})=0$ for $i>0$. There is a natural way to define a DGCA structure on $F_{\bullet}$, due to Buchsbaum and Eisenbud \cite{buchsbaum1982gorenstein}.  Define the symmetric square $S^2(F_{\bullet})$ to be 
\[
S^2(F_{\bullet})=(F_{\bullet} \otimes F_{\bullet})/M
\]
where $M$ is the submodule of $F_{\bullet} \otimes F_{\bullet}$ generated by 
\[
\{z_i \otimes z_j- (-1)^{ij} z_j \otimes z_i\ | \ z_k \in F_k\} \cup \{z_i \otimes z_i \ | \ z_i \in F_i, i \text{ odd} \}
\]
Then $S^2(F_{\bullet})$ is a chain complex, with the differential induced from the differential on $F_{\bullet} \otimes F_{\bullet}$
\[
d(z_i \otimes z_j)=dz_i \otimes z_j +(-1)^{i} z_i \otimes dz_j.
\]
The degree $k$ term can be identified with
\[
S^2(F)_k=\left( \sum_{i+j=k, \ i<j} F_i \otimes F_j \right) 	+ T_k,
\]
where
\begin{equation*}
	T_k=\begin{cases}
		0,  &\text{if}\ k  \ \text{is odd},\\
		\Lambda^2 F_{k/2}, &\text{if}\ k  \ \text{is of the form } 4n+2,\\
		S^2 F_{k/2}, &\text{if}\ k  \ \text{is of the form } 4n.\\
	\end{cases}
\end{equation*}
In particular, note that $S^2(F_{\bullet})$ is a complex of free modules. We have $S^2(F)_0=F_0 \otimes F_0$ and $S^2(F)_1=F_0 \otimes F_1$. Identifying the left $F_0$ with $R$, we obtain maps $S^2(F)_0=R \otimes F_0 \xrightarrow{} F_0$ and $S^2(F)_1=R \otimes F_1 \xrightarrow{} F_1$. These maps commute with differentials on $S^2(F_{\bullet})$ and $F_{\bullet}$. As $H_i(F_{\bullet})=0$ for $i \geq 1$, they can be extended to a map of chain complexes $\alpha: S^2(F_{\bullet}) \xrightarrow{} F_{\bullet}$. For homogeneous elements $z_i$ and $z_j$, we put $z_i \cdot z_j=\alpha(z_i \otimes z_j)$. It is clear that this defines a DGCA structure on $F_{\bullet}$.

\subsection{Construction of a null homotopy.}
Suppose that $F_{\bullet}$ and $G_{\bullet}$ are DGCA algebras, and that $\beta :F_{\bullet} \xrightarrow{} G_{\bullet}$ is a map of chain complexes. The map $\beta$ does not need to respect the DGCA structure on the nose, but it does so up to homotopy. The following result, Lemma 1.2 of \cite{kustin1983constructing}, makes this precise.

\begin{lemma} \label{null homotopy lemma}
Let $\beta : F_{\bullet} \xrightarrow{} G_{\bullet}$ be a map of complexes of free $R$-modules.
\begin{equation*} 
	\begin{tikzcd}
		\ldots \ar[r,"\phi_{n+1}"]& F_{n}\ar[r,"\phi_n"]  \ar[d,"\beta_{n}"] &\ar[r,"\phi_2"] \ldots & F_{1} \ar[r,"\phi_1"]  \ar[d,"\beta_1"] & F_0 \ar[d,"\beta_0"] \\
		\ldots \ar[r,"\psi_{n+1}"] & G_{n} \ar[r,"\psi_n"] &\ldots\ar[r,"\psi_2"] & G_1 \ar[r,"\psi_1"] & G_0
	\end{tikzcd} 
\end{equation*}
Suppose that $F_{\bullet}$ and $G_{\bullet}$ are both DGCA. If $H_i(G_{\bullet})=0$ for $i\geq 1$ , then there exists a collection of maps $\xi : F_i \otimes F_j \xrightarrow{} G_{i+j+1}$, defined for $i,j \geq 0$, such that:
\begin{enumerate}[label=(\roman*)]
	\item $\xi(z_i 	\otimes z_j)=(-1)^{ij} \xi (z_j \otimes z_i)$,
	\item $\xi(z_i \otimes z_i)=0$ if $i$ is odd,
	\item $\xi(z_0 \otimes z_i)=0$,
	\item $\beta_{i+j}(z_i \cdot z_j)-\beta_i(z_i)\cdot\beta_j(z_j)=\xi(\phi_i(z_i) \otimes z_j)+(-1)^{j}\xi(z_i \otimes \phi_j(z_j))+\psi_{i+j+1} (\xi(z_i \otimes z_j))$
	for $z_i \in F_i$.
\end{enumerate}
\end{lemma}
\begin{proof}
Consider the maps $F_i \otimes F_j \xrightarrow{} G_{i+j+1}$ defined by $z_i \otimes z_j \mapsto \beta_{i+j}(z_i \cdot z_j)-\beta(z_i)\cdot\beta(z_j)$. It is easy to check, using the Leibniz rule, that these maps induce a map of chain complexes $S_2(F_{\bullet}) \xrightarrow{} G_{\bullet}$, that is zero in degrees zero and one. Hence there is a null homotopy $\xi : F_i \otimes F_j \xrightarrow{} G_{i+j+1}$, and it is clear that $\xi$ satisfies properties (i)-(iv).
\end{proof}

\subsection{DGCA structures on self-dual resolutions.} Now suppose that $F_{\bullet}$ is of length $n-1$, so that the last non-zero module in $F$ is $F_{n-2}$, and suppose that the resolution $F_{\bullet}$ is self-dual, so that $F_{\bullet} \cong F_{\bullet}^{*}$, up to a shift in grading. We identify $R=F_0=F_{n-2}^{*}$. For each $i$, algebra multiplication $F_{i} \otimes F_{n-2-i} \xrightarrow{} R$ induces a map $s_i : F_{i} \xrightarrow{} F_{n-2-i}^{*}$. 

We then have the following statement, based on Proposition 3.4.4 of \cite{bruns1998cohen}, with a slight difference arising from our sign convention for the dual complex.
\begin{proposition} \label{DGC induced duality}
	The maps $s_i : F_i \xrightarrow{} F_{n-2-i}^{*}$ define a homomorphism of chain complexes $s: F_{\bullet} \xrightarrow{} F_{\bullet}^{*}$, lifting the isomorphism $F_0 \xrightarrow{} F_{n-2}^{*}$. This map is symmetric, meaning that for any $f_i \in F_i$ and $f_{n-2-i} \in F_{n-2-i}$ we have
	\[
	s_i(f_i)(f_{n-2-i})=(-1)^{i(n-2-i)} s_{n-2-i}(f_{n-2-i})(f_i).
	\]
	  If $R=k[x_1,\ldots,x_m]$ is the graded polynomial ring over a field $k$, and $F_{\bu}$ is a minimal free resolution, then this homomorphism is an isomorphism.
\end{proposition}
	
	\section{Resolution models and $\Omega$-matrices} \label{resolution model section}
We will be concerned with minimal graded free resolutions of the following two types of varieties. Throughout this section, let $k$ be a field.
\begin{definition}
	A genus one normal curve of degree $n$ is a smooth curve of genus one and degree $n$ that spans $\mathbb{P}^{n-1}$.
\end{definition}
In the literature, these curves are also known as elliptic normal curves.
	 
	 \begin{definition}
	 	We say a zero dimensional variety $X \subset \mb{P}^{n-2}$, defined over $k$, is a set of $n$ points in general position in $\mb{P}^{n-2}$, if $X$ is of degree $n$ and  the set of geometric points $X(\bar{k})$ consists of $n$ points in general position, meaning that no subset of $X$ of size $n-1$ is contained in a hyperplane. 
	 \end{definition}
Notably, such sets arise as hyperplane sections of genus one normal curves of degree $n$.	 
	 \begin{theorem} \label{min res theorem}
	 Let $n \geq 4$, and let $R=A[x_1,\ldots,x_m]$ be a polynomial ring over a base ring $A$. We consider chain complexes of length $n-1$ and of the form
	 	\begin{equation} \label{free res}
	 \begin{split}
	 0 \xrightarrow{} R(-n) \xrightarrow{\phi_{n-2}} R(-n+2)^{b_{n-3}}& \xrightarrow{\phi_{n-3}} R(-n+3)^{b_{n-4}} \xrightarrow{\phi_{n-4}}\ldots\\&
	 \ldots\xrightarrow{\phi_{3}} R(-3)^{b_{2}} \xrightarrow{\phi_{2}} R(-2)^{b_{1}} \xrightarrow{\phi_{1}} R \xrightarrow{} 0,
	 \end{split} \tag{$\ast$}
	 \end{equation}
	 where $b_i=n\binom{n-2}{i}-\binom{n}{i+1}$ for $1 \leq i \leq n-3$.
	 \begin{enumerate}[label=(\roman*)]
	 \item Let $C\subset \mb{P}^{n-1}$ be a genus one normal curve of degree $n$, defined over $k$, and let $I:=I(C)$ be the homogeneous ideal defining $C$.  Then $I$ is Gorenstein, and the minimal free resolution $F_{\bullet}$ of $A/I$ is of the form (\ref{free res}), with $R=k[x_1,\ldots,x_{n}]$.
	 
	\item  Let $X \subset \mb{P}^{n-2}$ be a set of $n$ points in general position. Let $I:=I(X)$ be the homogeneous ideal of $k[x_1,\ldots,x_{n-1}]$ defining $X$. Then $I$ is Gorenstein, and the minimal free resolution $F_{\bullet}$ of $A/I$ is of the form (\ref{free res}), with $R=k[x_1,\ldots,x_{n-1}]$.

\end{enumerate}
	 \end{theorem}

\begin{proof}
	Both of these results are well known, and in fact (ii) follows from (i). For a proof of (i) see \cite{bothmer2003geometric}, or Theorem 1.1 of \cite{fisher2006higher}, and for (ii), see the discussion of Theorem 138 in \cite{wilson}. As a byproduct of our work in Section \ref{unprojection chapter}, we reprove the theorem along the lines of the method used in \cite{fisher2006higher}.  
\end{proof}	
We will consider only rings $R=A[x_1,\ldots,x_m]$ with $m$ equal to either $n$, the case of curves, or $n-1$, the case of points.

\begin{definition} \label{resolution model}
	Let $S$ be an arbitrary commutative ring, $n \geq 3$ and $m \geq 2$ integers, and let $R=S[x_1,\ldots,x_{m}]$. A \textit{resolution model} of degree $n$, defined over $S$, is a collection of maps $(\phi_i)^{n-2}_{i=1}$ of graded $R$-modules, that fit together into a chain complex $F_{\bullet}$ of the form (\ref{free res})
	\begin{equation*} 
		\begin{split}
			0 \xrightarrow{} R(-n) \xrightarrow{\phi_{n-2}} R(-n+2)^{b_{n-3}}& \xrightarrow{\phi_{n-3}} R(-n+3)^{b_{n-4}} \xrightarrow{\phi_{n-4}}\ldots\\&
			\ldots\xrightarrow{\phi_{3}} R(-3)^{b_{2}} \xrightarrow{\phi_{2}} R(-2)^{b_{1}} \xrightarrow{\phi_{1}} R \xrightarrow{} 0,
		\end{split}
	\end{equation*}
	for which there exists an isomorphism $\eta : F_{\bullet} \xrightarrow{} \mathrm{Hom}(F_{\bullet},R(-n))$ of graded chain complexes, that is \textit{symmetric}, in the following sense. We have, for all $f_i \in F_i$ and $f_{n-2-i} \in F_{n-2-i}$
	\[
	\eta_i(f_i)(f_{n-2-i})=(-1)^{i(n-2-i)} \eta_{n-2-i}(f_{n-2-i})(f_i).
	\]
  When $m=n-1$, we also say $F_{\bullet}$ is a \textit{genus one model} of degree $n$.
\end{definition}

\begin{definition}
	We say that a resolution model $F_{\bullet}$ is \textit{non-degenerate} when the homology modules $H_i(F_{\bullet})$ vanish for $i>0$, so that $F_{\bullet}$ is a free resolution of an ideal $I$.  
\end{definition}

\begin{example}
	Let $C\subset \mb{P}^{n-1}$ be a genus one normal curve, defined over a field $k$. By Proposition \ref{DGC induced duality}, a minimal free resolution $F_{\bullet}$ of the coordinate ring of $C$, with a choice of basis of each module $F_i$, is a resolution model of degree $n$. We say that $F_{\bu}$ represents the $n$-diagram $[C \xrightarrow{} \mb{P}^{n-1}]$.
\end{example}

\begin{lemma} \label{uniqueness of duality lemma}
Let $F_{\bullet}$ be a chain complex of the form (\ref{free res}) that is a resolution of an ideal $I \subset S[x_1,\ldots,x_m]$. Suppose that there exists an isomorphism $\eta : F_{\bullet} \xrightarrow{} \mathrm{Hom}(F_{\bullet},R(-n))$. Then $\eta$ is unique up to scaling by the invertible elements of $S$, and $\eta$ is symmetric.
\end{lemma}
\begin{proof}
As $F_0 \cong F^{*}_0=R$, and $\eta_0$ is an isomorphism of graded modules, we see that $\eta_0$ is multiplication by an element of $S^{\times}$.  The map $\eta_0$ determines the map $\eta$ up to a chain homotopy $\gamma$. Looking at the grading of the modules $F_i$, we see that such a $\gamma$ is necessarily zero. 

Now note that the map $s$ constructed in Proposition \ref{DGC induced duality} is symmetric, and also a chain map lifting an isomorphism $R \xrightarrow{} \mathrm{Hom}(R(-n),R(-n))$. Thus the map $\eta$ is chain homotopic to an $S^{\times}$-multiple of $s$. Arguing as before, the homotopy must be zero, and $\eta$ is equal to an $S^{\times}$-multiple of $s$, and is thus symmetric.
\end{proof}

Denote the set of all resolution models defined over a ring $S$ by $X_n(S)$. A ring homomorphism  $S \xrightarrow{} S'$ induces a natural map $X_n(S) \xrightarrow{} X_n(S')$.
\begin{definition} \label{integral model definition}
Let $k$ be the field of fractions of an integral domain $S$, and let $C$ be a genus one normal curve of degree $n$, defined over $k$.  An \textit{$S$-integral model} of $C$  is a non-degenerate resolution model $F_{\bu} \in X_n(S)$ that maps into a genus one model of $C$ under the map $X_n(S) \xrightarrow{} X_n(k)$.
\end{definition}

\subsection{Equivalence of resolution models.} A resolution model includes a choice of basis for each module in the resolution, and so the group $\mc{G}_n^{res}=\mathrm{GL}_{b_{n-2}} \times \ldots \times \mathrm{GL}_{b_{0}}$ acts on the space $X_n$ of genus one models. We give an explicit description of this action. Let $F_{\bullet}$ be a resolution model defined over $S$, with the differentials represented by matrices  $\phi_{1},\ldots,\phi_{n-2}$.  For $0 \leq i \leq n-2$, an element $g_i \in \mathrm{GL}_{b_i}(S)$ acts on $F_{\bullet}$ by replacing $\phi_i$ and $\phi_{i+1}$ with $g_i \phi_i$ and $\phi_{i+1} g_i^{-1}$ respectively, with $\phi_0$ and $\phi_{n-1}$ understood to be zero.

We also have a lift of the standard action of the group $\mathrm{GL}_{m}$ on $\mathbb{P}^{m-1}$ to an action on the space $X_n$. Let $g=(g_{ij}) \in \mathrm{GL}_{m+1}$, and put $x'_j=\sum^{m}_{i=1} g_{ij} x_i $. Regard matrices $\phi_r=\phi_r(x_1,\ldots,x_{m})$ as functions of the variables $x_1,\ldots,x_{m}$, and put, for $1 \leq r \leq n-2$ 
  \[
  \phi'_r(x_1,\ldots,x_m)=\phi_r(x'_1,\ldots,x'_m)
  \]
  We define $g \cdot F_{\bullet}$ to be the resolution model specified by the differentials $\phi'_1,\ldots,\phi'_{n-2}$. We define $\mc{G}_n=\mc{G}^{res}_n \times \mathrm{GL}_{m}$, and say two models $F_{\bullet}$, $F_{\bullet}' \in X_n(S)$ are equivalent over $S$ if there exists a $g \in \mc{G}_n(S)$ with $g \cdot F_{\bullet}=F_{\bullet}'$.
  
\begin{proposition} \label{res model equiv}
	Let $F_{\bullet}$ and $F'_{\bullet}$ be resolution models that represent genus one normal curves $C \subset  \mb{P}^{n-1}$ and $C' \xrightarrow{} \mb{P}^{n-1}$, defined over a field $K$. The models $F_{\bullet}$ and $F'_{\bullet}$ are equivalent over $K$ if and only if there exists an element of $\mathrm{PGL}_{n}(K)$ that takes $C$ to $C'$.
\end{proposition}

\begin{proof}
This follows from uniqueness of minimal free resolutions.
\end{proof}

\begin{remark}
	For $n \leq 5$, our definition of a genus one model specializes to the existing definition of a genus one model given in \cite{g1inv}, i.e. the data . 
\end{remark}

	\subsection{$\Omega$-quadrics and the invariant theory of resolution models} \label{Omega quad section} 
	 Let $F_{\bullet}$ be a resolution model of degree $n$, for $R=S[x_1,\ldots,x_m]$, where $n\geq 3, m \geq 1$ are integers.  We represent the differentials $\phi_r$ of $F_{\bu}$ by matrices of homogeneous polynomials.  For any tuple $1\leq a_1,a_2,\ldots,a_{n-2} \leq m$, we define
	
	\[
	[a_1,a_2,\ldots,a_{n-2}]_{F }=\frac{\partial \phi_{1}}{\partial x_{a_1}} \frac{\partial \phi_{2}}{\partial x_{a_2}} \cdot \cdot \cdot \frac{\partial \phi_{n-2}}{\partial x_{a_{n-2}}},
	\]
	where a partial derivative of a matrix is the matrix of the partial derivatives of the entries, and the product is matrix multiplication. Note that $\phi_{n-2}$ and $\phi_{1}$ are represented by a matrices consisting of quadratic forms, while the maps  $\phi_i$, for $1<i<n-2$, are represented by matrices of linear forms, and so $[a_1,a_2,\ldots,a_{n-2}]_{F }$ is a quadratic form in $x_1,\ldots,x_n$. Let  $\sigma$ be the $(n-2)$-cycle $(12\ldots n-2)$ in $S_{n-2}$, and then define  
	\[
	[[a_1,a_2,\ldots,a_{n-2}]]_{F_{\bu} }=\sum_{k=1}^{n-2} [a_{\sigma^{2k}(1)},a_{\sigma^{2k}(2)},\ldots,a_{\sigma^{2k}(n-2)}].
	\]
	When there is no danger of confusion we will omit the subscript $F_{\bu}$. For each tuple $ (b_1,b_2,\ldots,b_{m-n+2})$, where $1 \leq b_i \leq m$ are distinct integers,  let $a_1,a_2,\ldots,a_{n-2}$ be the complement of $\{b_1,\ldots,b_{m-n+2}\}$ in $\{1,2,\ldots,m\}$, ordered so that \[(b_1,\ldots,b_{m-n+2},a_1\ldots,a_{n-2})\] is an even permutation of $\{1,\ldots,m\}$.
	\begin{definition}
	With notation as above, for each tuple $b_1,\ldots,b_{m-n+2}$ as above, we define a quadratic form
		$\Omega_{b_1,\ldots,b_{m-n+2}}=[[a_1,\ldots,a_{n-2}]]_{F_{\bu} }$.
	\end{definition}

	Let $V=\langle x_1,\ldots,x_{m} \rangle$ be the space of linear forms on $\mathbb{P}^{m-1}$, and denote the dual basis of $V^{*}$ by $x^*_1,\ldots,x_{m}^{*}$. To the collection of $\Omega$-quadrics we associate an element $\Omega$ of $\wedge^{m-n+2} V^* \otimes_{k} S^2 V $ via the formula
	\[
\Omega:=	\sum x^*_{b_1} \wedge x^*_{b_2} \wedge \ldots \wedge x^*_{b_{m-n+2}} \otimes \Omega_{b_1,\ldots,b_{m-n+2}}.
	\]
	where the sum is over all tuples $(b_1,\ldots,b_m)$ as above. The proposition below shows that this construction is well-defined, i.e. invariant under change of basis of $V$.
	
We are mainly interested in the cases $m=n$ and $m=n-1$, corresponding to resolution models of genus one curves and resolution models of sets of points respectively. In the first case, $\Omega \in \wedge^{2} V^* \otimes_{k} S^2 V $ can be represented by an alternating $n\times n$-matrix of quadratic forms in $x_1,\ldots,x_{n}$. In the second case $\Omega \in V^* \otimes_{k} S^2 V$ is represented by $n$ quadratic forms in $x_1,\ldots,x_{n-1}$. Thus according to which case we are in, $\Omega$ has either one or two subscripts. 

	\begin{remark}\label{defined up to scalar}
	The construction of $\Omega$-quadrics is independent of the choice of basis of the free $R$-modules in the resolution, except for the leftmost module $R(-n)$, where change of basis has the effect of multiplying all the $\Omega$-quadrics by the same constant. 
	
	To see this, observe that changing the basis of $F_i$ is the same as replacing $\phi_i$ by $\phi_i A$ and $\phi_{i+1}$ by $A^{-1}\phi_{i+1}$, where $A$ is a matrix with entries in $S$. Applying the product rule, we see that this does not change the expression $[a_1,a_2,\ldots,a_{n-2}]$.
\end{remark}
The following result of Tom Fisher describes how the matrix $\Omega$ transforms under the linear changes of coordinates. 
	\begin{proposition} \label{Omega quad change of coordinates}
	Let $F_{\bullet}$ be a degree $n$ resolution model. Let $x'_j=\sum_{i=1}^{m} g_{ij}x_i$ for some $g=(g_{ij}) \in \mathrm{GL}_m$. With respect to the new coordinates given by $x_j'$, $R/I$ has a free resolution $F'_{\bullet}$, again of the form (\ref{free res}), with the maps given by matrices
		$\phi'_1,\ldots,\phi'_{n-2}$ where
		\[
		\phi'_r(x_1,\ldots,x_{m})=\phi_r(x'_1,\ldots,x'_m).
		\]
		Let $\Omega$, $\Omega'$ be the elements of $\wedge^{n-m+2} V^* \otimes S^2 V $ associated to resolutions $\phi$, $\phi'$ respectively. Then we have
		\[
		\Omega'=\mathrm{det}(g) g \cdot \Omega,
		\]
		where the action of $g$ on $\Omega$ is the standard action of $\mathrm{GL}_m$ on  $\wedge^{m-n+2} V^* \otimes S^2 V $.
	\end{proposition}
We omit the proof here, since a  proof in the points case, which is exactly the same as the general case, is given in \cite{radfis}. A complete proof can be found in Section 3.4 of \cite{LazarThesis}.

	 \subsection{$\Omega$-matrices and genus one normal curves} \label{omega and curves section}
	  We now focus on the $\Omega$-quadrics associated to resolution models of genus one normal curves. In this section, we restrict to the case when the field $k$ is of characteristic 0.
	 \begin{example} \label{example omega}
	 	Let $F_{\bullet}$ be a genus one model of a genus one normal curve $C_n \subset \mb{P}^{n-1}$. Taking $m=n$, the above construction yields an element of $\wedge^2 V^{*} \otimes S^2 V$, which can be represented by an alternating matrix of quadratic forms, called the $\Omega$-matrix.
	 	
	 	\begin{enumerate}[label=(\roman*)]
	 		\item $n=3$. Suppose $C_3$ is defined by a ternary cubic $F$. The $\Omega$-matrix attached to the resolution $0\xrightarrow{} R(-3) \xrightarrow{\cdot F} R \xrightarrow{}0$ is 
	 		\[
	 		\Omega_3=\begin{pmatrix}
	 		0 & \frac{\partial F}{\partial x_3} & -\frac{\partial F}{\partial x_2} \\
	 		-\frac{\partial F}{\partial x_3} & 0 & \frac{\partial F}{\partial x_1}\\
	 		\frac{\partial F}{\partial x_2} & -\frac{\partial F}{\partial x_1} &0  
	 		\end{pmatrix}
	 		\]
	 		\item $n=4$. Suppose $C_4$ is defined by a pair of quadratic forms $F_1$ and $F_2$, and consider the free resolution given by the Koszul complex: 
	 			\[
			0 \xrightarrow{} R(-4) \xrightarrow{\phi_2} R(-2)^2 \xrightarrow{\phi_1} R \xrightarrow{} 0
			\]
			where $\phi_2=(F_2,-F_1)^T$ and $\phi_1=(F_1,F_2)$, and $F_1,F_2$.	 		Then the entries of $\Omega_4$ are given by
	 		\[
	 		(\Omega_4)_{i,j}=\frac{\partial F_1}{\partial x_k}\frac{\partial F_2}{\partial x_l}-\frac{\partial F_1}{\partial x_l}\frac{\partial F_2}{\partial x_k}
	 		\]
	 		where $(i,j,k,l)$ is an even permutation of $(1,2,3,4)$.
	 	\end{enumerate}
	 \end{example}
	
  Following the convention used in \cite{jacobians}, we denote the standard action of $\mathrm{GL}_n$ on the space $\wedge^2 V^{*} \otimes S^2 V$ by the symbol $\star$. Explicitly, it is given by
    \[
    g \star \Omega(x_1,\ldots,x_{n}) = g^{-T}\left( \Omega(\sum^{n}_{i=1} g_{i1} x_i,\ldots,\sum^{n}_{i=1} g_{in} x_i) \right) g^{-1}
    \]
 	where $g^{-T}$ is the inverse of the transpose of $g$. Note that this action differs from the action of Proposition \ref{Omega quad change of coordinates} by a factor of $\mathrm{det}(g)$. Under the $\star$-action, scalar matrices act trivially, and so $\star$ induces an action of $\mathrm{PGL}_n$ on $\wedge^2 V^{*} \otimes S^2 V$ .

 	\subsection{Invariant differentials.} Let $F_{\bu}$ be a resolution model of a genus one normal curve $C \subset \mb{P}^{n-1}$, and let $\Omega$ be the associated $\Omega$-matrix. We define the invariant differential associated to the matrix $\Omega$ to be a rational differential form on $C$, given by the expression
 	\[
 	\omega_{ij}=
 	(n-2)\frac{x_j^2 d(x_i/x_j)}{\Omega_{ij}(x_1,\ldots,x_{n})},
 	\]  
 	for any distinct $i,j$. The following Proposition is proved in Section \ref{unprojection chapter}.
 	
 	\begin{proposition}\label{first lemma differential}
 		Let $\Omega$ be an $\Omega$-matrix associated to a resolution model $F_{\bullet}$ of a curve $C \subset \mb{P}^{n-1}$.
 		\begin{enumerate}[label=(\roman*)]
 		\item The differential $\omega_{ij}$ is independent of indices $i$ and $j$, and we denote it just by $\omega$.
 		\item The differential $\omega$ is a regular differential form on $C$.
 	\end{enumerate}
 We say that the model $F_{\bullet}$ represents the pair $(C,\omega)$.
 	\end{proposition}
 
 	The differential $\omega$ transforms in  a natural way under the action of $\mathrm{GL}_n$, as we will see in Lemma \ref{invariant differential change of coords}. Now, following \cite{jacobians}, we define certain polynomials in the coefficients of the $\Omega$-matrix that are invariant under the action of $\mathrm{GL}_n$.  
 	
 	We first define, for any $1\leq i,j,k \leq n$,
 	\[
 	M_{ij}=\sum_{r,s=1}^{n}\frac{\partial \Omega_{ir}}{ \partial x_s} \frac{\partial \Omega_{js}}{\partial x_r},
 	\]
 	\[
 	N_{ijk}=\sum_{r=1}^n \frac{\partial M_{ij}}{\partial x_r} \Omega_{rk}.
 	\]
 	We then define the polynomials $c_4(\Omega)$ and $c_6(\Omega)$ as
 	\[
 	c_4(\Omega)=\frac{3}{2^4n(n-2)^2 {n+3 \choose 5}} \sum_{i,j,r,s=1}^n \frac{\partial^2 M_{ij}}{\partial x_r \partial x_s } \frac{\partial^2 M_{rs}}{\partial x_i \partial x_j}
 	\]
 	and 
 	\[
 	c_6(\Omega)=-\frac{1}{2^4n(n-2)^3 {n+5 \choose 7}} \sum_{i,j,k,r,s,t=1}^n \frac{\partial^3 N_{ijk}}{\partial x_r \partial x_s \partial x_t} \frac{\partial^3 N_{rst}}{\partial x_i \partial x_j \partial x_k}.
 	\]
 	One can show that these polynomials are invariants of $\Omega$, in the following sense:
 	\begin{proposition} \label{omega invariants}
 		For $g \in \mathrm{GL}_n(k)$ and $\Omega \in \wedge^2 V \otimes S^2 V$ , we have $c_k(g \star \Omega)=c_k(\Omega)$ for $k=4,6$.
 		
 	\end{proposition}
 	\begin{proof}
 		See Lemma 2.2 of \cite{jacobians}.
 	\end{proof}

	The reason we are interested in $\Omega$-matrices is that they can be used to give a formula for the Jacobian of the curve $C$. More precisely, we have the following.	
	
	\begin{theorem}[The formula for the Jacobian]\label{formula for Jacobian}
		Let $C/k$ be a genus one normal curve of degree $n$, where $n\geq 3$ is an odd integer, and let $\omega$ be an invariant differential of $C$. Let $F_{\bullet}$ be a genus one model of $I(C)$, and let $\Omega$ be the $\Omega$-matrix associated to $F_{\bullet}$. Then the Jacobian $E$ of $C$ is defined by the Weierstrass equation $W$
		\[
		y^2=x^3-27c_4(\Omega)x+54c_6(\Omega),
		\]
		where $c_4(\Omega)$ and $c_6(\Omega)$ are defined as above.  Moreover, there is a $\Bar{k}-$isomorphism  $\gamma:C\xrightarrow{}E$ such that,  for each $i \neq j$, 
		\[
		\gamma^*(3dx/y)= (n-2)\frac{x_j^2 d(x_i/x_j)}{\Omega_{ij}}.
		\]
	\end{theorem}
	The proof of this result will require complex analytic methods, and will be given in Section \ref{analytic chapter}.

	\begin{remark}
	In \cite{jacobians}, a different definition of the $\Omega$-matrix associated to a genus one curve is given, and the analogue of Theorem \ref{formula for Jacobian} is shown to hold for all values of $n$, without the restriction that $n$ is odd. We expect that the two definitions are equivalent. It is possible to check this is true for small values of $n$ by a generic calculation, but we still haven't found a proof that works for all $n$. We also expect that, with more work, one should be able to extend our proof of Theorem \ref{formula for Jacobian} to even values of $n$.
	\end{remark}
	
		 	\begin{remark}
 	    We have chosen a different normalisation for the differential $\omega$ and the invariants $c_4(\Omega)$ than the one used in \cite{jacobians} and \cite{LazarThesis}. We've done this because, for $n \leq 5$, with this normalisation, for an $\Omega$-matrix associated to a genus one model of degree $n$ the invariants $c_4(\Omega)$ and $c_6(\Omega)$ agree with the invariants $c_4$ and $c_6$ of the genus one model.

 		Furthermore, note that there are denominators appearing in the formulas for $c_4(\Omega)$ and $c_6(\Omega)$---they are the reason we restrict to characteristic 0 in this section. For $n \leq 5$, the invariants $c_4$ and $c_6$, viewed as functions in the coefficients of the genus one model instead of as functions of the coefficients of the matrix $\Omega$, have primitive integer coefficients.  This is another indicator that the scaling we choose is the correct one. It would be interesting to find denominator-free formulas for $c_4$ and $c_6$ for $n>5$, and we hope to investigate this in future work.
 	\end{remark}
	\subsection{The invariant differential under changes of coordinates}
	
	In this section we study how the differential $\omega$ associated to an $\Omega$-matrix behaves under changes of coordinates. The following lemma is a generalization of \cite[Proposition 5.19]{g1inv}, and immediately implies part (i) of Lemma \ref{first lemma differential}.
	
	\begin{lemma} \label{single bracket differential}
		Let $C$ be a genus one normal curve of degree $n$, let $F_{\bullet}$ be a free resolution of $R/I(C)$, and consider the associated quadrics $[a_1,\ldots,a_{n-2}]$. Then, for $i,j,a_1,\ldots,a_{n-2}$ an even permutation of $1,2,\ldots,n$ and $\tau \in S_n$, we have the following equality of rational differential forms on $C$
		\[
		\frac{x_j^2 d(x_i/x_j)}{[a_1,..,a_{n-2}]}=\mathrm{sgn}(\tau)
		\frac{x_{\tau(j)}^2 d(x_{\tau(i)}/x_{\tau (j)}) }{[\tau(a_1),\ldots,\tau(a_{n-2})]}.
		\]
	\end{lemma}
	
	We need two preliminary lemmas.
		\begin{lemma} \label{swaplemma}
		If $2\leq r \leq n-3$ then 
		\[
		[a_1,\ldots,a_{r},a_{r+1},\ldots,a_{n-2}]=-[a_1,\ldots,a_{r+1},a_{r},\ldots,a_{n-2}].
		\]
	\end{lemma}
	\begin{proof}
		For this range of $r$, both $\phi_r$ and $\phi_{r+1}$ are matrices of linear forms. We differentiate the relation $\phi_r \phi_{r+1}=0$. By the Leibniz rule, 
		\[
		0=\frac{\partial^2 (\phi_r \phi_{r+1})}{\partial x_{a_r} \partial x_{a_{r+1}}}=\frac{\partial \phi_r}{\partial x_{a_r}}\frac{\partial \phi_{r+1}}{\partial x_{a_{r+1}}}+\frac{\partial \phi_r}{\partial x_{a_{r+1}}}\frac{\partial \phi_{r+1}}{\partial x_{a_r}},
		\]
		hence the desired relation.
	\end{proof}
	\begin{lemma} \label{reverselemma}
	    \begin{enumerate}
	
	        \item We have
	    
	 $[a_1,a_2,\ldots,a_{n-2}]=\pm[a_{n-2},a_{n-3},\ldots,a_1]$, where the sign is $+1$ if $n \equiv 2,3 \ (\mathrm{mod} \ 4)$ and $-1$ if $n \equiv 0,1 \  (\mathrm{mod} \ 4)$. 
		\item We have $[a_1,a_2,\ldots,a_{n-3}a_{n-2}]=-[a_{n-2},a_2,\ldots,a_{n-3},a_1]$.
		\end{enumerate}
	\end{lemma}
	\begin{proof}
		This follows from the self-duality of the resolution, as explained Lemma \ref{DGC induced duality}, and our sign convention for the dual complex. Part (ii) then follows from (i) and Lemma \ref{swaplemma}. See Lemma 6.3 of \cite{radfis} for more details.
	\end{proof}

	\begin{proof}[Proof of Lemma \ref{single bracket differential}]
		To make notation clearer, we assume $i=1$, $j=2$ and $a_k=k+2$ for each $k$. It suffices to prove the lemma for the transpositions $(k,k+1)$, for $1 \leq k \leq n-1$, as they generate $S_n$. We view $S_n$ as a subgroup of $\mathrm{GL}_n$, embedded via the standard representation. From Proposition \ref{Omega quad change of coordinates}, when $\tau=(k,k+1)$ and $4\leq k \leq n-2$, we obtain
		\[
		[3,\ldots,k, (k+1),\ldots,n]=-[3,\ldots,(k+1),k,\ldots,n],
		\]
		so it is clear the conclusion holds for these transpositions. We thus only have to check the symmetries $(1,2)$, $(2,3)$ and $(n-1,n)$. The symmetry $(12)$ follows from the identity
		\[
		x_1^2 d(x_2/x_1)=-x_2^2 d(x_1/x_2).
		\]
		For the transposition $(n-1,n)$, observe that by differentiating the relation $\phi_{n-3}\phi_{n-2}=0$ with respect to $x_{n-1}$ and $x_n$, we obtain
		\[
		\frac{\partial \phi_1}{\partial x_3}\frac{\partial \phi_2}{\partial x_4}\cdot \cdot \cdot \frac{\partial \phi_{n-4}}{\partial x_{n-2}}(\frac{\partial \phi_{n-3}}{\partial x_{n-1}} \frac{\partial \phi_{n-2}}{\partial x_{n}}+ \frac{\partial \phi_{n-3}}{\partial x_{n}}\frac{\partial \phi_{n-2}}{\partial x_{n-1}}),
		\]
		\[
		=-\frac{\partial \phi_1}{\partial x_3}\frac{\partial \phi_2}{\partial x_4}\cdot \cdot \cdot \frac{\partial \phi_{n-4}}{\partial x_{n-2}}\phi_{n-3}\frac{\partial^2\phi_{n-2}}{\partial x_{n-1} \partial x_{n}}.
		\]
		Using the relation $\phi_r \frac{\partial \phi_{r+1}}{\partial x_{p}}=-\frac{\partial \phi_r}{\partial x_{p}}\phi_{r+1}$, this is equal to
		\[
		\pm \phi_1 \frac{\partial \phi_2}{\partial x_3}\cdot \cdot \cdot \frac{\partial \phi_{n-4}}{\partial x_{n-2}}\frac{\partial\phi_{n-3}}{\partial x_{n-3}}\frac{\partial^2\phi_{n-2}}{\partial x_{n-1} \partial x_{n}}.
		\]
		As entries of $\phi_1$ are in $I(C)$, the above expression is also in $I(C)$ and hence vanishes on $C$. The case $\tau=(34)$ follows from this case and the Lemma \ref{reverselemma}. We are left with the case $\tau=(2,3)$. Using Euler's identity, we compute 
		\begin{align*}
		&\sum_{i=1}^{n}  d(x_i/x_1) \frac{\partial \phi_1}{\partial x_i}=\sum_{i=1}^{n}\frac{dx_ix_1-x_idx_1}{x^2_1}\frac{\partial \phi_1}{\partial x_i}
		 =\frac{d\phi_1}{x_1}-\frac{\phi_1 dx_1 }{x_1^2}=0 \ \textrm{mod} \ I(C),
		\end{align*}
		and hence
		\begin{align*}
		0&=\left( d(x_i/x_1)\sum_{i=1}^{n} \frac{\partial \phi_1}{\partial x_i}\right) \cdot	\frac{\partial \phi_2}{\partial x_4}\frac{\partial \phi_3}{\partial x_5}\cdots \frac{\partial \phi_{n-2}}{\partial x_{n}} \ \textrm{mod} \ I(C)\\&
		=\sum_{i=2}^{n} d(x_i/x_1) [i,4,5,\ldots,n	] \ \textrm{mod} \ I(C).
		\end{align*}
		From Lemma \ref{swaplemma} we deduce $[i,4,5,\ldots,n]=0$ for $4<i<n$. From Lemma \ref{reverselemma}(ii), we see that $[n,4,5,\ldots,n]=0$. Furthermore, the term with $i=4$ is zero mod $I(C)$. This is because we have, reasoning in the same way as in the case $(n-1,n)$,
		\[
		2\frac{\partial \phi_1}{\partial x_4}\frac{\partial \phi_2}{\partial x_4}\cdot \cdot \cdot \frac{\partial \phi_{n-2}}{\partial x_{n}}=\pm \frac{\partial^2 \phi_1}{\partial x_4 \partial x_4}\frac{\partial \phi_2}{\partial x_5}\cdot\cdot\cdot\frac{\partial \phi_{n-3}}{\partial x_n}\phi_{n-2},
		\]
		and the entries of $\phi_{n-2}$ are elements of $I(C)$, as the resolution is self-dual, and this holds for the entries of $\phi_1$. We are left with the identity
		\[
		\frac{\partial \phi_1}{\partial x_2}\frac{\partial \phi_2}{\partial x_4}\cdot \cdot \cdot \frac{\partial \phi_{n-2}}{\partial x_{n}} d(x_2/x_1)+\frac{\partial \phi_1}{\partial x_3}\frac{\partial \phi_2}{\partial x_4}\cdot \cdot \cdot \frac{\partial \phi_{n-2}}{\partial x_{n}} d(x_3/x_1)=0,
		\]
		concluding the proof.
	\end{proof}
	\begin{lemma} \label{single bracket double bracket rel }
		With notation as in  Lemma \ref{single bracket differential}, we have the equality of differential forms on $C$,
		\[
		\frac{x_j^2 d(x_i/x_j)}{[a_1,\ldots,a_{n-2}]}=(n-2)\frac{x_j^2 d(x_i/x_j)}{[[a_1,\ldots,a_{n-2}]]}.
		\]
	\end{lemma}
\begin{proof}
	This follows from the definition of the symbol $[[\ldots]]$
		\[
	[[a_1,a_2,\ldots,a_{n-2}]]=\sum_{k=1}^{n-2} [a_{\sigma^{2k}(1)},a_{\sigma^{2k}(2)},\ldots,a_{\sigma^{2k}(n-2)}].
	\]
	where $\sigma$ is the $(n-2)$-cycle $(1,2,\ldots,n-1)$, and the previous lemma.
\end{proof}
	\begin{lemma} \label{invariant differential change of coords}
		Let $C$ and $C'$ be genus one normal curves in $\mathbb{P}^{n-1}$. Suppose $g \in \mathrm{GL}_{n}$ takes $C'$ to $C$, and let $x_j'=\sum_{i=1}^{n} g_{ij}x_i$. Fix a minimal free resolution $F $ of $I(C)$ and let $F' $  be the resolution of $I(C')$, such that the boundary maps $\phi_r$ and $\phi'_r$ satisfy
		\[
		\phi'_r(x_1,\ldots,x_n)=\phi_r(x'_1,\ldots,x'_n).
		\]
		Then, for $i \neq j$ we have the equality of rational differential forms on $C'$
		\[
		g^* \left( \frac{x_j^2 d(x_i/x_j)}{\Omega_{C,ij}}\right) =\mathrm{det}(g)\frac{x_j^2 d(x_i/x_j)}{\Omega_{C',ij}}.
		\]
	\end{lemma}
	\begin{proof}
		
		We wish to the prove the equality 
		\[
		(x'_j dx'_i-x'_i dx'_j) (\Omega_{C'})_{ij}(x_1,\ldots,x_n)\equiv \mathrm{det}(g)(x_j dx_i-x_i dx_j)(\Omega_{C})_{ij}(x'_1,\ldots,x'_n) 
		\]
		It suffices to check the lemma for a set of generators of $\mathrm{GL}_n$. When $g$ is a diagonal matrix or $g=I+tE_{21}$, the claim follows from Proposition \ref{Omega quad change of coordinates}. When $g$ is a permutation matrix, the claim follows from Lemma \ref{single bracket differential}.
	\end{proof}

\subsection{The discriminant form} \label{discriminant form section}
We now study the connection between genus one normal curves of degree $n$, and sets of $n$ points in general position. For a genus one curve $C \subset \mb{P}^{n-1}$, we associate a discriminant to every hyperplane $H \subset \mb{P}^{n-1}$.  

To construct these discriminants, we need the following result, proven in \cite{radfis}.

\begin{theorem} \label{algebra constructed from n points} Let
  $X \subset \PP^{n-2}$ be a set of $n$ points in general position,
  and let $\Omega_1,\ldots,\Omega_{n-1}$ be the quadratic forms
  associated to a minimal free resolution of $X$. Then there exists a
  commutative and associative $K$-algebra $A$, of dimension $n$, and a
  $K$-basis $1=\alpha_0,\alpha_1,\ldots,\alpha_{n-1}$ for $A$, such
  that for each $1\leq i,j \leq n-1$ we have
  \[
    \alpha_i \alpha_j =c^{0}_{ij}+\sum_{k=1}^{n-1} \frac{\partial^2
      \Omega_k}{\partial x_i \partial x_j}\alpha_k,
  \]
   The constants $c^{0}_{ij}$ are given by 
   \[
   \sum_{r=1}^{n-1}  \left( \frac{\partial^2 \Omega_r}{\partial x_j \partial x_k} \frac{\partial^2 \Omega_k}{\partial x_r \partial x_i}- \frac{\partial^2 \Omega_r}{\partial x_i \partial x_j} \frac{\partial^2 \Omega_k}{\partial x_r \partial x_k}\right)
   \]
 
 for any $1 \leq k \leq n$ with $k \neq i$. The algebra $A$ is isomorphic to
  the affine coordinate ring (i.e., ring of global functions) of $X$,
  and $\alpha_1,\ldots,\alpha_n$ span the trace zero subspace.
\end{theorem}

The key observation for the proof of Theorem \ref{Main theorem I} is that if $K=\Q$ and the resolution model $F_{\bu}$ is defined over $\Z$, then $\Omega_{ij}$ have integral coefficients, and so the multiplication table in the theorem defines an order in the algebra $A$. In particular, the discriminant $D$ associated to the basis $\alpha_0,\ldots,\alpha_{n-1}$ is an integer. Since $\alpha_i$ have trace zero for $i>0$ and $\alpha_0=1$, for the matrix $T$ representing the trace pairing we have $T_{00}=n$, $T_{0i}=0$ and $T_{i0}=0$ for $i>0$.

Let $[C \xrightarrow{} \mb{P}^{n-1}]$ be an $n$-diagram, defined over a field $K$ of characteristic 0. Fix a resolution model $F_{\bullet}$ of $[C \xrightarrow{} \mb{P}^{n-1}]$.  The intersection of $C$ and the hyperplane $x_1=0$ is a set $X$ of $n$ points in $\mb{P}^{n-2}$. A resolution model of $X$ can be obtained by setting $x_1=0$ in the differentials of $F_{\bullet}$, as we will now see.

 Let $A=R/\phi_1(F_1)$ be the homogenous coordinate ring of $C$. We identify the homogeneous coordinate ring of $X$ with $A/(x_1A)$. 
\begin{lemma} \label{hyperplane slice}
	\begin{enumerate}[label=(\roman*)]
\item	The chain complex $F^{0}_{\bu}=F_{\bu} \otimes_{R} (R/x_1R)$ is a minimal free resolution of $A/(x_1 A)$, viewed as a graded $R^{0}$-module, where $R^{0}=R/x_1R \cong K[X_1,\ldots,X_{n-1}]$. Concretely, if $\phi_i: F_i  \xrightarrow{} F_{i-1}$ are the differentials of $F_{\bu}$, viewed as matrices of homogeneous forms in $x_1,\ldots,x_{n-1}$, the differentials $\phi^{0}_i : F^0_i : \xrightarrow{} F^0_{i-1}$ are obtained by setting $\phi^{0}_i(x_1,\ldots,x_{n-1})=\phi_i(0,x_1,\ldots,x_{n-1})$.
\item Let $\Omega_{F_{\bu}}$ be the $\Omega$-matrix associated to the resolution model $F_{\bu}$. The $\Omega$-quadrics associated to the model $F^{0}_{\bu}$ are given by 
\[
\Omega_k(X_1,X_2,\ldots,X_{n-1})=\Omega_{1k}(0,X_1,X_2,\ldots,X_{n-1})
\]
for $1 \leq k \leq n-1$.
\end{enumerate}
\end{lemma}

\begin{proof}
	As $F$ is a free resolution of $A$, in the category of $R$-modules, by definition of the functor $\mathrm{Tor}$, the $i$th homology group of $F^0$ can be identified with $\mathrm{Tor}^{R}_{i}(A,R/x_1 R)$. The short exact sequence
	\[0\xrightarrow{} R \xrightarrow{\cdot x_1} R \xrightarrow{} R/x_1R \xrightarrow{} 0\] induces a long exact sequence
	\[
	\ldots\xrightarrow{}\mathrm{Tor}^{R}_{i}(A,R) \xrightarrow{\cdot x_1} \mathrm{Tor}^{R}_{i}(A,R) \xrightarrow{} \mathrm{Tor}^{R}_{i}(A,R/x_1R) \xrightarrow{}\mathrm{Tor}^{R}_{i-1}(A,R) \xrightarrow{}\ldots
	\]
	Computing $\mathrm{Tor}$ with respect to the second variable, we find that $\mathrm{Tor}^{R}_{i}(A,R)=0$ for $i >0$, and that for $i=0$, $\mathrm{Tor}^{R}_{0}(A,R)=A$. The long exact sequence then implies that $\mathrm{Tor}^{R}_{i}(A,R/x_1R)$ is zero for $i \geq 2$. 
	
	Next, note that $x_1$ is not a zero-divisor on $A$, as $C$ is a genus one normal curve, so in particular irreducible, and so if $x_1$ was a zero divisor then $C$ would be contained in the hyperplane $x_1=0$, which is impossible. Thus $\mathrm{Tor}^{R}_{0}(A,R) \xrightarrow{\cdot x_1} \mathrm{Tor}^{R}_{0}(A,R)$ is injective, with cokernel isomorphic to $A/x_1A$, and hence $\mathrm{Tor}^{R}_{1}(A,R/x_1 R)=0$ and $\mathrm{Tor}^{R}_{0}(A,R/x_1R)=A/x_1A$, concluding the proof of (i). Part (ii) then follows from the definition of differentials $\phi^{0}$.
\end{proof}

We regard the curve $C$ as a subvariety of the projective space $\mb{P}^{n-1}$, with coordinates $x_1,\ldots,x_{n}$, corresponding to a basis of the space of linear forms on $\mb{P}^{n-1}=\mb{P}(V)$. Consider the dual projective space $\mb{P}(V^*)$, with coordinates $u_1,\ldots,u_{n}$. To a point $(u_1:\ldots:u_{n}) \in \mb{P}(V^*)$, we associate a hyperplane $\{u_1x_1+\ldots+u_{n}x_{n}=0\} \subset \mb{P}(V)$.	

We now define the discriminant form associated to a resolution model $F_{\bu}$ of $C$, first as a function $D_{F_\bu}:K^n \xrightarrow{} K$. Given $(u_1,u_2,\ldots,u_{n}) \in K^n\setminus \{0\}$,  choose a $\gamma \in \mathrm{SL}_n(K) $ with $\gamma_{1,i}=u_i$ for $i=1,\ldots,n$. Note that $\gamma$ maps the hyperplane $\{u_1x_1+\ldots+u_{n}x_{n}=0\}$ to the hyperplane $\{x_1=0\}$. Let $G_{\bu}=\gamma^{-1} \cdot F_{\bu}$, and let $G^{0}_{\bu}$ be the resolution obtained by setting $x_1=0$ in $G_{\bu}$, as in Lemma \ref{hyperplane slice}. Then $G^{0}_{\bu}$ is a resolution model of a set of $n$ points in the hyperplane $\{x_1=0\} \cong \mb{P}^{n-2}$. Let $A$ be the $n$-dimensional $K$-algebra associated to $G^{0}_{\bu}$ by Theorem \ref{algebra constructed from n points}, together with the basis  $\alpha_0,\ldots,\alpha_{n-1}$. We define $D_{F_\bu}(u_1,\ldots,u_{n})$ to be the determinant of the matrix representing the trace pairing on $A$ with respect to the basis $\alpha_0,\ldots,\alpha_{n-1}$.

\begin{lemma} \label{dform well-defined}
	The function $D_{F_\bu} :K^{n} \xrightarrow{} K$ is well-defined, i.e. independent of the choice of $\gamma \in \mathrm{SL}_n(K)$. It depends only on the $\Omega$-matrix associated to $F_{\bu}$, and for all $\lambda \in K$ we have $D_{F_\bu}(\lambda u_1,\ldots,\lambda u_{n})=\lambda^{2n}D_{F_\bu}(\lambda u_1,\ldots,\lambda u_{n})$. We say that $D_{F_{\bu}}$ is the discriminant form associated to the model $F_{\bu}$.
\end{lemma}
\begin{proof}
	Suppose we are given $\gamma, \gamma' \in \mathrm{SL}_n(K)$ with $\gamma_{0i}=\gamma'_{0i}=u_i$ for $0 \leq i \leq n-1$. Let $X=\gamma(C) \cap \{x_1=0\}$ and $X'=\gamma'(C) \cap \{x_1=0\}$. Then $\beta=\gamma' \gamma^{-1}$  takes the set $X$ to $X'$, and preserves the hyperplane $\{x_1=0\}$. By construction, we have $\beta_{0,i}=0$ for $i>0$, and $\beta_{0,0}=1$. Let $G_{\bu}=\gamma^{(-1)} \cdot F_{\bu}$ and $G'_{\bu}={\gamma'}^{(-1)} \cdot F_{\bu}$, and let $G^{0}_{\bu}$ and ${G'}^{0}_{\bu}$ be resolution models of $X$ and $X'$ obtained by setting $x_0=0$ in $G_{\bu}$ and $G'_{\bu}$. Then ${G}_{\bu}^{0}=\beta^{0} \cdot {G'}_{\bu}^{0}$, where $\beta^{0} \in \mathrm{SL}_{n-1}(K)$ is obtained from $\beta$ by deleting the leftmost column, and the top row.
	
	Now let $A$ and $A'$ be the algebras associated to $X$ and $X'$, with bases $1=\alpha_0,\alpha_1,\ldots,\alpha_{n-1}$ and $1=\alpha'_0,\alpha'_1,\ldots,\alpha'_{n-1}$.  The matrix $\beta^{0}$ is the change-of-basis matrix relating these bases. Since $\mathrm{det}(\beta^{0})=1$ implies that the determinants of the matrices that represent the trace pairing for these two bases are equal, and  hence $D_{F_{\bu}}$ is well-defined.
	
	We now elaborate on how to compute $D_{F_{\bu}}$ from the matrix $\Omega_{F_{\bu}}$. By definition $D_{F_\bu}(u_1,\ldots,u_{n})$ as the determinant of an $(n-1)\times (n-1)$-matrix $n \cdot c^{0}$, where the entries of $c^{0}$ are quadratic polynomials in the coefficients of the $\Omega$-quadrics associated to the resolution model $G^{0}_{ij}$. Lemma \ref{hyperplane slice}(ii) expresses these coefficients in terms of $\Omega_{G_{\bu}}$, and by Proposition \ref{Omega quad change of coordinates}, $\Omega_{G_{\bu}}=\gamma^{-1} \cdot \Omega_{F_{\bu}}$.
		
	To show that $D_{F_{\bu}}$ is homogeneous, we extend the argument we used to prove that $D_{F_\bu}$ is well-defined. Let $\gamma,\gamma' \in \mathrm{SL}_n(K)$ be the matrices used to define the forms $D_{F_{\bu}}(u_1,\ldots,u_{n})$ and $D_{F_\bu}(\lambda u_1,\ldots,\lambda u_{n})$, let $\beta=\gamma'\gamma^{-1}$, and define $\beta^{0}$ as before. This time, $\mathrm{det}(\beta_0)=\frac{1}{\lambda}$. Now the result follows from Proposition \ref{Omega quad change of coordinates}---note that this time, $\beta^{0}$ is the change of basis relating the bases $\alpha_1,\ldots,\alpha_{n-1}$ and $\alpha'_1,\ldots,\alpha'_{n-1}$ after one of the bases is rescaled by $\lambda$. This rescaling contributes the factor $\lambda^{2n-2}$, and the determinant of $\beta^{0}$ contributes the remaining factor $\lambda^2$.
\end{proof}
In fact, we have a more precise statement.
\begin{lemma} \label{d is poly}
	 $D_{F_{\bu}}$ is a homogeneous polynomial of degree $2n$ in the variables $u_1,u_2,\ldots,u_{n}$, and its coefficients are homogeneous polynomials of degree $2n-2$ in the coefficients of the $\Omega$-matrix of $C$.
\end{lemma}

\begin{proof}
Observe that if $u_1$ is non-zero, we can take $\gamma$ to be the matrix
	\[
	\begin{pmatrix}
		u_1 & 0 & 0&0&\ldots&0 &0\\
		u_2 & 1/u_1 &0&0& \ldots&0 &0\\
		u_3 &0 & 1 &0 &\ldots&0& 0\\
		\vdots&\vdots&\vdots&\vdots&\ddots&\vdots&\vdots \\
		u_{n}&0 &0 &0&\ldots&0 &1
	\end{pmatrix}
	\]
	Thus $D_{F_{\bu}}(u_1,\ldots,u_{n})$ is a rational function in $u_1,\ldots,u_{n}$, with denominator a power of $u_1$. But the choice of the hyperplane $x_1=0$ was arbitrary. An argument  similar to the one we used to show that $D$ is well-defined shows that defining $D_{F_{\bu}}$  by requiring that the $i$-th column of $\gamma$ is $(u_1,\ldots,u_{n})^{T}$ gives the same  function $D_{F_{\bu}} : K^{n} \xrightarrow{} K$. Thus, arguing by symmetry, we conclude that $D$ is a polynomial, and by Lemma \ref{dform well-defined}, it must be a homogeneous polynomial of degree $2n$.  Going through the description of the various group actions involved, we see that its coefficients are homogeneous forms of degree $2n-2$ in the coefficients of the entries of $\Omega_{F_{\bu}}$.
\end{proof}

\begin{lemma} \label{discriminant form change of coord lemma}
	 Suppose we are given two resolution models $F_{\bu}$ and $F'_{\bu}$, of curves $C \subset \mb{P}^{n-1}$ and $C'\subset \mb{P}^{n-1}$ respectively, with $g \cdot F_{\bu}=F'_{\bu} $ for some  $g \in \mathrm{SL}_n(K)$. Let $x'_j=\sum^{n}_{i=1} g_{ij} x_{i}$, and let $u'_{j}=\sum^{n}_{i=1} g^{T}_{ij} u_i$. We then have \[D_{F_{\bu}}(u'_1,\ldots,u'_{n})=D_{F'_{\bu}}(u_1,\ldots,u_{n}).\]
\end{lemma}
\begin{proof}
Follows immediately from Proposition \ref{Omega quad change of coordinates}. Note that the definition of $u_j'$ comes from identifying $u_1,\ldots,u_{n}$ with the dual basis to $x_1,\ldots,x_{n}$, and considering the action of the group $\mathrm{SL}_n$ on the dual of the standard representation. 
\end{proof}

\begin{lemma} \label{Z-order}
	Suppose $K=\mb{Q}$ and the model $F_{\bullet}$ is $\mb{Z}$-integral. Then $D_{F_\bu} \in \mb{Z}[u_1,\ldots,u_{n}]$, and for any $(u_1,\ldots,u_{n}) \in \mb{Z}^{n}$, where the $u_i$ are not all zero, $D_{F_{\bu}}(u_1,\ldots,u_{n})$ is the discriminant of an order in the algebra $A$. 
\end{lemma}

\begin{proof}
We may assume that the integers $u_1,\ldots,u_{n}$ are coprime. Then, in the definition of the discriminant $D_{F_{\bu}} (u_1,\ldots,u_{n})$ we may take $\gamma$ to be an element of $\mr{SL}_n(\mb{Z})$---this follows from the standard fact that $\mathrm{SL}_n(\mb{Z})$ acts transitively on $\mb{P}^{n-1}(\mb{Q})$, see for example Lemma 3.13 of \cite{minred234}. Then $G^{0}_{\bu}$ is an integral resolution model, and so the structure constants that define the algebra $A$ also define an order $B$ in $A$, of discriminant $D_{F_{\bullet}}(u_1,\ldots,u_{n})$.	
\end{proof}

\begin{example}
	Let us take $n=3$, and let $C \subset \mb{P}^2$ be a plane cubic curve. A resolution model of $C$ is a ternary cubic $F(x_0,x_1,x_2)$ defining $F$. Let $\gamma$ be as defined in Lemma \ref{d is poly}. We have
	\[
	\gamma^{-1}=\begin{pmatrix}
		1/u_1 &0 &0 \\
		-u_2 & u_1 & 0\\
		-u_3/u_1 & 0 &1
	\end{pmatrix}
	\]
	 Set $G(x_0,x_1,x_2)=\gamma^{-1} \cdot F=F(x_1/u_1-u_2 x_2-u_3/u_1 x_3,u_1x_2,x_3)$. Next, we set $x_1=0$ in $G$, and obtain $G^{0}(x_2,x_3)=G(0,x_2,x_3)=F(-u_2x_2-u_3/u_1 x_3,x_3)$, a binary cubic form in $x_2$ and $x_3$, with coefficients rational functions of $u_1,u_2$ and $u_3$. Then we see that  $D_{F_{\bu}}(u_1,u_2,u_3)$ is the discriminant of $G^{0}$. Recall, as mentioned in the Chaper \ref{intro}, that the discriminant of a form  \[ax^3+bx^2y+cxy^2+dy^3\] is given by \[
	b^2c^2-4ac^3-4b^3d-27a^2d^2+18abcd.\] 
	By Lemma \ref{d is poly} we see that $D_{F_{\bu}}(u_1,u_2,u_3)$ is a degree 6 polynomial in $u_1,u_2$ and $u_3$---even though the coefficients of the binary form $G^{0}$ are rational functions, with denominators powers of $u_1$. There is no difficulty in using computer algebra software to compute formulas for the coefficients of $D_{F_{\bu}}$. These expressions are fairly involved, and so we omit them here.
\end{example}

\section{The unprojection construction} \label{unprojection chapter}

In this Section we prove the following theorem.
\begin{theorem}[Minimization theorem] \label{global minimization theorem}
Let $C \subset \mb{P}^{n-1}$ be a genus one normal curve of degree $n$, defined over the field $
\mb{Q}$. Suppose that the set $C(\mb{Q}_p)$ is non-empty for every prime $p$. Let $E$ be the Jacobian of $E$, defined by a minimal Weierstrass equation $W$ 
	\[
	y^2 + a_1xy + a_3y = x^3 + a_2x^2 + a_4x + a_6.
	\]
	Then one can choose coordinates on $\mb{P}^{n-1}$, and then a $\mb{Z}$-integral genus one model $F_{\bu}$ of $C$ with the associated invariant differential $\omega$, such that, for every prime $p$, there exists an isomorphism $\gamma : E \xrightarrow{} C$, defined over $\mb{Q}_p$, with 
	\[
	\omega=\gamma^{*}\left( \frac{dx}{2y + a_1x + a_3} \right).
	\]
 We say such a genus one model is a \textit{minimal} model of $C$. Furthermore, the invariants $c_4(F_{\bu})$ and $c_6(F_{\bu})$, as defined in Section 2, are equal to $c_4(W)$ and $c_6(W)$, where $c_4(W)$ and $c_6(W)$ are invariants of the equation $W$, as defined in Chapter III of \cite{silverman2009arithmetic}.
\end{theorem}

The theorem is proved by an induction on $n$. The base case of the induction, $n=3$, is Theorem 1.1 of \cite{minred234}. Our induction step can be viewed as a generalization of the method used to deduce cases $n=4$ and $n=5$ from the case $n=3$ in \cite{minred234} and \cite{minred5}. 

 Our main tool will be the technique of \textit{unprojection}, introduced in \cite{kustin1983constructing}. Unprojection was used to obtain minimal free resolutions of genus one normal curves by Fisher in \cite{fisher2006higher}, and our results can be viewed as refinements of his. Reid and Papadakis have also studied unprojection extensively, and they give a more intrinsic point of view in \cite{papadakis2000kustin}.

The first section consists of an informal discussion motivating and explaining our results. We then prove a technical result in homological algebra that will serve as the induction step in our proofs. We then prove Lemma \ref{first lemma differential}, i.e. that the differential form on a genus one normal curve $C$ constructed from its $\Omega$-matrix is regular.  We then prove, assuming Theorem \ref{formula for Jacobian}, the local version of the minimization theorem and from it deduce the global version.
\begin{remark}
 Note that the proof of the formula for the Jacobian, i.e. Theorem \ref{formula for Jacobian}, which we give in Section \ref{analytic chapter} uses Lemma \ref{first lemma differential}. Thus, to clarify, formally the proof of Theorem \ref{global minimization theorem}  proceeds as follows: in this Section we prove, using unprojection, Lemma \ref{first lemma differential}. In Section \ref{analytic chapter} we deduce Theorem \ref{formula for Jacobian} from Lemma \ref{first lemma differential} and Lemma \ref{analytic set pts}, via a complex analytic calculation. Lemma \ref{analytic set pts} does not depend on the methods of unprojection, and is proved using only the results of Section \ref{resolution model section}. Then using Theorem \ref{formula for Jacobian} and unprojection we deduce Theorem \ref{local minimization theorem}, and from this result we deduce Theorem \ref{global minimization theorem}. 
\end{remark}

\subsection{The mapping cone construction} \label{mapping cone section}
In this section we give an algebraic construction that will provide the induction step for our later arguments. Our work is a refinement of the results of Section 9 of \cite{fisher2006higher}, and our notation and exposition will follow \cite{fisher2006higher} and \cite{kustin1983constructing} closely. 
 
Let $R=S[x_1,\ldots,x_{m-1}]$ be a graded ring. Throughout we consider graded $R$-modules. 

\subsection{Mapping cones.} We recall some basic facts about mapping cones, as defined in Section 1.5 of \cite{weibel1995introduction}.
\begin{definition} \label{mapping cone definition}
	Let $\alpha : A_{\bullet} \xrightarrow[]{} B_{\bullet}$ be a map of chain complexes $(A_{\bullet},d_A)$ and $(B_{\bullet},d_B)$. The mapping cone $M(\alpha)$ of $\alpha$ is the chain complex with $M(\alpha)_n=B_{n}\oplus A_{n-1}$ and the differential $d^n$ represented by the matrix
	\[
	d^n=
	\begin{pmatrix}
	d^{n}_B & -\alpha_{n-1}\\
	0 & -d^{n-1}_{A} \\
	\end{pmatrix},
	\]
	i.e. $d^n(a_{n-1},b_{n})=(-d^{n-1}_A(a_{n-1}),-\alpha_{n-1}(a_{n-1})+d^{n}_B(b_{n}))$.
\end{definition}	

\begin{lemma} \label{mapping cone homotopy lemma}
Suppose that $\alpha,\beta : A_{\bullet} \xrightarrow{} B_{\bullet}$ are chain homotopic morphisms of chain complexes. The mapping cones $M(\alpha)$ and $M(\beta)$ are isomorphic, with an isomorphism determined by a choice of homotopy between $\alpha$ and $\beta$.
\end{lemma}

\begin{proof}
 Let $h_i : A_i : \xrightarrow{} B_{i+1}$ be a chain homotopy, with  $\beta_n-\alpha_n=h_{n-1}d^{n}_A + d^{n+1}_{B}h_n$. We define maps $u_n : M(\alpha)_n \xrightarrow{} M(\beta)_n$ by
 \[
  u_n=\begin{pmatrix} id_{B_n} & h_{n-1}\\ 0 & id_{A_{n-1}}
 	\end{pmatrix}
 \]
 We verify that the $u_n$ define a chain map:
 \[
 d^{n}_{M(\beta)}u_n=\begin{pmatrix}
 	d^{n}_B & -\beta_{n-1}\\
 	0 & -d^{n-1}_{A} \\
 \end{pmatrix}
\begin{pmatrix} id_{B_n} & h_{n-1}\\ 0 & id_{A_{n-1}}
\end{pmatrix}=\begin{pmatrix}
d^{n}_B & d^{n}_{B}h_{n-1}-\beta_{n-1}\\
0 & -d^{n-1}_{A} \\
\end{pmatrix},
 \]
which is equal to
 \[
 u_{n-1}d^{n}_{M(\alpha)}=\begin{pmatrix} id_{B_{n-1}} & h_{n-2}\\ 0 & id_{A_{n-2}}
 \end{pmatrix}\begin{pmatrix}
 	d^{n}_B & -\alpha_{n-1}\\
 	0 & -d^{n-1}_{A} \\
 \end{pmatrix}=\begin{pmatrix}
 d^{n}_B & -\alpha_{n-1}-h_{n-2}d^{n-1}_{A}\\
 0 & -d^{n-1}_{A} \\
\end{pmatrix}.
 \]
An inverse is provided by the maps $v_n : M(\beta)_n \xrightarrow{} M(\alpha)_n$
 \[
v_n=\begin{pmatrix} id_{B_n} & -h_{n-1}\\ 0 & id_{A_{n-1}}
\end{pmatrix}.
\]
\end{proof}

Now suppose we are given the following data. 
\begin{definition}[Unprojection data] \label{unprojection data}
	\leavevmode
\begin{enumerate}[label=(\roman*)]
	\item  $(F_{\bullet},\phi)$ and $(G_{\bullet},\psi)$ be chain complexes of graded $R$-modules, of length $n-2$ and $n-1$ respectively, with $F_0=G_0=R$.
	\item Chain maps $\alpha : G_{\bullet}(-1) \xrightarrow{} F_{\bullet}[-1] $ and $\beta : F_{\bullet}(-1)  \xrightarrow{} G_{\bullet}(-1)$, with $\beta_0 : F_0 \xrightarrow{} G_0$ equal to the identity map.
	\item A null homotopy $\gamma$ for the map $\alpha \circ \beta$. In other words, maps  $\gamma_{i+1}: F_{i} \xrightarrow{} F_{i}$ that satisfy the identity
	\[
		\alpha_i\beta_{i}=\phi_i \gamma_{i+1}+\gamma_{i} \phi_{i},
	\]
	for all $i$. We furthermore require that $\gamma_1=0$.
\end{enumerate}
\end{definition}
We summarise this information in the following commutative diagram.
\begin{equation} \label{big diagram}
\begin{tikzcd}
 0 \ar[r] \ar[d] & F_{n-2}(-1)\ar[r, "\phi_{n-2}"] \ar[ddl,"\gamma_{n-2}", pos=0.7] \ar[d,"\beta_{n-2}"] & \ldots \ar[r, "\phi_{3}"] & F_{2}(-1) \ar[r, "\phi_{2}"] \ar[d,"\beta_{2}"] & F_{1}(-1) \ar[r, "\phi_{1}"] \ar[d,"\beta_1"] \ar[ddl, "\gamma_2", pos=0.7] & F_{0}(-1) \ar[d,"\beta_0"] \ar[ddl, "\gamma_1", pos=0.7]   \\
 G_{n-1}(-1)\ar[r, "\psi_{n-1}"]  \ar[d,"\alpha_{n-1}"] & G_{n-2}(-1)\ar[r, "\psi_{n-2}"]  \ar[d,"\alpha_{n-2}"] &\ar[r, "\psi_{3}"]  \ldots& G_{2}(-1) \ar[r, "\psi_{2}"] \ar[d,"\alpha_{2}"] & G_{1}(-1) \ar[r, "\psi_{1}"]  \ar[d,"\alpha_1"] & G_{0}(-1) \ar[d] \\
 F_{n-2} \ar[r,"\phi_{n-2}"] & F_{n-3} \ar[r,"\phi_{n-3}"] &\ldots\ar[r,"\phi_2"] & F_1 \ar[r,"\phi_{1}"] & F_0 \ar[r] & 0
\end{tikzcd} \tag{$\ast$}
\end{equation}

\begin{lemma} \label{homotopy chain map} \begin{enumerate}[label=(\roman*)]
		\item 
	
	The maps $h_i : M(\beta)_i:=G_{i}(-1)\oplus F_{i-1} \xrightarrow{} F_{i-1}$ given by $h(g,f)=-\alpha_i(g)+\gamma_{i}(f)$ define a map of chain complexes $h: M(\beta)[1] \xrightarrow[]{} F_{\bullet}$.
	
	\item If we replace the null homotopy $\gamma$ by a different null homotopy $\gamma'$, the resulting map $h': M(\beta)[1] \xrightarrow[]{} F_{\bullet}$ is chain homotopic to $h$. The same conclusion holds if we replace $\alpha$ by a chain homotopic map $\alpha'$.
\end{enumerate}
\end{lemma}
\begin{proof}
\begin{enumerate}[label=(\roman*)]
\item We need to check that $h_{i-1} d_{i}^M=\phi_i h_i$ for all $i$. We have
	\begin{equation*}
	h_{i-1} d_{i}^M=
	\begin{pmatrix}
	-\alpha_{i-1} & \gamma_{i-1}
	\end{pmatrix}
	\begin{pmatrix}
	\psi_i &  -\beta_{i-1}\\
	0 &  -\phi_{i-1} \\
	\end{pmatrix}=
	\begin{pmatrix}
	-\alpha_{i-1}\psi_i &  \alpha_{i-1}\beta_{i-1}-\gamma_{i-1}\phi_{i-1}
	\end{pmatrix},
	\end{equation*}
	and
	
	\begin{equation*}
	\phi_{i} h_i=
	\phi_{i}
	\begin{pmatrix}
	-\alpha_{i} &
	\gamma_{i}
	\end{pmatrix}
	=
	\begin{pmatrix}
	-\phi_{i}\alpha_i & \phi_{i}\gamma_{i}\\
	\end{pmatrix}.
	\end{equation*}
	As $\alpha$ is a map of chain complexes, we have $-\alpha_{i-1}\psi_i= -\phi_{i+1}\alpha_i$, and by assumption we have $\alpha_{i-1}\beta_{i}-\gamma_{i-1}\phi_{i-1}= \phi_{i+1}\gamma_{i}$.
	\item Let $\gamma''=\gamma-\gamma'$, then $(h_i-h'_i)(f,g)=\gamma''_{i-1}(f)$,  and $\phi_i \gamma''_{i+1}+\gamma''_{i} \phi_{i}=0$ for all $i$. From this, we deduce that the maps $(-1)^i\gamma''_i$ define a chain map $F_{\bu}(-1) \xrightarrow{} F_{\bu}$. As $\gamma''_1=0$, this chain map is null homotopic, and hence there exist maps $\rho_{i} : F_i(-1) \xrightarrow{} F_{i+1}$ with $\gamma''_i=-\rho_{i-1}\phi_{i}+\phi_{i+1}\rho_{i}$. Then $\rho$ defines a null homotopy for $h-h'$, as seen from the following identity
\[
\begin{pmatrix}
	0 & \gamma''_{i-1}
\end{pmatrix}= 	\begin{pmatrix}
0 & \rho_{i-2}
\end{pmatrix}
\begin{pmatrix}
\psi_i &  -\beta_{i-1}\\
0 &  -\phi_{i-1} \\
\end{pmatrix}+ \phi_{i}\begin{pmatrix}
0 & \rho_{i-1}
\end{pmatrix}.
\]
Now suppose $\alpha$ and $\alpha'$ are chain homotopic. Put $\alpha''=\alpha-\alpha'$, and let $h$ be a null homotopy for $\alpha$. Then $h$ defines a null homotopy for $h-h'$:
\[
\begin{pmatrix}
	\alpha''_{i-1} & 0
\end{pmatrix}= 	\begin{pmatrix}
	h_{i} & 0
\end{pmatrix}
\begin{pmatrix}
	\psi_i &  -\beta_{i-1}\\
	0 &  -\phi_{i-1} \\
\end{pmatrix}+ \phi_{i}\begin{pmatrix}
	h_{i-1} & 0
\end{pmatrix}.
\]
\end{enumerate}
\end{proof}

\begin{lemma} \label{independence of choice}
	For any triple of maps $\alpha,\beta,\gamma$ as above, we define $\mc{H}(\alpha,\beta,\gamma)$ to be the mapping cone of the map $h : M(\beta)[1] \xrightarrow{} F$. 
	
	The complex $\mc{H}(\alpha,\beta,\gamma)$ is independent (up to an isomorphism of chain complexes) of the chain homotopy classes of $\alpha$ and $\beta$, and independent of the choice of the homotopy $\gamma$. When there is no danger of confusion, we denote $\mc{H}(\alpha,\beta,\gamma)$ by $\mc{H}$.
\end{lemma}
\begin{proof}
This follows from Lemma \ref{mapping cone homotopy lemma} and Lemma \ref{homotopy chain map}.
\end{proof}

\begin{lemma} \label{exactness lemma}
	Suppose that we have $H_i(F_{\bullet})=0$ for $i \geq 1$ and $H_i(G_{\bullet})=0$ for $i \geq 2$. Then the homology groups $H_i(\mc{H})$ vanish for $i \geq 2$.
\end{lemma}

\begin{proof}
	We have the mapping cone long exact sequences:
	\begin{equation*}
	\ldots\xrightarrow[]{}H_{i+1}(M(\beta)) \xrightarrow[]{}H_i (\mathcal{F}) \xrightarrow{\beta_*} H_i(\mathcal{G}) \xrightarrow{} H_i (M(\beta)) \xrightarrow[]{}\ldots
	\end{equation*}
	and 
	\begin{equation*}
	\ldots\xrightarrow[]{}H_{i+1}(\mathcal{H}) \xrightarrow[]{}H_{i+1} (M(\beta)) \xrightarrow{h_{*}} H_{i}(\mathcal{F}) \xrightarrow{} H_i (\mathcal{H}) \xrightarrow[]{}\ldots
	\end{equation*}
	As $H_i(F)=H_i(G)=0$ for $i \geq 1$, from the first sequence we deduce that $H_i(M(\beta))=0$ for $i \geq 2$, and then from the second one that $H_i(\mc{H})=0$ for $i\ge 2$.
\end{proof}

Let us now adjoin an indeterminate $x_m$ to $R$, and put $\mc{R}=R[x_m]=S[x_1,\ldots,x_m]$. We extend scalars, and work in the category of graded $\mc{R}$-modules, putting $\mathcal{F}=F \otimes \mc{R}$, $\mathcal{G}=G \otimes \mc{R}$.
Let 
\[\delta_i=\gamma_i+(-1)^i x_{m}: \mathcal{F}_{i-1}(-1) \xrightarrow{} \mathcal{F}_i
\]
It is simple to verify that this map satisfies
\begin{equation*}
\alpha_i\beta_{i} = \phi_i\delta_{i+1} + \delta_i\phi_i.
\end{equation*}
Thus $\alpha,\beta$ and $\delta$ also satisfy the conditions of Definition \ref{unprojection data}, and we may form the chain complex of graded $\mc{R}$-modules $\mc{H}$. Unwrapping the definitions, we have the following explicit description of $\mc{H}$:
\begin{align*}
&\mc{G}_{n-1}(-1)\oplus\mc{F}_{n-2}(-1) \xrightarrow{d_{n-1}}\ldots \\& \ldots\xrightarrow{d_{i+1}} \mathcal{F}_i \oplus \mathcal{G}_{i}(-1) \oplus \mathcal{F}_{i-1}(-1) \xrightarrow{d_i} \mathcal{F}_{i-1} \oplus \mathcal{G}_{i-1}(-1) \oplus \mathcal{F}_{i-2}(-1) \xrightarrow{d_{i-1}}\ldots\\&
\ldots\xrightarrow{d_3} \mathcal{F}_2 \oplus \mathcal{G}_{2}(-1) \oplus \mathcal{F}_{1}(-1) \xrightarrow{d_2} \mathcal{F}_1 \oplus \mathcal{G}_{1}(-1) \oplus \mc{F}_0(-1) \xrightarrow{d_1} \mathcal{F}_0 \oplus \mathcal{G}_0(-1) 
\end{align*}
with the differential given by the matrix
\begin{equation*} \label{differentialH}
d_i=
\begin{pmatrix}
\phi_i &  \alpha_i & -\delta_i\\
0 &  -\psi_i & \beta_{i-1}\\
0 &  0 & \phi_{i-1}\\
\end{pmatrix} \tag{$\dagger$}
\end{equation*}
where we regard the maps and modules labelled with negative indices as zero objects. When there is no danger of confusion, we refer to $\mc{H}(\alpha,\beta,\gamma)$ as $\mc{H}(\alpha,\beta)$ or simply as $\mc{H}$. One can visualise this chain complex by taking direct sums of the modules on each (right-to-left) diagonal of the big diagram (\ref{big diagram}) as modules in the complex, with the maps as specified by the diagram.

Let $I=\phi_1(F_1) \subset F_0$ and $J=\psi_1(G_1) \subset G_0$. As $F_0=G_0=R$, we can identify $I$ and $J$ with homogeneous ideals of $R$, and complexes $F_{\bullet}$ and $G_{\bullet}$ with their graded free resolutions. Fix a basis $a_1,\ldots,a_t$ of $F_1$ and a basis $b_1,\ldots,b_s$ of $G_1$. Then $I$ and $J$ are generated by the elements $\phi_1(a_i)=f_i$, $i=1,\ldots,t$ and $\psi_1(b_i)=g_i$, $i=1,\ldots,s$, respectively. Furthemore, let $h_i=\alpha_1(b_i) \in R$. Now define $\mc{I} \subset \mc{R}$ to be the ideal generated by $f_1,\ldots,f_t,x_mg_1+h_1,\ldots,x_mg_s+h_s$. 

Note that $\mc{I}$ is a homogeneous ideal, since the degrees of the elements $g_i$ are one less than the degrees of $h_i$, by the grading of the maps $\alpha$ and $\beta$.
\begin{proposition} \label{exactness lemma 2}
	With assumptions as in Lemma \ref{exactness lemma}, we have $H_1(\mc{H})=0$, and $H_0(\mc{H})=\mc{R}/\mc{I}$. Hence $\mc{H}$ is a graded free resolution of $\mc{R}/\mc{I}$. 
\end{proposition}

\begin{proof}
	As stated, this is Lemma 9.3 of \cite{fisher2006higher}. See also Theorem 1.3 of \cite{kustin1983constructing}.
\end{proof}

\begin{proposition} \label{quadrics and unprojection}
	Let $1\leq a_1,a_2,\ldots,a_{n-2} \leq m$ be integers, and suppose that $a_k=m$ for $2\leq k \leq n-3$. Then 
	\[
	[a_1,a_2,\ldots,a_{n-2}]_{\mathcal{H}_{\alpha,\beta,\delta}}=(-1)^{k+1}[a_1,a_2,\ldots,a_{k-1},a_{k+1},\ldots,a_{n-2}]_{F}.
	\]
	where the square bracket symbols are as defined in Section \ref{Omega quad section}.
\end{proposition}
\begin{proof}
	Consider the formula (\ref{differentialH}) describing the differential of the complex $\mc{H}(\alpha,\beta,\delta)$. Notice that, as $\delta$ is the only map whose definition directly involves the monomial $x_m$, we have:
	
	\begin{equation*}
	\frac{\partial d_k}{\partial x_{m}}=
	\begin{pmatrix}
	0 & 0 & (-1)^{k+1}\\
	0 &  0 & 0\\
	0 &  0 & 0\\
	\end{pmatrix}.
	\end{equation*}
	But for any upper-triangular matrix, we have
	\[
	\begin{pmatrix}
	a & b & c\\
	0 &  d & e\\
	0 &  0 & f\\
	\end{pmatrix}
	\begin{pmatrix}
	0 & 0 & 1\\
	0 &  0 & 0\\
	0 &  0 & 0\\
	\end{pmatrix}
	=
	\begin{pmatrix}
	0 & 0 & a\\
	0 &  0 & 0\\
	0 &  0 & 0\\
	\end{pmatrix},
	\]
	and
	\[
	\begin{pmatrix}
	0 & 0 & 1\\
	0 &  0 & 0\\
	0 &  0 & 0\\
	\end{pmatrix}
	\begin{pmatrix}
	a & b & c\\
	0 &  d & e\\
	0 &  0 & f\\
	\end{pmatrix}
	=
	\begin{pmatrix}
	0 & 0 & f\\
	0 &  0 & 0\\
	0 &  0 & 0\\
	\end{pmatrix}.
	\]
	Thus $[a_1,a_2,\ldots,a_{n-2}]=\frac{\partial d_{1}}{\partial x_{a_1}} \frac{\partial d_{2}}{\partial x_{a_2}} \cdot \cdot \cdot \frac{\partial d_{n-2}}{\partial x_{a_{n-2}}}$ will be the product of top-leftmost entries of $ \frac{\partial d_{i}}{\partial x_{a_i}}$ for $i<k$, multiplied by the product of the bottom-rightmost entries of $ \frac{\partial d_{i}}{\partial x_{a_i}}$ for $i>k$, with the sign $(-1)^{k+1}$ in front. Examining what these entries are, we find that 
	
	\[[a_1,a_2,\ldots,a_{n-2}]=(-1)^{k+1}[a_1,a_2,\ldots,\hat{a_k},\ldots,a_{n-2}].
	\]
\end{proof}
\subsection{Self-duality.} Now suppose that $F_{\bullet}$ and $G_{\bullet}$ are self-dual. Fix isomorphisms $\eta : {F_{\bullet}^*} \xrightarrow[]{} F_{\bullet}$ and $\theta : G_{\bullet} \xrightarrow[]{} G^*_{\bullet}$, and define maps $\beta^T : {G_{\bullet}}(-1) \xrightarrow{} {F_{\bullet}}[-1]$ and $ \delta^{T} : {F_{\bullet}} \xrightarrow{} {F_{\bullet}}[-1]$ by setting
\[ \beta^T= \eta \circ \beta^*\circ \theta, \]
and 
\[ \delta^{T}=\eta \circ \delta^*\circ \theta. \]
\begin{lemma}
Suppose that we have $\alpha=\beta^T$. Then there exists a null homotopy $\delta$ for the map $\alpha\circ\beta$, with $\delta=\delta^T$.
\end{lemma}
\begin{proof}
	We identify  $(F_{\bu},\phi) \cong (F_{\bu}^*,\phi^{*})$ via $\eta$ and  $(G_{\bu},\psi) \cong (G_{\bu}^*,\psi^{*})$ via $\theta$. We have $\phi_{n-2-i}=\phi_i^*$, and $\beta^*_i=\beta^{T}_{n-2-i}$. Let $\delta$ be any null homotopy for $\beta^{T}\circ\beta$. Dualizing the identity
	\[
	\beta^{T}_i\beta_{i}=\phi_{i}\delta_{i+1}+\delta_{i}\phi_i,
	\]
	we find that
	\[
	\beta^{T}_{n-2-i}\beta_{n-2-i}=\phi_{n-2-i}\delta^{*}_{i}+\delta^*_{i+1}\phi_{n-2-i},
	\]
	For $2i<n-2$, we replace $\delta_{i}$ with $\delta^{*}_{n-2-i}$. The resulting $\delta$ is still a null-homotopy for $\beta^{T}\circ\beta$ and satisfies $\delta=\delta^{T}$. 
\end{proof}

\begin{lemma} \label{H is self-dual}
 Assume that we have $\alpha=\beta^{T}$ and $\delta=\delta^{T}$. Then the chain complex $\mc{H}(\alpha,\beta,\delta)$ is self-dual. 
\end{lemma}

\begin{proof}
	The dual complex $\mc{H}(\alpha,\beta,\delta)^*$ is given by
	
	\begin{align*}
	&\mc{G}^*_{n-1}(-1)\oplus\mc{F}^*_{n-2}(-1) \xleftarrow{d^*_{n-1}}\ldots\xleftarrow{d^*_{i+1}} \mathcal{F}^*_i \oplus \mathcal{G}^*_{i}(-1) \oplus \mathcal{F}^*_{i-1}(-1) \xleftarrow{d^*_i} \ldots
	\\&
	\ldots\xleftarrow{d^*_3} \mathcal{F}^*_2 \oplus \mathcal{G}^*_{2}(-1) \oplus \mathcal{F}^*_{1}(-1) \xleftarrow{d^*_2} \mathcal{F}^*_1 \oplus \mathcal{G}^*_{1}(-1) \xleftarrow{d^*_1} \mathcal{F}^*_0 \oplus \mathcal{G}^*_0(-1) 
	\end{align*}
	with $d_i^*$ given by
	\begin{equation*}
	d_i^*=
	\begin{pmatrix}
	\phi^*_i &  0 & 0\\
	\alpha_i^* &  -\psi^*_i & 0\\
	-\delta_i^* &  \beta_i^*& \phi^*_{i-1}\\
	\end{pmatrix}
	\end{equation*}
	Switching the places of the two $\mc{F}^*$ terms, the matrix representing $d_i^*$ will be
	\begin{equation*}
	\begin{pmatrix}
	\phi_{i-1}^* &  \beta_i^* & -\delta_i^*\\
	0 &  -\psi_i^* & \alpha_i^*\\
	0 &  0 & \phi^*_{i}\\
	\end{pmatrix}
	\end{equation*}
	Now identify $F_{\bullet}$ and $F_{\bullet}^*$ with $\eta$, and identify $G_{\bullet}$ and $G_{\bullet}^*$ with $\theta$. The matrix representing $d^{*}_i$ is the same as the matrix representing $d_{n-1-i}$. In other words, we have proven that the map

	\[
	\mc{H} \xrightarrow{} \mc{H}^*: (f_1,g,f_2) \mapsto (\eta(f_2),\theta^{-1}(g),\eta(f_1))
	\]
	is an isomorphism of chain complexes.
\end{proof}

\subsection{Construction of a minimal free resolution.}  The free resolution $\mc{H}(\alpha,\beta,\delta)$ is not minimal. When $\alpha_{n-2}$ is an isomorphism, we can eliminate from $\mc{H}$ to produce a chain complex $\mc{H}_{min}$  that is a minimal resolution. Set
 \begin{align*}\mc{H}_{min}(\alpha,\beta,\delta)_{n-1}&=\mc{F}_{n-2}(-1), \ \ \ \ \ \mc{H}_{min}(\alpha,\beta,\delta)_{n-2}=\mc{F}_{n-3}(-1) \oplus \mc{G}_{n-2}(-1),\\
 \mc{H}_{min}(\alpha,\beta,\delta)_{1}&=\mc{F}_{1} \oplus \mc{G}_1(-1), \ \ \ \ \	 \mc{H}_{min}(\alpha,\beta,\delta)_{0}=\mc{F}_{0}. 
\end{align*}
When $n>4$, for the differentials, we set 
\[
d'_{n-1}=\begin{pmatrix}
\beta_{n-2}-\psi_{n-1}\alpha_{n-2}^{-1} \delta_{n-1}\\
\phi_{n-2}
\end{pmatrix}, \
d'_{n-2}=\begin{pmatrix}
\alpha_{n-2} &-\delta_{n-2}\\
-\psi_{n-2} &\beta_{n-3}\\
0 & \phi_{n-3}
\end{pmatrix},
\]
\[
d'_2=\begin{pmatrix}
\phi_2&\alpha_2 &	-\delta_2\\
0&-\psi_2 &\beta_1
\end{pmatrix}, \ 
d'_1=\begin{pmatrix}
\phi_1 & \alpha_1-\delta_1\beta^{-1}_1\psi_1
\end{pmatrix}
\]	
In all other cases, we take $\mc{H}_{min}(\alpha,\beta,\delta)_i=\mc{H}(\alpha,\beta,\delta)_i$ and $d'_i=d_i$.

 When $n=4$, we need to slightly modify the formulae for $d'_2=d'_{n-2} :  \mc{G}_{2}(-1)\oplus \mc{F}_{1}(-1)\xrightarrow{} \mc{F}_{1} \oplus \mc{G}_1(-1)$, and instead take
 \[
d'_2=\begin{pmatrix}
	\alpha_2 &	-\delta_2\\
	-\psi_2 &\beta_1
\end{pmatrix}
 \]

The complex $\mc{H}_{min}$ is also a self-dual free resolution of the ideal $I_n$, and the final term of the complex is $\mc{R}(-n)$. When $S=k$ is a field, then $\mc{H}_{min}$ is a minimal free resolution of $I_n$. Furthermore, this modification does not affect the forms $[a_1,\ldots,a_{n-2}]$. In other words, we have

\[
[a_1,a_2,\ldots,a_{n-2}]_{\mathcal{H}_{\alpha,\beta,\delta}}=[a_1,a_2,\ldots,a_{n-2}]_{\mathcal{H'}_{\alpha,\beta,\delta}}.
\]
\subsection{Construction of the unprojection data} \label{construction of the unprojection data}
In this section we give a method to construct maps $\beta: F_{\bullet}(-1) \xrightarrow{} G_{\bullet}(-1)$ and $\alpha : G_{\bullet}(-1) \xrightarrow{} F_{\bullet}[-1]$ that satisfy the conditions of Definition \ref{unprojection data}. Recall that in Section \ref{DGCA section} we have put a differential graded algebra structure on both $F_{\bullet}$ and $G_{\bullet}$. 

\begin{proposition} \label{unprojection data construction}
Let $I$ and $J$ be homogeneous ideals of $R=S[x_1,\ldots,x_{m-1}]$, with $I \subset J$. Assume that $I$ is generated by forms of degree $r+1$, while $J$ is generated by forms of degree $r$, for some positive integer $r$.

Let $F_{\bullet}$ and $G_{\bullet}$ be graded free resolutions of $I$ and $J$, of length $n-2$ and $n-1$ respectively. Assume that the maps $F_{\bu} \xrightarrow{} F_{\bu}^{*}$ and $G_{\bu} \xrightarrow{} G_{\bu}^*$ induced by choices of a DGCA structure on $F_{\bu}$ and $G_{\bu}$ are isomorphisms. 
Then there are maps $\alpha,\beta$ and $\delta$ that satisfy the conditions of Definition \ref{unprojection data}, and the chain complex $\mc{H}_{min}$ is a self-dual free resolution. 
\end{proposition}

\subsection{The map $\beta$.} The inclusion $I \subset J$ induces the comparison map of resolutions $\beta: F_{\bullet}(-1) \xrightarrow{} G_{\bullet}(-1)$.

\begin{equation*}
	\begin{tikzcd}
		0 \ar[r]& F_{n-2}(-1)\ar[r, "\phi_{n-2}"]  \ar[d,"\beta_{n-2}"] & \ldots\ar[r, "\phi_{3}"] & F_{2}(-1) \ar[r, "\phi_{2}"] \ar[d,"\beta_{2}"] & F_{1}(-1) \ar[r, "\phi_{1}"] \ar[d,"\beta_1"] & F_{0}(-1) \ar[d,"\beta_0"] \ar[r] &0  \\
		G_{n-1} \ar[r,"\psi_{n-1}"] &G_{n-2} \ar[r,"\psi_{n-2}"]&\ldots \ar[r,"\psi_3"] & G_2 \ar[r,"\psi_{2}"] & G_1 \ar[r,"\psi_1"] & G_0 \ar[r] &0
	\end{tikzcd}
\end{equation*} 
\subsection{The map $\alpha$ as the dual of $\beta$.}  By assumption, the induced pairings $F_{i} \otimes F_{n-2-i} \xrightarrow{} R$ and $G_{i} \otimes G_{n-1-i} \xrightarrow{} R$ are perfect. Following the proof of Theorem 1.4 of \cite{kustin1983constructing}, we define the map $\alpha$. Define $\alpha_i : G_i(-1) \xrightarrow{} F_{i-1}$ by taking $\alpha(z_i)$ to be the unique element of $F_{i-1}$ such that
 \[\alpha_i(z_i) \cdot z_{n-1-i}=(-1)^{i+1} z_i \cdot \beta_{n-1-i}(z_{n-1-i})\] 
The map $\alpha$ is the same as the map $\beta^{T}$ defined in the previous section, with isomorphisms $F_{\bullet} \cong F^*_{\bullet} $ and $G_{\bullet} \cong G^*_{\bullet}$ induced by the DGCA structures, as explained in Section \ref{DGCA section}. Note that the map $\alpha_{n-2}$ corresponds to multiplication by the unit $(-1)^{n}$. Thus $\alpha$ is a homomorphism of chain complexes, represented by the following diagram.  

\begin{equation*} \label{a diagram}
	\begin{tikzcd}
		0 \ar[r]& G_{n-1}(-1)\ar[r, "\psi_{n-1}"]  \ar[d,"\alpha_{n-1}"] &\ar[r, "\psi_{3}"] \ldots& G_{2}(-1) \ar[r, "\psi_{2}"] \ar[d,"\alpha_{2}"] & G_{1}(-1) \ar[r, "\psi_{1}"]  \ar[d,"\alpha_1"] & G_{0}(-1) \ar[d] \\
		0 \ar[r] & F_{n-2} \ar[r,"\phi_{n-2}"] &\ldots\ar[r,"\phi_2"] & F_1 \ar[r,"\phi_{1}"] & F_0 \ar[r] & 0
	\end{tikzcd}
\end{equation*}
The following lemma shows that $\alpha$ and $\beta$ correspond to unprojection data.
\begin{lemma}
The map $\alpha \circ \beta : F_{\bullet} (-1) \xrightarrow{} F_{\bullet}[-1]$ is null homotopic. 
\end{lemma}
\begin{proof}
Let $\xi : F_{i}(-1) \otimes F_{j}(-1) \xrightarrow{} F_{i+j}$ be the chain homotopy defined in Lemma \ref{null homotopy lemma}. We define a null homotopy $h : F_{i}(-1) \xrightarrow{} F_i$ by taking $h_i(z_i)$ to be the unique element such that $h_i(z_i) \cdot z_{n-2-i}=(-1)^{i+1} \xi (z_i \otimes z_{n-2-i})$ for all $z_{n-2-i}\in F_{n-2-i}(-1)$. By (iii) of Lemma \ref{null homotopy lemma}, $h_0=h_{n-2}=0$. By part (iv), we compute
\begin{align*}
(h_{i-1}\phi_i+\phi_i h_i)(z_i) \cdot  z_{n-1-i}&=h_{i-1}(\phi_i(z_i))\cdot z_{n-2-i}+\phi_i(h_i(z_i)) \cdot z_{n-1-i} \\
&=(-1)^{i} \xi (\phi_i(z_i) \otimes z_{n-1-i})+(-1)^{i+1} h_i(z_i) \cdot (\phi_{n-1-i}(z_{n-1-i})) \\
&=(-1)^{i} \xi (\phi_i(z_i) \otimes z_{n-1-i})+(-1)^{i} \xi(z_i \otimes\phi_{n-1-i}(z_{n-1-i})) \\&
=(-1)^{i+1} \beta_i(z_i) \cdot \beta_{n-1-i}(z_{n-1-i}) \\&
=\alpha_i(\beta_i(z_i)) \cdot z_{n-1-i}.
\end{align*}  
Note that in the second and the third equality we used the Leibniz rule. Hence $\alpha_i \beta_i=h_{i-1}\phi_i+\phi_i h_i$, concluding the proof of the lemma, and Proposition \ref{unprojection data construction}.
\end{proof}
We are interested in constructing resolutions for ideals that arise from genus one normal curves of degree $n$. In this situation, the ideals $I$ and $\mc{I}$ are generated by quadrics, and define genus one normal curves $C \subset \mb{P}^{m-1}$ and $\mc{C} \subset \mb{P}^m$. The curve $\mc{C}$ contains the point $P=(0:0:\ldots:0:1)$, and linear projection $\pi$ from $P$ extends to a map $\mc{C} \xrightarrow{} C$. The ideal $J$ is the defining ideal of the point $\pi(P)$ in $\mb{P}^{n-1}$.

\begin{definition} \label{unprojection triple}
	Let $S=k$ be a field, so that $R=k[x_1,\ldots,x_{m-1}]$. Let $I$ and $J$ be Gorenstein ideals of $R$, and let $\mc{I}$ be an ideal of $\mc{R}$. We say $I, J$ and $\mc{I}$ form an unprojection triple if they satisfy the following conditions:
	\begin{enumerate}[label=(\roman*)]
		\item $I$,$J$ and $\mc{I}$ are generated by forms of degrees $2$, $1$ and $2$ respectively.
		\item $\mc{I} \cap R=I$ and $I \subset J$.
		\item $J$ is the set of leading coefficients of $\mc{I}$, viewed as an ideal of $\mc{R}=R[x_m]$.
		\item $I$ and $J$ are Gorenstein ideals of codimension $n-2$ and $n-1$, respectively. 
		\item Let  $(F_{\bullet},\phi)$  and $(G_{\bullet},\psi)$ be the minimal graded free resolutions of $I$ and $J$. Then the grading on leftmost modules in $F_{\bu}$ and $G_{\bu}$ is specified by  $F_{n-2}\cong R(-n)$ and $G_{n-2} \cong R(-n+1)$. 
	\end{enumerate}
\end{definition}
We extend this definition to $n=3$ by modifying the condition (i) to require that $I$ is generated by a single form of degree 3.

For this section we adopt the convention that the dual of the graded module $M$ is the module $M^*=\mathrm{Hom}(M,R(-n+1))$. Then the fact that $I$ and $J$ are Gorenstein implies that $G_{\bullet} \cong G^{*}_{\bullet}$ and $F_{\bullet} \cong F^{*}_{\bullet}(1)$.

\subsection{The map $\alpha$ for an unprojection triple.} The inclusion $I \subset J$ induces the map $\beta : F_{\bullet}(-1) \xrightarrow{} G_{\bullet}(-1)$ in the same way as before. We now give an alternative way to describe the map $\alpha$. Fix a basis $a_1,\ldots,a_t$ of $F_1$ and a basis $b_1,\ldots,b_s$ of $G_1$. Then $I$ is generated by $\phi_1(a_i)=f_i$, $i=1,\ldots,t$ and $J$ is generated by $\psi_1(b_i)=g_i$, $i=1,\ldots,g_s$. By conditions (i), (ii) and (iii), there exist quadratic forms $h_i \in R$ such that 
\[
\mc{I}=\left(f_1,\ldots,f_t,x_m g_1+h_1,\ldots,x_m g_{s}+h_s\right) .
\]
Define $\alpha_{1}:G_1 (-1) \xrightarrow{} F_0$ by $\alpha(b_i)=h_i$. We extend $\alpha$ to a map of chain complexes $G(-1) \xrightarrow{} F[-1]$, represented by the following diagram.
\begin{equation*} 
	\begin{tikzcd}
		0 \ar[r]& G_{n-1}(-1)\ar[r, "\psi_{n-1}"]  \ar[d,"\alpha_{n-1}"] &\ar[r, "\psi_{3}"] \ldots& G_{2}(-1) \ar[r, "\psi_{2}"] \ar[d,"\alpha_{2}"] & G_{1}(-1) \ar[r, "\psi_{1}"]  \ar[d,"\alpha_1"] & G_{0}(-1) \ar[d] \\
		0 \ar[r] & F_{n-2} \ar[r,"\phi_{n-2}"] &\ldots\ar[r,"\phi_2"] & F_1 \ar[r,"\phi_{1}"] & F_0 \ar[r] & 0
	\end{tikzcd} 
\end{equation*}
To do this, we check that the image of the map $\alpha_1 \psi_2$ is contained in $\phi_{1}(F_1)=I$. This implies that there exists a map $\alpha_{2}:G_2(-1) \xrightarrow{} F_1$ with $\phi_1\alpha_2=\alpha_1\psi_2$. As the complex $F[-1]$ is acyclic in degrees greater than 2, and $G(-1)$ is a complex of projective modules, we can extend $\alpha_1$ to a chain map.

Let $e_1,e_2,\ldots,e_p$ be a basis of $G_2$. Then $\psi_2(e_j)=\sum \xi_{ij}b_i$ for some $\xi_{ij} \in R$. Since $\psi_1\psi_2=0$ we have $\sum \xi_{ij}g_i=0$. Then
\[
\alpha_1 \psi_2 (e_j)=\sum \xi_{ij}h_i=\sum \xi_{ij}(x_{m}g_i+h_i) \in \mc{I} \cap R=I
\]
as desired. We now check that this agrees with our previous definition of $\alpha$.
\begin{lemma} \label{lemma part of unproj data}
Let $I,J$ and $\mc{I}$ be an unprojection triple, and suppose that the ideal $\mc{I}$ is prime. Then the map $\alpha:G_{\bullet} (-1) \xrightarrow{} F_{\bullet}[-1]$, as defined above, is chain homotopic to a non-zero scalar multiple of the map $\beta^{T}:G_{\bullet} (-1) \xrightarrow{} F_{\bullet}[-1]$, obtained by dualizing the map $\beta$.
\end{lemma}

\begin{proof}
Consider the diagram obtained by dualizing the diagram defining the map $\alpha$.
	\begin{equation*}
		\begin{tikzcd}
			0 \ar[r]& F^*_0 \ar[r, "\phi^*_{1}"]  \ar[d,"\alpha_{1}^*"] &\ar[r, "\phi_{2}^*"] \ldots& {F}^*_{n-4} \ar[r, "\phi_{n-3}^*"] \ar[d,"\alpha_{n-3}^*"] & {F}^*_{n-3} \ar[r, "\phi^*_{n-2}"]  \ar[d,"\alpha^*_{n-2}"] & {F}^*_{n-2} \ar[d,"\alpha^*_{n-1}"] \\
			{G}^*_{0}(1)  \ar[r,"\psi_{1}^*"] & {G}^*_{1}(1) \ar[r,"\psi_{2}^*"] &\ldots\ar[r,"\psi_{n-3}^*"] & {G}^*_{n-3}(1) \ar[r,"\psi^*_{n-2}"] & {G}^*_{n-2}(1) \ar[r,"\psi^*_{n-1}"] & {G}^*_{n-1}(1)
		\end{tikzcd}	
	\end{equation*}
     It suffices to show that $\alpha^{*}_{n-1} : F^{*}_{n-2}(-1) \xrightarrow{} G^{*}_{n-1}(1)$ is an isomorphism of graded modules, since the above diagram is isomorphic to the one defining the map $\beta$.
	
	As $\alpha^{*}_{n-1}$ respects the grading, it must be multiplication by a scalar element of $k$. To show it is an isomorphism, it suffices to show that it is non-zero. 
	
	Suppose otherwise. Then $\alpha^{*}$, as a lift of the zero map, is null-homotopic, and so there exist maps $\rho_i : {F}^*_i \xrightarrow{} {G}^*_i$, with $\rho_{n-3}=0$, and 
	\[
	\alpha^*_i=\rho_i \phi^*_i+\psi^*_i\rho_{i-1}
	\]
	for all $i \le n-3$. In particular $\alpha^*_1=\rho_1 \phi^*_1+\psi^*_1\rho_0$, and dualizing again, we see  $\alpha_1=\rho^*_1 \phi_1+\psi_1\rho^*_0$. As $\rho^*$ is a degree zero map from ${G}_0(-1) \cong R(-1)$ to ${F}_{0} \cong R$, it must be multiplication by some linear form $l \in R$. For $1\leq i \leq s$, we find 
	\[
	lg_i-h_i=\rho^*_0 \psi_1(b_i)-\alpha_1(b_i)=\phi_1(\rho^*_1(b_i)) \in I \subset \mc{I}.
	\]
		As $x_m g_i + h_i \in \mc{I}$, we have $(l+x_m)g_i \in I$ for all $i$, and hence $(l+x_m) g=0$ for every linear form $g \in J$, contradicting our assumption.	
\end{proof}
To summarise,  by Lemma \ref{exactness lemma 2}, the complex $\mc{H}$ is the minimal free resolution of the ideal $\mc{I}$, constructed entirely from the data of resolutions $F_{\bullet}$, $G_{\bullet}$ and the comparison map $\beta$. In other words, if $I,J$ and $\mc{I}$ form an unprojection triple, we can recover $\mc{I}$ from $I$ and $J$.

\subsection{Unprojection of genus one normal curves of degree $n$}
 We now apply the machinery developed in the previous section to genus one curves. Throughout, $k$ will be a field, $n \geq 3$ an integer, and $R=k[x_1,\ldots,x_{n}]$ and $\mc{R}=k[x_1,\ldots,x_{n+1}]$. 
 
 We need some preliminary lemmas. The following argument is a special case of the well known fact that genus one normal curves embedded in projective space are projectively normal.

\begin{lemma} \label{proj norm lemma}
	Let $E$ be an elliptic curve over a field $k$ of characteristic $0$, $D$ a divisor of $E$ of degree $n \geq 3$, and let $x_1,x_2,\ldots,x_n$ be a basis of the $n$-dimensional Riemann-Roch space $\mathcal{L}(D)$. Then the monomials $x_ix_j$, $1\leq i,j \leq n$ span the $2n$-dimensional space $\mathcal{L}(2D)$. 
\end{lemma}

\begin{proof}
	For the proof we are free to extend our base field and assume $k=\bar{k}$ is algebraically closed, and to replace $D$ by any linearly equivalent divisor. For a point $Q \in E(k)$, let $\tau_{Q}$ denote the translation by $Q$ map. We may also replace $D$ by $\tau^*_{Q}(D)$ for any $Q \in E(k)$. Now recall that for an elliptic curve, two divisors are linearly equivalent if and only if they have the same degree and the same sum. 
	
	As $k$ is algebraically closed, there exists a $Q \in E(k)$ with $n\cdot Q= \mathrm{sum}(D)$, so that $\mathrm{sum}(\tau^*_{Q}(D))=\mathrm{sum}(D)-nQ=0_E$. Hence $\mathrm{sum}(\tau^*_{Q}(D))$ has sum $0_E$ and degree $n$, and so is linearly equivalent to $n \cdot 0_E$. Thus we may replace $D$ with $n \cdot 0_E$. Note that it suffices to prove the claim for a single choice of the basis $x_1,x_2,\ldots,x_n$, i.e. we are asserting that the natural map $S^2(\mathcal{L}(D)) \xrightarrow{} \mathcal{L}(2D)$ is surjective. As $1,x,y,x^2,xy,\ldots,x^{m/2}$ form a basis of $\mathcal{L}(m \cdot 0_E)$ for $m$ even, and $1,x,y,x^2,xy,\ldots,x^{\frac{m-1}{2}}y$ form a basis for $m$ odd, this statement is clear. 
\end{proof}
In what follows, let $\mc{C}\subset \mathbb{P}^{n}$ be a genus one normal curve of degree $n+1$ defined over a field $k$, where $n \geq 3$,  that passes through the point $P=(0:0:\ldots:0:1)$.
\begin{lemma}
The projection map $\pi$ from the point $P$,
	\[
	\pi: \mathbb{P}^{n} \xrightarrow{} \mathbb{P}^{n-1}: (x_1:\ldots:x_{n+1}) \xrightarrow{} (x_1:\ldots:x_{n}),
	\]
    restricts to a regular map $\mc{C} \xrightarrow{} \mb{P}^{n-1}$, also denoted $\pi$. The image of $\pi$ is a genus one normal curve of degree $n$, denoted $C$. 
\end{lemma}
\begin{proof}
 We identify $C$ with its Jacobian $E=\mathrm{Pic}^0(C)$ via the map $Q \mapsto (Q)-(P)$. Let $H$ be a hyperplane in $\mathbb{P}^{n-1}$ passing through $P$, and regard $D=\mc{C} \cap H$ as a divisor of $E$. Then the embedding $\mc{C} \subset \mathbb{P}^{n}$ is identified with the embedding of $E$ into $\mathbb{P}^{n}$ associated to the complete linear system $|D|$, and the forms $x_i$, for $i \leq n$, are identified with elements of $\mc{L}(D-P)$. As $\mc{C}$ is a genus one normal curve, $x_1,\ldots,x_{n}$ is basis of of $\mc{L}(D-P)$. Hence $C \subset \mb{P}^{n-1}$ is the embedding associated to the complete linear system $|D-P|$, and thus is a genus one normal curve of degree $n$.
\end{proof}

\begin{lemma} \label{geometric criterion for unprojection triple}
	 Let $\mc{J}$ be the ideal generated by the linear forms in $\mc{R}$ that vanish on the tangent space $T_{P}C$ at $P$. Let $\mc{I}$ and $I$ be the homogeneous ideals of $\mc{R}$ and $R$ that define the curves $\mc{C}$ and $C$ respectively. Then the ideals $I,J=\mc{J} \cap R$ and $\mc{I}$ form an unprojection triple.
\end{lemma}
\begin{proof}
	We check condition (i) od Definition \ref{unprojection data}. By Theorem \ref{min res theorem}, ideals $\mc{I}$ and $I$ are generated by quadrics. By definition, $J$ is generated by linear forms.
	
	For condition (ii),  as $C=\pi(\mc{C})$, we have $I=\mc{I} \cap R$. We may identify the space of linear forms vanishing at $T_P \mc{C}$ with the Riemann-Roch space $\mathcal{L}(D-2P)$. As $D-2P=(D-P)-P$, we can identify $\mathcal{L}(D-2P)$ with the space of linear forms on $\mb{P}^{n-1}$ that vanish at $\pi(P)$, and the ideal $J=\mc{J} \cap R$ with the ideal generated by these forms. As $\pi(P) \in C(k)$, we have $I \subset J$.
	
	Let us verify condition (iii).  Identify $x_0,x_1,\ldots,x_n$ with a basis of $\mathcal{L}(D)$, with $x_0,\ldots,x_{n-1}$ also serving as a basis of $\mathcal{L}(D-P)$. As $\mc{I}$ is generated by forms of degree 2, its leading coefficient ideal, denoted $J'$, must be generated by linear forms. We claim that $J=J'$. Let us first show that $J' \subset J$.  
	
	Let $f \in \mc{I}$ be a quadric, and write $f=x_{n+1}l+h$, where $l$ is a linear form and $h \in R$, so that $l$ is the leading coefficient of $f$. The form $h$ is a quadric in $x_1,\ldots,x_{n}$, so vanishes at $P$ with multiplicity at least $2$. As $x_{n+1}$ does not vanish at $P$, $l$ must vanish at $P$ with multiplicity at least 2, and hence $l \in \mc{J}$. As $l$ is a linear form, it can be written as a linear combination of $x_1,\ldots,x_{n+1}$. Since $l(P)=0$, the coefficient of $x_{n+1}$ must be $0$, and thus $l \in R$, and hence $l \in J=\mc{J} \cap R$.
	
	By Lemma \ref{proj norm lemma}, the forms $x_ix_j$ for $1 \leq i,j \leq n$ span the space $\mathcal{L}(2D-2P)$. For any $l \in \mathcal{L}(D-2P)$, we have $x_{n+1}l \in \mathcal{L}(2D-2P)$. Thus by Lemma \ref{proj norm lemma}, we have $x_{n+1}l=\sum_{i,j} c_{i,j} x_i x_j$ for some constants $c_{i,j}$, and $x_{n+1}l-\sum_{i,j} c_{i,j} x_i x_j$ is a quadric in $\mc{I}$ with leading coefficient $l$. Thus $J'\subset J$, and hence $J=J'$.
	
	The condition (iv) of Definition \ref{unprojection data} follows from general results on the structure of minimal free resolutions of genus one normal curves, e.g. Theorem \ref{min res theorem}.
\end{proof}

We now give our first application, finishing the proof of Lemma \ref{first lemma differential}(ii). 

\begin{lemma} \label{unprojecting differential lemma}
Let $C_n \subset \mb{P}^{n-1}$ be a genus one normal curve of degree $n$, defined over a field $k$ of characteristic 0, let $I$ be the defining ideal of $C$, let $F^{n}_{\bu}$ be a resolution model of $C$ and let $\Omega$ be the $\Omega$-matrix associated to this data.

Then the rational differential form $\omega$, associated to $\Omega$, is regular on $C_n$. More precisely, for $1 \leq i<j \leq n$,  the differential 
\[
	\omega=(n-2)\frac{x_j^2 d(x_i/x_j)}{\Omega_{ij}(x_1,\ldots,x_{n})}
\]
is regular on $C_n$. 
\end{lemma}

\begin{proof}
The proof is by induction on $n$. The base case $n=3$, is just the statement that if $F$ is the ternary cubic that defines $C$, the differential $\frac{x_1^2 d(x_2/x_1)}{\partial F/ \partial x_3}$ is regular on $C$, and this is an elementary fact from the theory of plane curves. Let us suppose that lemma is true for $n-1$.	
	
For the proof of the statement we are free to enlarge the base field, and assume that $k=\bar{k}$ is algebraically closed, and thus that there exists a $k$-rational point $P \in C_n(k)$. By Lemma \ref{invariant differential change of coords}, we are also free to make a change of coordinates on $\mb{P}^{n-1}$, and thus assume that $P=(0:0:\ldots:0:1)$. Let $i,j,a_1,\ldots,a_{n-2}$ be a permutation of the set $\{1,2,\ldots,n\}$. From, Lemma \ref{single bracket double bracket rel } we have
	\[
\omega =\frac{x_j^2 d(x_i/x_j)}{[a_1,..,a_{n-2}]}.
\]
 By Lemma \ref{single bracket differential}, we are free to reorder  we can reorder $(i,j,a_1,\ldots,a_{n-2})$ as we like, since that can only affect the sign of $\omega$. Thus we choose an ordering so that $n-1 \not\in \{i,j\}$.

 Now let $\pi$ be the projection from point $P$, and let $C_{n-1}=\pi(C_{n})$. Let $I_{n-1}$ and $I_n$ be the ideals defining $C_{n-1}$ and $C_n$, and let $J$ be the ideal defining $\pi(P)$.  By Lemma \ref{geometric criterion for unprojection triple}, $I_{n-1},J$ and $I_n$ form an unprojection triple. Let $F^{n-1}_{\bullet}$ and $G_{\bu}$ be minimal graded free resolutions of $I_{n-1}$ and $J$. 
 
The matrix $\Omega$, and hence the differential $\omega$, depends on the choice of the resolution $F^{n}_{\bu}$ only up to a non-zero scalar multiple. Thus, we may assume $F^{n}_{\bu}$ is the chain complex $\mc{H}(\alpha,\beta,\delta)$ constructed from $F^{n-1}_{\bu}$ and $G_{\bu}$. By Proposition \ref{quadrics and unprojection}, 
	\[[a_1,a_2,\ldots,a_{n-2}]=\pm [a_1,a_2,\ldots,\widehat{n},\ldots,a_{n-2}],\]
since we assumed that $a_k=n$ for some $k$. As the projection $\pi$ is given by $(x_1:\ldots:x_{n-1} : x_{n}) \mapsto (x_0 : \ldots :x_{n-1})$, we have
\[
\frac{x_j^2 d(x_i/x_j)}{[a_1,\ldots,a_{n-2}]}=\pi^{*} \left( \pm  \frac{x_j^2 d(x_i/x_j)}{ [a_1,a_2,\ldots,\widehat{n},\ldots,a_{n-2}]} \right) 
\]
and we are done by the induction hypothesis, as $\pi : C_n \xrightarrow{} C_{n-1} $ is an isomorphism.
\end{proof}

\subsection{The local minimization theorem}
We now give our first arithmetic application, the local version of the minimization theorem. Throughout, we fix a prime $p$. Let $\mb{Q}_p$ be the field of $p$-adic numbers and let $\mb{Z}_p$ be the ring of $p$-adic integers. The following theorem is a generalization of Theorem 3.4 of \cite{minred234}. 
\begin{theorem} \label{local minimization theorem}
Let $p$ be a prime, and let $C \subset \mb{P}^{n-1}$ be a genus one normal curve of degree $n$, defined over the field of $p$-adic numbers $\mb{Q}_p$. Suppose that the set $C(\mb{Q}_p)$ is non-empty. Let $E$ be the Jacobian of $E$, defined by a minimal Weierstrass equation $W$ 
\[
y^2 + a_1xy + a_3y = x^3 + a_2x^2 + a_4x + a_6.
\]
Then one can choose coordinates on $\mb{P}^{n-1}$, and then a $\mb{Z}_p$-integral genus one model $F_{\bu}$ of $C$ with the associated invariant differential $\omega$, such that there exists an isomorphism $\gamma : E \xrightarrow{} C$ with 
\[
\omega=\gamma^{*}\left( \frac{dx}{2y + a_1x + a_3} \right)
\]
We say such a genus one model is a \textit{minimal} model of $C$. 
\end{theorem}

Once we prove Theorem \ref{formula for Jacobian}, this theorem will immediately imply the following corollary.
\begin{corollary}
Let $n$ be an odd integer. If $\Omega$ is the $\Omega$-matrix of a minimal model of a genus one curve $C$ as above, we have $c_k(
\Omega)= c_k(W)$, for $k=4,6$, where $c_k(W)$ are the $c$-invariants of a minimal Weierstrass equation $W$ of $E$, as defined in Chapter III of \cite{silverman2009arithmetic}.
\end{corollary}

The induction step in the proof of the theorem is provided by the following lemma, which may be viewed as a generalization of Lemma 3.14 of \cite{minred234}.

\begin{lemma} \label{induction step lemma}
	
	Let $E/\mb{Q}_p$ be an elliptic curve and let $D$ be a degree $n$ divisor on $E$, for $n \geq 3$. Let $i : E \xrightarrow{} \mb{P}^{n-1}$ be an embedding induced by a choice of basis of $\mc{L}(D)$, and let $C_n=i(E) \subset \mb{P}^{n-1}$. Suppose we are also given a point $P \in E(\mb{Q}_p)$.  
	
	Let $F_{\bullet}$ be a $\mb{Z}_p$-integral genus one model of $C_n$, with $I_n=\phi_{1}(F_1) \subset \mb{Z}_p[x_1,\ldots,x_{n}]$. Further, let $l_1,\ldots,l_{n-1}$ be a $\mb{Z}_p$-basis for the space of linear forms with integral coefficients that vanish at $i(P) \in \mb{P}^{n-1}$, and denote by $G_{\bu}$ the Koszul 
	complex associated to $J=(l_1,l_2,\ldots,l_{n-1})$. Then:
	\begin{enumerate}[label=(\roman*)]
		\item	The complexes $F_{\bu}$ and $G_{\bu}$ satisfy the hypothesis of Proposition \ref{unprojection data construction}, and the associated complex $\mc{H}_{min}$ is a $\mb{Z}_p$-integral genus one model of $C_{n+1} \subset \mb{P}^{n}$, where $C_{n+1}$ is the image of $E$ under an embedding induced by a choice of basis of $\mc{L}(D+P)$. 

\item The point $P$ maps to the point $(0:0:\ldots:0:1) \in C_{n+1}(\mb{Q}_p)$, and projection from $P$ defines an isomorphism $\pi : C_{n+1} \xrightarrow{} C_{n}$. Moreover, if $\omega_n$ and $\omega_{n+1}$ are differentials associated to models $F_{\bu}$ and $\mc{H}_{min}$, we have $\pi^{*}(\omega_{n})=\pm \omega_{n+1}$.
 
\end{enumerate} 

\end{lemma}

\begin{proof}
\begin{enumerate}[label=(\roman*)]
\item
We first prove that $l_1,\ldots,l_{n-1}$ is a regular sequence on $\mb{Z}_p[x_1,\ldots,x_{n}]$. By the definition of the forms $l_i$ and the theory of elementary divisors, we can find a linear form $l_{n} \in \mb{Z}_p[x_0,\ldots,x_{n-1}]$ so that $l_1,\ldots,l_{n}$ is a basis for the $\mb{Z}_p$-module of integral linear forms in $\Z_p[x_1,\ldots,x_{n}]$, so that we in fact have $\mb{Z}_p[x_1,\ldots,x_{n}]=\mb{Z}_p[l_1,\ldots,l_{n}]$, and hence $l_1,\ldots,l_{n-1}$ is a regular sequence. 

Recall from \ref{koszul def} that the Koszul complex $G_{\bu}$ is a free resolution of the ideal $J=(l_1,l_2,\ldots,l_{n-1})$ and that the DGC structure on $G_{\bu}$ induces an isomorphism $G_{\bu} \xrightarrow{} G_{\bu}^*$. By definition $F_{\bu}$ is a free resolution of $I_{n}$, with a DGC-structure inducing an isomorphism $F_{\bu} \xrightarrow{} F_{\bu}^*$ 

We next prove that $I_n \subset J$. Since we have $\mb{Z}_p[x_1,\ldots,x_{n}]=\mb{Z}_p[l_1,\ldots,l_{n}]$, every quadric $f \in I_n$ can be expressed a $\sum_{ij} c_{ij} l_i l_j$, with $c_{ij} \in \Z_p$.  As $f(P)=0$, the coefficient of $l^2_{n}$ is zero, and $f \in J$. Thus the conditions of Proposition \ref{unprojection data construction} are satisfied.
\item We are free to extend scalars from $\mb{Z}_p$ to $\mb{Q}_p$, and identify $I_n$ and $J$ with ideals of $\mb{Q}_p[x_1,\ldots,x_{n}]$. Let $I_{n+1}$ be the ideal of $\mb{Q}_p[x_1,\ldots,x_{n+1}]$ resolved by $\mc{H}^{min}(\alpha,\beta,\delta)$. There exists a choice of basis of $\mc{L}(D+P)$ so that the corresponding embedding $j:E \xrightarrow{} \mb{P}^{n}$ we have $j(P)=(0:\ldots:0:1)$, and projection $\pi$ from $P$ defines an isomorphism $C_{n+1} \xrightarrow{} C_n$.

 Let $I'_{n+1}$ be the ideal defining $C_{n+1}$. By Lemma \ref{geometric criterion for unprojection triple}, $I_n,J$ and $I_{n+1}$ are an unprojection triple, and so $I'_{n+1}$ is resolved by the complex $\mc{H}^{min}(\alpha',\beta',\delta')$. By Lemma \ref{independence of choice}, $\mc{H}^{min}(\alpha',\beta',\delta')$ is isomorphic to $\mc{H}^{min}(\alpha,\beta,\delta)$, and hence we have $I'_{n+1}=I_{n+1}$.

Finally, since the invariant differential $\omega_n$ associated to $F_{\bu}$ is
\[
\omega_n=\frac{(n-2)x_j^2 d(x_i/x_j)}{\Omega_{ij}(x_1,\ldots,x_{n})}=\frac{x_j^2 d(x_i/x_j)}{[a_1,\ldots,a_{n-2}]},
\]
 the assertion that $\pi^{*}(\omega_n)=\pm \omega_{n+1}$ follows from Lemma \ref{quadrics and unprojection}, in the same way as in the proof of Lemma \ref{unprojecting differential lemma}.                                                                       \end{enumerate}                                                              
\end{proof}
 
\begin{proof}[Proof of Theorem \ref{local minimization theorem}.]
The proof is by induction. Choose a point $P \in C_n(\mb{Q}_p)$. As $\mb{Z}_p$ is a principal ideal domain, by elementary divisor theory, the group $\mathrm{SL}_n(\mb{Z}_p)$ acts transitively on $\mb{P}^{n-1}(\mb{Q}_p)$. We may assume, making a $\mathrm{SL}_n(\mb{Z}_p)$-change of coordinates if necessary, that $P=(0:0:\ldots:0:1)$. Let $C_{n-1}=\pi(C_n) \subset \mb{P}^{n-2}$ be the projection of $C_n$ from $P$. By Lemma \ref{induction step lemma}, the assertion of the theorem is true for the curve $C_{n-1}$ if and only if it is true for the curve $C_{n}$.

We are reduced to the case $n=3$. Now let $D$ be a hyperplane section of $C_3$, identified with a degree 3 divisor on $E$. The description of the group law on $E$ implies that $D$ is linearly equivalent to a divisor $4 \cdot 0_E - Q$, for some $Q \in E(\mb{Q}_p)$. By Lemma \ref{induction step lemma}, it suffices to prove the theorem for the curve $C_4$, the image of $E$ embedded via the linear system $|4 \cdot 0_E|$. Applying Lemma \ref{induction step lemma} for the final time, we are reduced to the case of $E$ embedded in $\mb{P}^2$ by $|3 \cdot 0_E|$. Choose the basis $1,x,y$ of $\mc{L}(3\cdot 0_E)$. Then the image of $E$ in $\mb{P}^2$ is defined by the cubic
\[
F=y^2z+a_1xyz+a_3yz^2-x^3-a_2x^2z-a_4xz^2-a_6z^3
\]
and we compute directly that the differential $\omega$ associated to $F$ agrees with the Neron differential on $E$.
\end{proof}

\subsection{Proof of the global minimization theorem}

We now deduce Theorem \ref{global minimization theorem} from the local result, Theorem \ref{local minimization theorem}, using (basic) results on strong approximation. The idea is to begin by choosing arbitrarily a resolution model $F_{\bu}$ of $C \subset \mb{P}^{n-1}$ defined over $\mb{Q}$. Clearing denominators, we may assume $F_{\bu}$ is defined over $\mb{Z}$. It follows from Theorem \ref{formula for Jacobian} that $F_{\bu}$ is potentially not minimal for a finite set $S$ of primes, namely those that divide both $c_4(F_{\bu})$ and $c_6(F_{\bu})$. For every $p \in S$, we construct a minimal model using Theorem \ref{local minimization theorem}. We then use strong approximation to modify $F_{\bu}$, so that it is $\mb{Z}_p$-equivalent to this minimal model for every $p \in S$, while preserving minimality for $p \not\in S$, concluding the proof.  This is essentially the approach taken in \cite{fisher2007new} to prove a minimization theorem for binary quartics and ternary cubics.

 For a prime $p$, let $||\cdot||_p$ be the $p$-adic absolute value. For any $A=(a_{ij}) \in \mathrm{Mat}_{n}(\mb{Q}_p)$, let $||A||_{p}=\max_{1 \leq i,j \leq n}||a_{ij}||$.
\begin{lemma}\label{surjective SLn}
Let $S$ be a finite set of primes of $\mb{Q}$. Suppose we are given, for each $p\in S$, an element $A_{p} \in \mathrm{SL}_n(\mb{Z}_p)$. For any $\epsilon>0$, there exists $A \in \mathrm{SL}_n(\mb{Z})$ with $||A-A_p||_p < \epsilon$ for all $p \in S$.
\end{lemma}

\begin{proof}
This is Lemma 3.2 of \cite{fisher2007new}. Alternatively, this follows from the well known fact that for any integer $N$, the map $\mathrm{SL}_n(\mb{Z}) \xrightarrow{} \mathrm{SL}_n(\mb{Z}/N\mb{Z})$ is surjective ---see Lemma 6.3.10 in \cite{cohen2008number}.  
\end{proof}
Note that the result is false if we replace $\mathrm{SL}_n$ with $\mathrm{GL}_n$, as can already be seen in the case $n=1$. We however have the following result.
\begin{lemma} \label{generalization of surjective SLn}
Let $S$ be a finite set of primes of $\mb{Q}$ and let $\delta \in \mb{Z}$ be an integer. Suppose we are given, for each $p\in S$, $A_{p} \in \mathrm{Mat}_{n}(\mb{Z})$ with $\mathrm{det}(A_{p})=\delta$. Let $\epsilon > 0$.  There exists $A \in \mathrm{Mat}_n(\mb{Z})$ such that $\mathrm{det}( A)=\delta$, and $||A-A_{p}||_{p} < \epsilon$ for each $p \in S$	.
\end{lemma}
\begin{proof}
Follows from the previous lemma, together with a Smith normal form argument. See Lemma 3.3 of \cite{fisher2007new} for details.
\end{proof}

\begin{lemma} \label{approx lemma}
Let $A \in \mathrm{GL}_n(\mb{Q}_p)$. There exists an $\epsilon>0$, such that for every $B \in \mathrm{GL}_n(\mb{Q}_p)$ with $||A-B||_{p} < \epsilon$, we have $A=UB$, for some $U \in \mathrm{GL}_n(\mb{Z}_p)$.
\end{lemma}
\begin{proof}
We have $AB^{-1}=I_n+(A-B)B^{-1}=I_n+\mathrm{det}(B)^{-1}  (A-B)\mathrm{adj}(B)$, where $\mathrm{adj}(B)$ denotes the adjugate matrix of $B$. Since $\mathrm{det}(B)$ and $\mathrm{adj}(B)$ depend continuously on $B$, it is clear there exists an $\epsilon>0$ such that $||A-B||_p < \epsilon$ implies that $AB^{-1} \in \mathrm{Mat}_n(\mb{Z}_p)$ and $||\mathrm{det}(B)||_p=||\mathrm{det}(A)||_p$, and hence $AB^{-1} \in \mathrm{GL}_n(\mb{Z}_p)$.
\end{proof}
\begin{proof}[Proof of Theorem \ref{global minimization theorem}] To begin, choose any resolution model $F_{\bullet}$ that represents the $n$-diagram $[C \subset \mb{P}^{n-1}]$ and is defined over $\Q$. Clearing denominators if needed, we may assume $F_{\bu}$ is defined over $\mb{Z}$. By Theorem \ref{formula for Jacobian}, $c_k(F_{\bu})$ are the $c$-invariants of a Weierstrass equation for $E$. Hence there exists $\lambda \in \mb{Q}$ with $c_k(F_{\bu})=\lambda^k c_k(W)$. Let $S$ be the set of primes that contains all of the primes that appear in the numerator and denominator of $\lambda$, as well as the primes for which $F_{\bullet}$ is not $\mb{Z}_p$-integral. The idea is to use Lemma \ref{generalization of surjective SLn} to construct an integral model $F'_{\bu}$, that is $\mb{Q}$-equivalent to $F_{\bu}$, and $\mb{Z}_p$-equivalent to a minimal model for each $p \in S$. 
	
As $C(\mb{Q}_p)$ is non-empty for every $p \in S$, we may apply Theorem \ref{local minimization theorem}. Thus $F_{\bu}$ is $\mb{Q}_p$-equivalent to a $\mb{Z}_p$-integral model $F^{p}_{\bu}$ with $c_k(F^{p}_{\bu})= c_K(W)$. Recall that $\mc{G}_n=\mc{G}_n^{res} \times \mathrm{GL}_{n}$, where  $\mc{G}_n^{res}=\mathrm{GL}_{b_{n-1}} \times \ldots \times \mathrm{GL}_{b_{0}}$ acts on the modules in $F_{\bu}$, while $\mathrm{GL}_{n}$ acts by linear substitution in the variables $x_0,x_1,\ldots,x_{n-1}$. Let $g^{p}=(g^{p}_{n-2},\ldots,g^{p}_{1},g^{p}_0) \times (\gamma^{p}) \in \mc{G}_n(\mb{Q}_p)$ be such that we have $g^{p} \cdot F_{\bu}=F^{p}_{\bu}$.

We now modify the transformations $g^{p}$ so that we can apply Lemma \ref{generalization of surjective SLn}. We first show that we can assume that for every $p \in S$, we have $g^{p}_i \in \mathrm{Mat}_{b_i}(\mb{Z}_p)$ for all $i$, as well as $\gamma^{p} \in \mathrm{Mat}_{n}(\mb{Z}_p)$. To see this, observe that replacing $x_i$ by $t x_i$, for some non-zero $t \in \mb{Q}$ and all $i$, has the same effect as replacing $\phi_i$ by $t^{-\mathrm{deg}(\phi_i)} \phi_i$, for $1 \leq i \leq n-2$, as $\phi_i$ is a matrix of homogeneous forms in $x_1,\ldots,x_{n}$, of degree $\mathrm{deg}(\phi_i)$. But this is the same as rescaling the basis of each module $F_i$ by $t^{\sum_{j<i} \mathrm{deg}(\phi_j)}$. In other words, we can replace $g^{p}$ with $(t^{-n}\cdot g^{p}_{n-2},\ldots,t^{-2} \cdot g^{p}_{1},g^{p}_0) \times (t \cdot  \gamma^{p})$. We choose $t$ so that $t \cdot \gamma^{p} \in \mathrm{Mat}_n(\mb{Z}_p)$ for all $p$ in $S$, and then modify the transformations $g^{p}$ as described. Furthermore, for any $t' \in \mb{Q}_p$, we can rescale the basis of each $F_i$ by $t'$---this does not change the matrices $\phi_i$, and hence does not change $\lambda$. In doing this, we replace $g^{p}$ with $(t'\cdot g^{p}_{n-2},\ldots,t'\cdot g^{p}_{1},t' \cdot g^{p}_0) \times (\gamma^{p})$. We choose $t' \in \mb{Q}$ to ensure that $g^{p}_i \in \mathrm{Mat}_{b_i}(\mb{Z}_p)$ for all $i$.

Finally, we modify the local minimal models $F^{p}_{\bullet}$, so that we have, for each $i$ $\mathrm{det}(g^{p}_i)= \delta_i$ for some $\delta_i \in \mb{Z}$, as well as $\mathrm{det}(\gamma^{p})=\delta$, for some $\delta \in \mb{Z}$. We do this as follows. Let $\mathrm{det}(g^{p}_i)=\delta^{p}_i$. As the set $S$ is finite, and $\mb{Z}$ is a principal ideal domain, there exists a $\delta_i \in \mb{Z}$ with $||\delta_i||_{p}=||\delta^{p}_i||_{p}$ for every $p \in S$, and $||\delta_i||_{p}=0$ for $p \not\in S$. Now multiply a basis vector of $F^{p}_{i}$ by $\delta_i/\delta^{p}_i$---this has the effect of rescaling a column of $g^{p}_i$ by $\delta_i/\delta^{p}_i$, and corresponds to a $\mc{G}_n(\mb{Z}_p)$-transformation, so the modified model is still minimal. We then do the same for the matrix $\gamma^p$. 

By Lemma \ref{approx lemma} and Lemma \ref{generalization of surjective SLn}, there exist $g'_i \in \mathrm{Mat}_{b_i}(\mb{Z})$ and $\gamma' \in \mathrm{Mat}_{n}(\mb{Z})$, with the property that for all $p \in S$, we have $g^{p}_i=U^{p}_{i} g'_i$ and $\gamma^{p}=U^{p} \gamma$, for some $U^{p}_{i} \in \mathrm{GL}_{b_i}(\mb{Z}_p), U^{p} \in \mathrm{GL}_n(\mb{Z}_p)$.  Put $F'_{\bu}=g' \cdot F_{\bu}$. Then $F'_{\bu}$ is $\mb{Z}_p$-equivalent to a minimal model for every $p \in S$ via the transformation $(U^{p}_{n-2},\ldots,U^{p}_0) \times (U^{p})$. 

The model $F'_{\bu}$ is defined over $\mb{Z}$. Indeed, it suffices to check it is $\mb{Z}_p$-integral for every prime $p$. For primes $p \in S$, this is by construction. For primes $p \not\in S$, by construction we have $g' \in \mc{G}_n(\mb{Z}_p)$, and as $F_{\bu}$ is integral, then so is $F'_{\bu}$. 

Next, we prove that for every prime $p$ we have $||c_k(F'_{\bu})||_p=|| c_k(W)||_p$. For primes $p \in S$, this is by construction. For $p \not\in S$, observe that our modification of $F_{\bu}$ does not change $||c_k(F_{\bu})||_p$ for $p \not\in S$, and since we assumed $||\lambda||_{p}=1$ for $p \not\in S$, we have $||c_k(F'_{\bu})||_p=||c_k(F'_{\bu})||_p=|| F_{\bu}||_{p}$. Hence, for every prime $p$ we have $||c_k(F'_{\bu})||_p=|| c_k(W)||$. Thus  $c_k(F'_{\bu})=\pm c_k(W)$.  By Theorem \ref{formula for Jacobian}, we have $c_k(F'_{\bu})=c_k(W)$, and hence the model $F'_{\bu}$ is a global minimal model.	
\end{proof}

\section{Analytic theory of elliptic normal curves} \label{analytic chapter}
In this Section we finish the proof of Theorem \ref{formula for Jacobian}. In Section \ref{background heisenberg section} we review necessary background material. Let $[C \xrightarrow{} \mb{P}^{n-1}]$ be an $n$-diagram, defined over $\mb{C}$, and let $E$ be the Jacobian of $C$.  The action of the $n$-torsion subgroup $E[n]$ on $C$ lifts to an action on $\PP^{n}$ by projective linear transformations.  We explain how to use the theory of elliptic functions to choose coordinates on $\mb{P}^{n-1}$ so that this action of $E[n]$ and the embedding $C \subset \mb{P}^{n-1}$ take a particularly simple form, known as the Heisenberg normal form. This construction is classical, and dates back to Klein. In Section \ref{formula for Jacobian} we obtain $q$-expansions for the coefficients of $\Omega$-matrices in Heisenberg normal form, reducing the proof of Theorem \ref{formula for Jacobian} to a calculation already performed in \cite{jacobians}.

\subsection{Heisenberg invariant diagrams} \label{background heisenberg section}

The standard reference for the material in this section is Chapter I of \cite{hulek}. Let $n\geq 3$ be an odd integer. A remark on the notation---in this Section, we label the coordinates on $\mb{P}^{n-1}$ starting from 0, as $X_0,\ldots,X_{n-1}$, since this will simplify the notation slightly.
 
We recall some basic properties of the Weierstrass $\sigma$-function, the proofs of which can be found in \cite{Sil2}. Fix $n \geq 3$ odd. Let $\Lambda=\mb{Z}\omega_1 \oplus \mb{Z} \omega_2 \subset \mb{C}$ be a lattice, with $\mathrm{Im}(\omega_2/\omega_1)>0$. The quasi-period map $\eta : \Lambda \xrightarrow{} \mb{C}$ is defined by
\begin{equation*}
\eta(\omega):=-\int^{\omega}_{0} \wp(z,\Lambda)  dz
\end{equation*}
where $\wp(z,\Lambda)$ is the Weierstrass p-function. Let $\eta_1=\eta(\omega_1)$ and $\eta_2=\eta(\omega_2)$ be the period constants of $\Lambda$. They satisfy the Legendre relation:  $\omega_2 \eta_1-\omega_1 \eta_2 = 2\pi i$. The Weierstrass sigma function $\sigma(z,\Lambda)$ is defined as the infinite product
	\[
		\sigma(z,\Lambda):=z \prod_{\omega \in \Lambda, \omega \ne 0} (1-z/\omega)e^{z/\omega+(1/2)(z/\omega)^2}
		\]
Now let $\sigma_{pq}(z,\Lambda):=\sigma(z-\frac{p \omega_1+q \omega_2}{n},\Lambda)$, and define constants \[
\omega:=e^{-\frac{n-1}{2}\frac{\eta_2 \omega_1}{2}}, \theta:=e^{-\frac{\omega_1}{2n}}
\]
For all $m \in \mb{Z}$, we define functions $x_m(z,\Lambda)$ by the product
\[
x_m(z, \Lambda) := \omega^m \theta^{m^2} e^{m} \prod_{i=0}^{n-1} \sigma_{m,i}(z, \Lambda).
\]

These functions have the following properties (Theorem I.2.3 of \cite{hulek}):

\begin{proposition} \label{factor_automorphy}
	\begin{enumerate}[label=(\roman*)]
		\item For all $m \in \mb{Z}$ we have $x_{n+m}=x_{m}$.
		\item Let $\zeta:=e^{-\frac{2 \pi i}{n}}$. There exist 	holomorphic nowhere vanishing functions $j,j_1$ and $j_2$, such that,  for all $m \in \mb{Z}$, we have the following formulas:
		\begin{enumerate}
			\item $x_m(-z,\Lambda) =j(z) x_{-m}(z,\Lambda)  $,
			\item $x_m(z-\frac{\omega_1}{n},\Lambda) = j_1(z) \cdot x_{m+1}(z,\Lambda)$,
			\item $x_m(z+\frac{\omega_2}{n},\Lambda) =j_2(z) \cdot \zeta^{m} x_m(z,\Lambda)$
			\item $x_m(0,\Lambda)=-x_{n-m}(0,\Lambda)$
		\end{enumerate}
	\end{enumerate}
\end{proposition}
The proof is a straightforward calculation, see Section I of \cite{hulek}.  

\begin{proposition} \label{line_bundle heisenberg diagram }
The functions $x_0,\ldots,x_{n-1}$ can be viewed as a basis of the space of global sections of the line bundle $\mathcal{O}_{\mb{C}/\Lambda}(n \cdot 0)$, and hence the map

\[
\mb{C}/\Lambda \xrightarrow{} \mb{P}^{n-1} : z \mapsto (x_0(z,\Lambda):\ldots.:x_{n-1}(z,\Lambda))
\]
defines an $n$-diagram, which we denote by $B_{\Lambda}$. Furthermore, identifying $\mb{C}/\Lambda$ with its image, which we denote by $C_{\tau}$, translation by $\frac{\omega_1}{n}$ extends to the automorphism of $\mb{P}^{n-1}$ defined by
\[
(x_0 : x_1 :\ldots :x_{n-1}) \mapsto (x_{n-1} : x_ 0 : \ldots : x_{n-2}),
\]
translation by $\frac{\omega_2}{n}$ is defined by
\[
(x_0 : x_1 : \ldots : x_{n-1}) \mapsto (x_0 : \zeta x_1 :\ldots.:\zeta^{n-1}x_{n-1}),
\]
and the inversion map $z \mapsto -z$ corresponds to
\[
(x_0 : x_1 : x_2 \ldots x_{n-2}: x_{n-1}) \mapsto (x_0 : x_{n-1} : x_{n-2}\ldots:x_2 : x_1).
\]
\end{proposition}

This proposition follows immediately from the previous one, except for the statement on the linear independence of the functions $x_0,\ldots,x_{n-1}$, the proof of which can be found in Section I of \cite{hulek}.

\begin{remark}\label{heisenberg_group}
The projective representation of $E[n]$ on $\mb{P}^{n-1}$ lifts to a linear representation of the Heisenberg group $H_n$ on an $n$-dimensional vector space $V$. $H_n$ can be defined as the subgroup of $\mathrm{GL}_n(\mb{C})$ generated by the lifts $\sigma_{\omega_1/n}$ and $\sigma_{\omega_2/n}$ of the translation maps by $\frac{\omega_1}{n}$, $\frac{\omega_2}{n}$:
\[
\sigma_{\omega_1/n}:=\begin{pmatrix}
0 & 0 & 0 & \hdots & 0 & 1 \\
1 & 0 & 0 & \hdots & 0 & 0 \\
0 & 1 & 0 & \hdots & 0 & 0 \\
\vdots & \vdots &\vdots & \vdots & \vdots&\vdots \\
0 & 0 & 0 & \hdots & 1 & 0 \\
\end{pmatrix}, \ 
\sigma_{\omega_2/n}:=\begin{pmatrix}
1 & 0 & 0 & \hdots & 0 & 0 \\
0 & \zeta & 0 & \hdots & 0 & 0 \\
0 & 0 & \zeta^2 & \hdots & 0 & 0 \\
\vdots & \vdots &\vdots & \vdots & \vdots&\vdots \\
0 & 0 & 0 & \hdots & 0 & \zeta^{n-1} \\
\end{pmatrix}.
\]
$H_n$ is a non-abelian group of order $n^3$, with the centre $Z(H_n)=\{\zeta^{r} I_n : r \in \mb{Z} \}$.  For this reason we refer to diagrams $[\mb{C}/\Lambda \xrightarrow{} \mb{P}^{n-1}]$ constructed in this way as the Heisenberg invariant diagrams.

\end{remark}

We now recall the product expansion of the Weierstrass $\sigma$-function. Let $\mc{H}$ denote the upper halfplane. For $\tau \in \mc{H}$, let $\Lambda_{\tau}:=\mb{Z} \oplus \mb{Z}\tau$ be a lattice. For legibility we denote the diagram $B_{\Lambda_{\tau}}$ by $B_{\tau}$. Set $q:=e^{2\pi i \tau}$ and $u:=e^{2\pi i z}$. We then have (see Theorem 6.4 \cite{Sil2}):
\[
\sigma(z,\Lambda_\tau)=-\frac{1}{2\pi i}e^{\frac{1}{2}\eta(1)z^2}e^{-\pi i z}(1-u)\prod_{n \geq 1}\frac{(1-q^n u)(1-q^n u^{-1})}{(1-q^n)^2},
\]

\begin{proposition}
	The functions $x_r(z,\Lambda_{\tau})$ admit the product expansion:
	\begin{equation*} \label{x fun def}
	x_r(z,\Lambda_{\tau})=C \cdot (-1)^r q^{\frac{(2r-n)^2}{8n}} u^{-r}\prod_{m \geq 1}(1-q^{nm-r}u^n)\prod_{m \geq 0}(1-q^{nm+r}u^{-n})\prod_{m \geq 1}(1-q^{m})^{-2n}
	\end{equation*}
	where the factor $C:=e^{\frac{\eta(1)}{2}(nz^2-(n-1)z)}q^{\frac{n}{8}} u^{n/2}$ is independent of $r$.
\end{proposition}
\begin{proof}
A simple computation, using the definition of the $x_r$ and the Legendre relation.
\end{proof}
Rescaling the functions $x_0,\ldots,x_{n-1}$ by a nowhere vanishing holomorphic function of $q$ does not change the corresponding embedding $\mb{C}/\Lambda_{\tau} \xrightarrow{} \mb{P}^{n-1}$. 
Thus for each $r$ we multiply the above expression for $x_r$ by  $\frac{1}{C} \prod_{m \geq 1}(1-q^{m})^{2n} \prod_{m \geq 1}(1-q^{nm})$,  and define
\[
x_r(z,\tau):=(-1)^r q^{\frac{(2r-n)^2}{8n}} u^{-r}\prod_{m \geq 1}(1-q^{nm-r}u^n)\prod_{m \geq 0}(1-q^{nm+r}u^{-n})\prod_{m \geq 1}(1-q^{nm}).
\]
We furthermore define the functions $\alpha_0(\tau),\ldots,\alpha_{n-1}(\tau)$ by
\[
\alpha_r(\tau):=x_r(0,\tau)=(-1)^r q^{\frac{(2r-n)^2}{8n}}\prod_{m \geq 1}(1-q^{nm-r})\prod_{m \geq 0}(1-q^{nm+r})\prod_{m \geq 1}(1-q^{nm}).
\]
By Proposition \ref{factor_automorphy}, $\alpha_0(\tau)=0$, and $\alpha_{r}(\tau)=-\alpha_{n-r}(\tau)$. For $k>0$, using the fact that $\sigma(z,\tau)$ vanishes only at  the points $z \in \Lambda_{\tau}$, it is easy to see that $\alpha_k(\tau)$ is a nowhere vanishing function of $\tau$.	

Functions $\alpha_0,\ldots,\alpha_{n-1}$ satisfy a functional equation. This result is probably classical, but we could not find a convenient reference. We have learned about this proof from Tom Fisher.
\begin{proposition} \label{functional_equation} For $0 \leq r \leq n-1$ we have:
	\[
	\alpha_r(\tau+1)=\zeta^{\frac{(2r-n)^2}{8n}}\alpha_r(\tau),
	\]
	where we take $\zeta=e^{2i\pi/8}$, and 
	\[
	\alpha_r(-1/\tau)=(-1)^{(n-1)/2}\sqrt{\frac{i\tau}{n}}\sum_{s=0}^{n-1} \zeta	^{rs}\alpha_s(\tau).
	\]
\end{proposition}

\begin{proof}
The first equation is an immediate consequence of the product expansion of $\alpha_r$. To prove the second equation, we make use of the theta function $\Theta_1$:
\begin{align*}
\Theta_1(z,\tau)&=-i \sum_{m \in \mb{Z}} e^{(2m+1)iz}q^{(2m+1)^2/8}\\
&=2q^{1/8} \mathrm{sin} \ z \prod_{m \geq 1}  (1-q^m)(1-e^{2iz}q^{m})(1-e^{-2iz}q^{m}),
\end{align*}
where the equality of two expressions follows from the Jacobi triple product identity. Comparing the infinite products for $x_r$ and $\Theta_1$, we rewrite $\alpha_r$ in terms of $\Theta_1$:
\[
\alpha_r(\tau)=iq^{r^2/2n}\Theta_1(\pi r \tau,n \tau)
\]
The theta function satisfies the functional equation (page 476, \cite{whittaker2020course}):
\[
\sqrt{\frac{\tau}{i}}\Theta_1(z,\tau)=-e^{\frac{-i z^2}{\pi \tau}} \Theta_1(-\frac{z}{\tau},\frac{1}{\tau}).
\]
Substituting $z:=\pi r \tau$ and $\tau:=n\tau$, we obtain
\[
\sqrt{\frac{n\tau}{i}}\Theta_1(\pi r \tau,\tau)=iq^{-r^2/(2n)}\Theta_1 (\pi r/n,-1/(n\tau)).
\]
Thus we have
\[
\alpha_r(\tau)=(-1)^{r}\sqrt{\frac{i}{n \tau}} \Theta_1 (\pi r/n,-1/(n\tau)).
\]
Replacing $\tau$ by $-1/\tau$ we find
\begin{align*}
\alpha_r(-1/\tau)&=-(-1)^{r}\sqrt{\frac{\tau}{in}}\Theta_1(\pi r/n,-\tau/n)\\&=(-1)^{r}\sqrt{\frac{\tau}{in}}\Theta_1(\pi r/n,\tau/n).
\end{align*}
We then find, using the power series for $\Theta_1$, 
\begin{align*}
\Theta_1(\pi r/n,\tau/n)&=-i\sum_{m \in \mb{Z}}(-1)^m e^{(2m+1)i\pi r/n} q^{(2m+1)^2/8n}\\
&=(-1)^{(n-1)/2} i \sum_{m \in \mb{Z}}(-1)^{m}e^{(2m-n)i\pi r /n}q^{(2m-n)^2/8n}\\
&=(-1)^{(n-1)/2} i \sum_{m \in \mb{Z}}(-1)^{m}\zeta^{rm} q^{(2m-n)^2/8n}\\
&=(-1)^{(n-1)/2}i(-1)^{r} \sum_{s=1}^{n} \zeta^{rs} \alpha_{s}(\tau).
\end{align*}
and the functional equation now follows immediately. Note that the second equality was obtained by replacing $m$ with $m-(n-1)/2$.
\end{proof}

If $n \ge 5$, the point $(0:\alpha_1:\ldots:\alpha_{n-1})$, the image of $0 \in \mb{C}/\Lambda_{\tau}$ in $\PP^{n-1}$, determines the diagram $B_{\tau}$ uniquely, in the following sense.

\begin{lemma} \label{cusp_switch_lemma}
	\begin{enumerate}[label=(\roman*)]
		\item Let $H_1=[C_1 \subset \mb{P}^{n-1}]$ and $H_2=[C_2 \subset \mb{P}^{n-1}]$ be $n$-diagrams. Suppose $n>3$. If $|C_1 \cap C_2|>2n$, then $C_1=C_2$. The same conclusion holds if $n=3$ and $|C_1 \cap C_2|>9$.
		\item Consider the matrices $S=((-1)^{(n-1)/2}\sqrt{\frac{i}{n}}\zeta^{ij})^{n-1}_{i,j=0}$ and $T=(\delta_{ij}\zeta^{(2i-n)^2/8})^{n-1}_{i,j=0}$. We have $S \cdot B_{\tau}=B_{-1/\tau}$ and $T \cdot B_{\tau}=B_{\tau+1}$.
	\end{enumerate}
\end{lemma}
\begin{proof}
	\begin{enumerate}[label=(\roman*)]
		\item For $n>3$, both $C_1$ and $C_2$ are defined by quadrics. Let $f$ be a quadric vanishing on $C_1$, then $f$ vanishes on $C_1 \cap C_2$, and hence the zero set of $f$ meets $C_2$ in more than $2n$ points. By Bezout's theorem, it follows that $f$ vanishes on $C_2$, and hence $C_2 \subset C_1$. By the same argument, $C_1 \subset C_2$, and hence $C_1=C_2$. A similar argument applies in the case $n=3$.
		\item By Proposition \ref{functional_equation}, if we regard $S$ and $T$ as elements of $\mathrm{PGL}_n$,
		\[
		S \cdot (0:\alpha_1(\tau):\ldots:\alpha_{n-1}(\tau))=(0:\alpha_1(-1/\tau):\ldots:\alpha_{n-1}(-1/\tau)),
		\]
		\[
				T \cdot (0:\alpha_1(\tau):\ldots:\alpha_{n-1}(\tau))=(0:\alpha_1(\tau+1):\ldots:\alpha_{n-1}(\tau+1))
		\]
		
		Let $F_{\tau} \subset \mb{P}^{n-1}$ be the image of the set of $n$-torsion points of $\mb{C}/\Lambda_{\tau}$. We first show that $B_{-1/\tau}=\{S \cdot f \ | f \in F_{\tau}\}$ and $B_{\tau+1}=\{T \cdot f \ | f \in F_{\tau}\}$. By Remark \ref{heisenberg_group}, we see that $F_{\tau}=\{h \cdot (0:\alpha_1(\tau):\ldots:\alpha_{n-1}(\tau)) | \ h \in H_n \}$, and thus 
		\[
		\{S \cdot f \ | f \in F_{\tau}\}=\{S \cdot h \cdot (0:\alpha_1(\tau):\ldots:\alpha_{n-1}(\tau))\}
		\]
		\[
				=\{ShS^{-1} \cdot (0:\alpha_1(-1/\tau):\ldots:\alpha_{n-1}(-1/\tau) )\}
		\] 
		It is simple to check, using the description of $H_n$ in terms of generators $\sigma_{\omega_1/n}$ and $\sigma_{\omega_2/n}$, that $S\cdot H_n \cdot S^{-1}=H_n$, and hence the above set is equal to $F_{-1/\tau}$. The same argument applies for $T$ as well. Now note that $|B_{\tau}|=n^2$, and that $F_{-1/\tau}\subset S\cdot C_{\tau} \cap C_{-1/\tau}$ $F_{\tau+1}\subset T \cdot C_{\tau} \cap C_{\tau+1}$. For $n>3$, we are done by (i).
		
		For the case $n=3$, the image of $0$ is the point $(0:1:-1)$, independent of $\tau$, so we have to use a different method. One can generalize the functional equation proved in Proposition \ref{functional_equation} to a functional equation for the functions $x_r(z,\tau)$, of the form
		\[
				x_r(z,\tau+1)=k_{S}(z,\tau) \cdot \zeta^{\frac{(2r-n)^2}{8}}x_r(z,\tau),
		\]
		\[
	x_r(z,-1/\tau)=k_{T}(z,\tau)(-1)^{(n-1)/2}\sqrt{\frac{i\tau}{n}}\sum_{s=0}^{n-1} \zeta	^{rs}x_s(z\tau,\tau).
	\]
	where $k_S$ and $k_T$ are nowhere vanishing functions of $z$, independent of $r$, with $k_S(0,\tau)=1$. The proof is the same as before, with some added bookkeeping. These equations then imply (ii) directly, for all $n$.
	\end{enumerate}
\end{proof}
\subsection[Analytic description of the $\Omega$-matrix]{Analytic description of the $\Omega$-matrix, and proof of the formula for the Jacobian} \label{analytic description of Omega}
In this section we prove Theorem \ref{formula for Jacobian}. The first step is to prove an  analytic version of the theorem. Recall that, by Lemma \ref{unprojecting differential lemma} the differential on a genus one curve $C$ associated to an $\Omega$-matrix is regular, and that regular differentials on a complex torus $\mb{C}/\Lambda$ are scalar multiples of $dz$.
\begin{theorem}\label{analytic jacobian}
For any $\tau \in \mc{H}$, let $B_{\tau}=[\mb{C}/\Lambda_{\tau} \xrightarrow{\phi} \mb{P}^{n-1}]$ denote the $n$-diagram constructed in the previous section. Let $\Omega_{\tau}$ be the $\Omega$-matrix associated to this diagram, scaled so that if $\omega$ is the differential determined by $\Omega_{\tau}$, we have $\phi^{*} \omega = dz$. Let $c_4(\Omega_{\tau}),c_6(\Omega_{\tau})$ be the invariants of $\Omega_{\tau}$ defined in Section \ref{omega and curves section} and let $E_4(\tau)$ and $E_6(\tau)$ be the Eisenstein series. We then have
\[
c_4(\Omega_{\tau})=(2\pi)^4 E_4(\tau),
\]
and
\[
c_6(\Omega_{\tau})=(2\pi)^6 E_6(\tau).
\]	
\end{theorem}
We start by proving that $c_k(\Omega_{\tau})$ satisfy the functional equation.
\begin{lemma} \label{fe}
The functions $c_4(\Omega_{\tau})$ and $c_6(\Omega_{\tau})$ are weakly modular of weight $4$ and $6$, i.e. for $g=\begin{pmatrix} a & b \\ c & d  \end{pmatrix} \in  \mathrm{SL}_2(\mb{Z})$, we have
\[
c_k(\Omega_{\frac{a\tau+b}{c\tau+d}})=(c\tau+d)^{k} c_k(\Omega_{\tau}).
\]
\end{lemma}
\begin{proof}
Write $\tau'=\frac{a\tau+b}{c\tau+d}$. Over an algebraically closed field, any two $n$-diagrams with the same Jacobian curve are isomorphic, and hence there exists a commutative diagram
					\[ \begin{tikzcd}
	\mb{C}/\Lambda_{\tau} \arrow{r}{\phi_{\tau}} \arrow[swap]{d}{\psi} & \mb{P}^{n-1} \arrow{d}{\rho} \\%
	\mb{C}/\Lambda_{\tau'} \arrow{r}{\phi_{\tau'}}& \mb{P}^{n-1}
	\end{tikzcd}
	\]
The left vertical arrow is $\psi: z \mapsto \frac{z}{c\tau+d}$, and the right vertical arrow is an automorphism $\rho$ of $\mb{P}^{n-1}$.Let $\Omega'$ be an $\Omega$-matrix for $\phi_{\tau'}$, scaled so that for the associated differential $\omega'$ we have $\phi_{\tau'}^{*} \omega'=(c\tau+d)dz$. We claim that $\rho^{*} (\omega')=\omega$.

Indeed, note that as $\psi^* ((c\tau+d)dz)=dz$, we have $(\phi_{\tau'}\circ \psi)^{*}(\omega')=dz$. As the above diagram commutes, $dz=(\phi_{\tau'}\circ \psi)^{*}(\omega')=(\phi_{\tau} \circ \rho )^{*}(\omega')$, i.e. $\rho^{*}(\omega')=\omega$.

 We may assume that $\rho$ is represented by a matrix $N \in \mathrm{GL}_n(\mb{C})$.  Now it follows, from our scaling of $\Omega'$, Proposition \ref{Omega quad change of coordinates} and Lemma \ref{invariant differential change of coords}, that we have $\Omega_{\tau}=N \star \Omega'$. Hence, by Lemma \ref{omega invariants}, we have $c_k(\Omega_{\tau})=c_k(\Omega')$. 
 
 By the definition of $\Omega'$, we have $(c\tau+d)\Omega'=\Omega_{\tau'}$. As $c_4(\Omega)$ and $c_6(\Omega)$ are homogeneous polynomials of degree 4 and 6 in the coefficients of $\Omega$, we find that $c_k(\Omega_{\frac{a\tau+b}{c\tau+d}})=(c\tau+d)^{k} c_k(\Omega_{\tau})$, as desired.
\end{proof}

Now it remains to prove that $ c_k(\Omega_{\tau})$ is holomorphic at the cusp, which forces it to be a scalar multiple of $E_k$, and then to check that this scalar is 1. To do this we need an explicit description of $\Omega_{\tau}$. Recall that $\Omega_{\tau}$ is an alternating matrix of homogenous quadratic forms.   

Throughout the Section, to avoid confusion with the functions $x_r(z,\tau)$, we label the standard coordinates on $\mb{P}^{n-1}$ by $X_0,\ldots,X_{n-1}$
\begin{proposition} \label{analytic coefficients}
There exist functions $u_i : \mc{H} \xrightarrow{} \mb{C}$, for $0\leq i \leq n-1$, such that for $0\leq,k,l \leq n-1$, we have 
\[
(\Omega_{\tau})_{k,l}=u_{l-k}(\tau)X_k X_{l-k}-\frac{\frac{\partial x_0}{\partial z}(0,\tau)}{n-2} \cdot \left( \sum_{1\leq i \leq n-1, i \neq k-l}  \frac{\alpha_{l-k}(\tau)}{\alpha_i(\tau) \alpha_{i+l-k}(\tau)} X_{i+k}X_{i+l-k}\right) ,
\]
where the subscripts are read mod $n$.
\end{proposition}

We divide the proof into a sequence of lemmas.

\begin{lemma} \label{Heisenberg invariance omega}
For $0 \leq k,l \leq n-1$, we have
\begin{enumerate}[label=(\roman*)]
	\item $(\Omega_{\tau})_{k,l}(X_0,\zeta X_1\ldots,\zeta^{n-1}X_{n-1})=\zeta^{k+l}(\Omega_{\tau})_{k,l}(X_0,X_1,\ldots,X_{n-1})$
	\item $(\Omega_{\tau})_{k,l}(X_1,X_2,\ldots,X_{n-1},X_{0})=(\Omega_{\tau})_{k+1,l+1}(X_0,X_1,\ldots,X_{n-1})$
\end{enumerate}

\end{lemma}
\begin{proof}
The automorphisms of $\mb{P}^{n-1}$ defined in Remark \ref{heisenberg_group}, $\sigma_{1/n}$ and $\sigma_{\tau/n}$, take the curve $C_{\tau}$ to itself. Since $\mathrm{det}(\sigma_{1/n})=\mathrm{det}(\sigma_{\tau/n})=1$, by Lemma \ref{invariant differential change of coords}, $\sigma_{1/n} \star \Omega_{\tau}$ and $\sigma_{\tau/n} \star \Omega_{\tau}$ all have the same associated differential on $C_{\tau}$. Thus we must have $\Omega_{\tau}=\sigma_{1/n} \star \Omega_{\tau}$ and $\Omega_{\tau}=\sigma_{\tau/n} \star \Omega_{\tau}$, and the lemma follows.
\end{proof}

To compute $q$-expansions of the coefficients of $\Omega_{\tau}$, we  intersect the curve $C_{\tau}$ with the hyperplane $X_0=0$. 
\begin{lemma}
The intersection of the curve $C_{\tau}$ and the hyperplane $X_0=0$ is the set 
\[ Z_{\tau}:=\{(0:\zeta^{i}\alpha_1(\tau):\zeta^{2i}\alpha_2(\tau):\ldots:\zeta^{i(n-1)}\alpha_{n-1}) : i \in \{0,1,\ldots,n-1\}\}\]
\end{lemma}
\begin{proof}
First note that, for $0 \leq i \leq n-1$,  
\begin{align*}
x_r(i/n,\tau)&=(-1)^r q^{\frac{(2r-n)^2}{8n}} \zeta^{-i}\prod_{m \geq 1}(1-q^{nm-r}\zeta^n)\prod_{m \geq 0}(1-q^{nm+r}\zeta^{-n})\prod_{m \geq 1}(1-q^{nm})\\
&=\zeta^{-i}(-1)^r q^{\frac{(2r-n)^2}{8n}}\prod_{m \geq 1}(1-q^{nm-r})\prod_{m \geq 0}(1-q^{nm+r})\prod_{m \geq 1}(1-q^{nm})=\zeta^{-i} \alpha_r(\tau)
\end{align*}
and hence $\phi_{\tau}(i/n)=(0:\zeta^{-i}\alpha_1(\tau):\zeta^{-2i}\alpha_2(\tau):\ldots:\zeta^{-i(n-1)}\alpha_{n-1}) \in C_{\tau} \cap \{X_0=0\}$. As $C_{\tau}$ is a curve of degree $n$, $ C_{\tau} \cap \{X_0=0\}$ consists of $n$ points, hence the conclusion.
\end{proof}

\begin{lemma}\label{analytic coefficients lemma}
Let $F^{0}_{\bu}$ be a resolution model of the set $Z_{\tau} \subset \mb{P}^{n-2}$, and let $\{\Omega_1,\ldots,\Omega_{n-1}\}$ be the set of $\Omega$-quadrics associated to $F^{0}_{\bu}$. For $1 \leq k \leq n-1$, we have
\[
\Omega(X_1,X_2,\ldots,X_{n-1})=(\Omega_{\tau})_{0k}(0,X_1,X_2,\ldots,X_{n-1})
\]
and furthermore, we have
\[
(\Omega_{\tau})_{0k}(0,X_1,\ldots,X_{n-1})=\lambda \cdot \sum_{1\leq i \leq n-1, i \neq n-k}  \frac{\alpha_{k}(\tau)}{\alpha_i(\tau) \alpha_{k-i}(\tau)} X_{i}X_{k-i}.
\]
where $\lambda=-\frac{\frac{\partial x_0}{\partial z}(0,\tau)}{n-2}$.
\end{lemma}

The proof is based on a calculation of $\Omega$-quadrics for a set of $n$ points carried out in \cite{radfis}.

\begin{lemma} \label{analytic set pts}
	Let $n\geq4$ be an integer. Let $a_1,a_2,...,a_{n-1} \in \mathbb{C}^{\times }$ be non-zero complex numbers, and let $\zeta \in \mb{C}$ be a primitive $n$th root of unity. Let $Y=\{Q_1,\ldots,Q_n\} \subset \mathbb{P}^{n-2}$ be a set of $n$ points in general position, with coordinates given by $Q_i=(\zeta^i a_1:\zeta^{2i} a_2:...:\zeta^{(n-1)i}a_{n-1})$. Then the quadrics $\Omega_k$ associated to a resolution model of $X$ are given, up to a scalar, by
	\[
	\sum_{i=1, i \ne k}^{n-1} \frac{a_k}{a_i a_{n+k-i}}x_i x_{n+k-i},
	\]
	where the subscripts are read modulo $n$.
\end{lemma}
\begin{proof}
    Define a set $X$ of $n$ points in general position in $\PP^{n-2}$ by $X=\{P_1,P_2,\ldots,P_n)$ where $P_1=(1:0\ldots:0),P_2=(0:1:0\ldots:0),\ldots,P_{n-1}=(0:0:\ldots:0:1)$ and $P_{n}=(1:1:\ldots:1)$
 By Lemma 5.3 of \cite{radfis}, the $\Omega-$quadrics associated to $X$ are given by $\Omega_i=nx_i^2-\sum_{i=1}^{n-1}x_i^2$.
 
 	Define $g \in \mathrm{GL}_{n-1}(\mb{C})$ by $g_{ij}=a_{j}\zeta_{ij}$. Then $g$ takes the set $X$ to $Y$. Let $x_j'=\sum_{i=1}^{n-1} g_{ij}x_i$. By Proposition \ref{Omega quad change of coordinates}, it suffices to check that 
	\[
	\Omega_{Y,k}(x_1',x_2',...,x_{n-1}')=\sum_{r=1}^{n-1} g_{rk} \Omega_{X,r}(x_1,x_2,...,x_{n-1}),
	\]
	holds for all $k$. We first compute the right hand side:
	\begin{align*}
		\sum_{r=1}^{n-1} g_{rk} \Omega_{X,r}&=
		\sum_{r=1}^{n-1} a_k \zeta^{rk}(n x_r^2-2x_r\sum_{p=1}^{n-1}x_p)\\
		&=\sum_{r=1}^{n-1} (n-2)a_k \zeta^{rk} x_r^2- \sum_{1\leq p<r \leq n-1}2(\zeta^{rk}+\zeta^{pk})a_k x_r x_p  .
	\end{align*}
	The left hand side is given by
	\begin{align*}
		\sum_{i=1, i \ne k}^{n-1} \frac{a_k}{a_i a_{n+k-i}}x_i' x_{n+k-i}',
		=-\sum_{i=1, i \ne k}^{n-1} \frac{a_k}{a_i a_{i-k}}(\sum_{p=1}^{n-1} a_i \zeta^{pi} x_{p}  x_{n+k-i}) (\sum_{r=1}^{n-1} a_{n+k-i} \zeta^{(n+k-i)r} x_r)
	\end{align*}
	The coefficient of $x_r^2$ is $(n-2)a_k \zeta^{rk}$ and the coefficient of $x_r x_p$, for $r$ and $p$ distinct, is given by
	\begin{align*}
		a_{k}(\sum_{i=1, i \ne k}^{n-1} \zeta^{pi+r(n+k-i)}+ \zeta^{ri+p(n+k-i)})
		&=a_k\left(  (-1-\zeta^{k(p-r)})\zeta^{rk}+(-1-\zeta^{k(r-p)})\zeta^{pk}\right) 
		\\&=2(\zeta^{rk}+\zeta^{pk})a_k,
	\end{align*}
	as desired.
\end{proof}

\begin{proof}[Proof of Lemma \ref{analytic coefficients lemma}]
The first part of the lemma follows from Lemma \ref{hyperplane slice}(ii). The existence of the constant $\lambda$ follows from Lemma \ref{analytic set pts}, where we computed the $\Omega$-quadrics for the set $Z_{\tau}$. To compute $\lambda$ exactly, we need to take the scaling of the matrix $\Omega_{\tau}$ into account.

The differential $\omega$ associated to $\Omega_{\tau}$ is given by
\[
\omega=(n-2)\frac{X_0^2 d(X_k/X_0) }{(\Omega_{\tau})_{0,k}}
\]
for any $1 \leq k \leq n-1$. Hence, we have
\begin{equation} \label{eq:dif}
dz=\phi^{*}\omega=(n-2)\frac{x_0(z,\tau)\frac{\partial x_k}{\partial z}(z,\tau)-x_k(z,\tau)\frac{\partial x_0}{\partial z}(z,\tau) }{(\Omega_{\tau})_{0,k}(x_0(z,\tau),\ldots,x_{n-1}(z,\tau) )} \cdot dz .
\end{equation}
The expression in front of $dz$ must be identically equal to 1. Setting $z=0$, and noting that $x_i(0,\tau)=\alpha_i(\tau)$, as well as $x_0(0,\tau)=0$, we find
\[
x_0(0,\tau)\frac{\partial x_k}{\partial z}(z,\tau)-x_k(z,\tau)\frac{\partial x_0}{\partial z}(0,\tau)=-\alpha_k(\tau)\frac{\partial x_0}{\partial z}(0,\tau),
\]
and 
\begin{align*}
(\Omega_{\tau})_{0,k}(x_0(0,\tau),\ldots,x_{n-1}(0,\tau) )&=(\Omega_{\tau})_{0,k}(0,\alpha_1(\tau)\ldots,\alpha_{n-1}(\tau) )\\&
=\lambda \cdot \sum_{1\leq i \leq n-1, i \neq n-k}  \frac{\alpha_{k}(\tau)}{\alpha_i(\tau) \alpha_{k-i}(\tau)} \alpha_{i}(\tau)\alpha_{k-i}(\tau)\\&
=\lambda(n-2)\alpha_k(\tau).
\end{align*}
The functions $\alpha_{k}(\tau)$ are nowhere vanishing for $k>0$, and so we may divide them out and find that $\lambda=-\frac{\partial x_0}{\partial z}(0,\tau)$.
\end{proof}
We can now finish the proof of Proposition \ref{analytic coefficients}. By Lemma \ref{Heisenberg invariance omega}(ii), it suffices to prove the identity for the entries $(\Omega_{\tau})_{0,k}$, and by Lemma \ref{Heisenberg invariance omega}(i), the monomials occurring in $(\Omega_{\tau})_{0,k}$ are $X_iX_{k-i}$, as $i$ ranges over $\mb{Z}/N\mb{Z}$. Lemma \ref{analytic coefficients lemma} determines the coefficients of all of these monomials, except for the coefficient in front of $X_0X_k$, which we define to be $u_{k}(\tau)$, concluding the proof.

To prove Theorem 2.1, we need to compute the beginning of the $q$-expansion of the coefficients of $(\Omega_{\tau})_{k,l}$.  

\begin{lemma} \label{q-expansions and constant term lemma}
Coefficients of $(\Omega_{\tau})_{0,k}$ are holomorphic functions of $q^{1/(8n)}$. In other words, they can be expressed as power series in the variable $q^{1/(8n)}$. Furthermore the only coefficient of $(\Omega_{\tau})_{0,k}$ with a non-zero constant term is $u_k$, with the constant term equal to $-2\pi i (n-2k)$.
\end{lemma}
\begin{proof}
By Proposition \ref{analytic coefficients}, for $1\leq i \leq n-1$, $i \neq k$, the coefficient of $X_iX_{k-i}$  is
\[-\frac{\partial x_0}{\partial z}(0,\tau)\frac{\alpha_{k}(\tau)}{\alpha_i(\tau) \alpha_{k-i}(\tau)}.\]
Differentiating and applying the product rule to the infinite product
\begin{equation*} 
x_r(z,\tau)=(-1)^r q^{\frac{(2r-n)^2}{8n}} u^{-r}\prod_{m \geq 1}(1-q^{nm-r}u^n)\prod_{m \geq 0}(1-q^{nm+r}u^{-n})\prod_{m \geq 1}(1-q^{nm}). 
\end{equation*}
we obtain
\begin{equation} \label{equation dif}
\frac{\partial x_r}{\partial z}(z,\tau)=2\pi i x_r\left( -r + n \sum_{m \geq 0} \frac{q^{nm+r}u^{-n}}{1-q^{nm+r}u^{-n}}-n \sum_{m \geq 1} \frac{q^{nm-r}u^{n}}{1-q^{nm-r}u^{n}}\right) 	,
\end{equation}
and so in particular
\[
\frac{\partial x_0}{\partial z}(u,\tau)=2\pi i x_0\left( - n \sum_{m \geq 1} \frac{q^{nm}u^{-n}}{1-q^{nm}u^{-n}}-n \sum_{m \geq 1} \frac{q^{nm}u^{n}}{1-q^{nm}u^{n}}\right) +\frac{ 2 \pi n x_0}{1-u^{-n}}.
\]
Since $x_0(0,\tau)=0$, the first summand vanishes at $u=0$. For the second one, we have
\[
\frac{ 2 \pi n x_0}{1-u^{-n}}=2\pi n \cdot q^{\frac{n}{8}} \prod_{m \geq 1}(1-q^{nm}u^n)(1-q^{nm}u^{-n})(1-q^{nm}),
\]
and hence we find
\[
\frac{\partial x_0}{\partial z}(0,\tau)=2\pi n \cdot q^{\frac{n}{8}} \prod_{m \geq 1}(1-q^{nm})^3.
\]
From these formulas it is clear that the coefficient of $X_iX_{k-i}$ is a meromorphic function of $q^{1/(8n)}$. We compute, reminding the reader that subscripts are read mod $n$, that the valuation of the coefficient is equal to \[\frac{(2k-n)^2+n^2-(2i-n)^2-(n-2i+2k)^2}{8n}=\frac{(2i-2k)(4n-4i)}{8n}>0,\]
 in the case that $i>k$, and equal to
 
  \[\frac{(2k-n)^2+n^2-(2i-n)^2-(2k-2i-n))^2}{8n}=\frac{(2k-2i)4i}{8n}>0\]
   otherwise. Hence these coefficients are holomorphic functions of $q^{1/(8n)}$, with zero constant term. 

We now deal with the coefficient $u_k$. The equation \eqref{eq:dif} can be restated as
\begin{align*}
(n-2)(x_0\frac{\partial x_k}{\partial z}-x_k\frac{\partial x_0}{\partial z})=u_{k}(\tau)x_0 x_{k}-\frac{\partial x_0}{\partial z}(0,\tau) \sum_{1\leq i \leq n-1, i \neq n-k}  \frac{\alpha_{k}(\tau)}{\alpha_i(\tau) \alpha_{k-i}(\tau)} x_{i}x_{k-i},
\end{align*}
or
\[
u_{k}(\tau)x_0 x_{k}=(n-2)(x_0\frac{\partial x_k}{\partial z}-x_k\frac{\partial x_0}{\partial z})+\frac{\partial x_0}{\partial z}(0,\tau) \sum_{1\leq i \leq n-1, i \neq n-k}  \frac{\alpha_{k}(\tau)}{\alpha_i(\tau) \alpha_{k-i}(\tau)} x_{i}x_{k-i}
\]
Using this formula, it is easy to see that $u_k(\tau+n)=u_k(\tau)$, and that $u_k$ can be expressed as a Laurent series in $q^{1/n}$. Regard both sides of the above equation as power series in $q^{1/(8n)}$ with coefficients power series in $u$. The smallest power of $q$ apperaing in $x_0x_k$ is $q^{((2k-n)^2+n^2)/(8n))}$, with the coefficient equal to
\begin{equation*}
(-1)^k(1-u^{-n})u^{-k}.
\end{equation*}
On the other hand, using the formula (\ref{equation dif}), we can compute that the smallest power of $q$ appearing on the right is $q^{((2k-n)^2+n^2)/8n)}$, with the coefficient equal to
 \begin{equation*}
(-1)^{k+1} 2\pi i (n(n-2) u^{-k-n}k(1-u^{-n})u^{-k}-n((k-1)u^{-k}+(n-1-k)u^{-k-n})).
\end{equation*}
\[
=(-1)^{k+1} 2\pi i (n-2k) (1-u^{-n})u^{-k}
\]
There are no negative powers of $q$ appearing in the $q$-expansion of $u_k$, and comparing the above expressions we find that the constant coefficient is equal to $-2\pi i (n-2k)$, concluding the proof of the lemma.
\end{proof}
To finish the proof of Theorem \ref{analytic jacobian} we need Lemma 8.4 of \cite{jacobians}:
\begin{lemma}
	The alternating matrix of quadratic forms 
	\[
	\Omega^{const}=\begin{pmatrix} 0& (n-2)X_0X_1& (n-4)X_0 X_2&(n-6) X_0X_3 &\hdots &(2-n)X_0X_{n-1}\\
	&0&(n-2)X_1X_2&(n-4)X_1 X_3 & \hdots &(4-n)X_1X_{n-1}\\
	&&0&(n-2)X_2X_3&\hdots&(6-n)X_2X_{n-1}\\
	&-&&&&\vdots\\
	&&&&\ddots&(n-2)X_{n-2}X_{n-1}\\
	&&&&&0
	\end{pmatrix}	
	\]
	has invariants $c_4(\Omega^{const})=1$ and $c_6(\Omega^{const})=-1$.	
\end{lemma}
The proof is a computation, and can be found on pages 20 and 21 of \cite{jacobians}. Note that our normalisation of invariants $c_k(\Omega)$ differs from the one used in \cite{jacobians} by a factor of $(n-2)^k$.
\begin{proof}[Proof of Theorem \ref{analytic jacobian}]
By Lemma \ref{fe}, $c_k(\Omega_{\tau})$ are invariant under $\tau \mapsto \tau+1$, and so are holomorphic functions of $q$. By Lemma \ref{fe} again, they are modular forms of weight $k$. The space of modular forms of weight $k$ for $\mathrm{SL}_2(\mb{Z})$ is 1-dimensional, and spanned by the Eisenstein series $E_k$, for $k=4$ and $k=6$.

Write $\Omega_{\tau}=-2\pi i\Omega^{const} + \Omega'$, where the coefficients of the entries of $\Omega'$ have no constant term, viewed as power series of $q^{1/(8n)}$. As $c_k(\Omega)$ are homogenous of degree $k$, we have
 \[
 c_k(\Omega_{\tau})=(2\pi)^k(1 +o(q))
 \]
 and therefore we conclude  $ c_k(\Omega_{\tau})=(2\pi)^k E_k(\tau)$.
\end{proof}

Theorem \ref{analytic jacobian} is related to the formula for the Jacobian given in Theorem \ref{formula for Jacobian} via the following standard result on uniformization of elliptic curves. As stated, this is Lemma 3.5 of \cite{ncovbds}.
\begin{lemma} \label{uniformization}
	Let $E$ be an elliptic curve over $\mb{C}$ with Weierstrass equation $W$
	\begin{equation*} \label{w eq}
		y^2 + a_1xy + a_3y = x^3 + a_2x^2 + a_4x + a_6
	\end{equation*}
	Let $\Lambda$ be the period lattice obtained by integrating the differential $dx/(2y + a_1x + a_3)$. Then the map
	\[
	\rho : E(\mb{C}) \xrightarrow{} \mb{C}/\Lambda :(x,y) \mapsto \int_{\gamma} \frac{dx}{2y + a_1x + a_3} 
	\]
	where $\gamma$ is an arbitrary path from $0_E$ to $(x,y)$, is a $\mb{C}$-analytic isomorphism, with $\rho^{*}(dz)= \frac{dx}{2y + a_1x + a_3}$. Furthermore, choose a basis $\omega_1,\omega_2$ for $\Lambda$ so that $\tau = \omega_2/\omega_1 \in \mc{H}$. Then the invariants $c_4$ and $c_6$ of the Weierstrass equation $W$ are given by $c_k = (\frac{2\pi}{\omega_1})^k E_k(\tau)$.
\end{lemma} 

\begin{proof}[Proof of Theorem \ref{formula for Jacobian}]
Note that it suffices to prove the theorem with the ground field replaced by $\mb{C}$. Thus let $[C \xrightarrow{} \mb{P}^{n-1}]$ be an $n$-diagram defined over $\mb{C}$, let $\Omega$ be an $\Omega$-matrix for  $[C \xrightarrow{} \mb{P}^{n-1}]$, with the associated invariant differential $\omega$. We identify $C$ with its Jacobian $E$. Fix a Weierstrass equation for $E$ such that $\omega$ is identified with $dx/(2y + a_1x + a_3)$.

Let $\Lambda=\mb{Z}\omega_1\oplus \mb{Z}\omega_2$ be as in the above lemma, and let $\Lambda_{\tau}=\frac{1}{\omega_1}\Lambda$. Let $\psi: \mb{C}/\Lambda_{\tau} \xrightarrow{} \mb{C}/\Lambda$ be the isomorhism $z \mapsto \omega_1 z$. By composing with $\rho^{-1}\circ\psi$, we obtain an $n$-diagram $[\mb{C}/\Lambda_{\tau} \xrightarrow{\phi} \mb{P}^{n-1}]$, with $\phi^{*} \omega=\omega_1 dz$. As $\mb{C}$ is algebraically closed, this diagram is isomorphic to the Heisenberg invariant diagram, and thus by Theorem \ref{analytic jacobian} and the above lemma, we conclude $c_k(\Omega)=(\frac{2\pi}{\omega_1})^k E_k(\tau)=c_k$, as required.
\end{proof}
\subsection{Analytic theory over the real numbers}\label{analytic over real}
We now collect some standard results about elliptic curves over $\mb{R}$ that we need in Section \ref{explicit bounds chapter}. 
\begin{proposition}\label{real_uniformization}
Let $E/\mb{R}$ be an elliptic curve, defined by a Weierstrass equation $W$. Let $\Lambda$ be the period lattice obtained by integrating $\omega=dx/(2y + a_1x + a_3)$. Then
\begin{enumerate}[label=(\roman*)]
	\item There exists a basis $\omega_1,\omega_2$ of $\Lambda$ with $\tau=\omega_2/\omega_1 \in \mc{H}$ and $\mathrm{Re}(\tau) \in \{0,1/2\}$. If the discriminant of $E$ is positive, then $\mathrm{Re}(\tau)=0$, otherwise $\mathrm{Re}(\tau)=1/2$.
	\item For such a $\tau$, we have $q=e^{2i\pi\tau} \in \mb{R}$, with $|q|<1$. Furthermore, the unifomization $\mb{C}$-analytic isomorphism $\mb{C}/\Lambda_{\tau}\cong\mb{C}^{*}/q^{\mb{Z}}\xrightarrow{} E(\mb{C})$ commutes with complex conjugation, and hence restricts to an $\mb{R}$-analytic isomorphism $\mb{R}^{*}/q^{\mb{Z}} \xrightarrow{} E(\mb{R})$.
\end{enumerate}
\end{proposition}
This is a standard fact, e.g. see Proposition VI.2.3 of \cite{Sil2}. As a consequence, note that as $q \in \mb{R}$, and the functions $x_i(z,q)$ were given in terms of $q$-expansions, it is easy to show that the Heisenberg invariant embedding $\mb{C}/\Lambda_{\tau} \xrightarrow{} \mb{P}^{n-1}$ is commutes with complex conjugation and is hence defined over $\mb{R}$.

\begin{lemma} \label{archimedean place trivial lemma}
	Let $E/\mb{R}$ be an elliptic curve, and let $n$ be an odd integer. Then the Galois cohomology group $H^1(\mb{R},E[n])$ is trivial.
\end{lemma}

\begin{proof}
As $|\mathrm{Gal}(\mb{C}/\mb{R})|=2$ is coprime to $|E[n](\mb{C})|=n^2$, this is standard, a consequence of the fact that for any finite group $G$ and a $G$-module $A$, the group $H^1(G,A)$ is annihilated by multiplication by $|G|$, see  Chapter VII, Proposition 6 of \cite{serre2013local}.	
\end{proof}	

By the lemma, any two $n$-diagrams defined over  $\mb{R}$ with the same Jacobian curve are isomorphic.

\begin{lemma} \label{real heisenberg matrix}
Every $n$-diagram $[C \xrightarrow{} \mb{P}^{n-1}]$ , defined over $\mb{R}$, is $\mathrm{PGL}_n(\mb{R})$-equivalent to a Heisenberg invariant diagram $B_{\tau}$, for some $\tau \in \mc{H}$ with $\mathrm{Re}(\tau) \in \{0,n/2\}$. Let $E$ be the Jacobian of $C$. If the discriminant of $E$ is positive, then we have $\mathrm{Re}(\tau)=0$, and if it is negative, $\mathrm{Re}(\tau)=n/2$
\end{lemma}
\begin{proof}
By Lemma \ref{real_uniformization}, we can identify $E$ with the complex torus $\mb{C}^{\star}/q^{\mb{Z}}$, where $q=e^{2i\pi\tau}$ and $\mathrm{Re}(\tau)\in \{0,1/2\}$. We are free to translate $\tau$ by any integer in $\mb{Z}$, and so as $n$ is odd, if $\mathrm{Re}(\tau)=1/2$,  we replace $\tau$ by $\tau+(n-1)/2$.  Hence we may assume that $\mathrm{Re}(\tau)\in \{0,n/2\}$. By Lemma \ref{archimedean place trivial lemma}, it suffices to show that the diagram $B_{\tau}$ is defined over $\mb{R}$.
	
Recall the series that for the functions $x_r(z,q)$ that define the embedding $B_{\tau}$
\[
x_r(z,q)=(-1)^r q^{\frac{(2r-n)^2}{8n}} u^{-r}\prod_{m \geq 1}(1-q^{nm-r}u^n)\prod_{m \geq 0}(1-q^{nm+r}u^{-n})\prod_{m \geq 1}(1-q^{m})^{-2n}.
\]
Note that $q^{1/n}=e^{2i\pi\tau/n} \in \mb{R}$ precisely when $\mathrm{Re}(\tau)\in \{0,n/2\}$. For such $\tau$, all powers of $q$ in the above expression are real, except for the factor $q^{{(2r-n)^2}/{8n}}=q^{n/8}\cdot q^{{(r^2-rn)}/{(2n)}}$. As $n$ is odd, $2|r^2-rn$ for all $r$, and hence $q^{{(r^2-rn)}/{(2n)}} \in \mb{R}$ for all $r$. The non-real factor $q^{n/8}$ is independent of $r$. Thus the map $u \mapsto (x_0(u,q):\ldots:x_{n-1}(u,q))$ is compatible with complex conjugation, and hence the diagram $B_{\tau}=[\mb{C}^{*}/q^{\mb{Z}} \xrightarrow{} \mb{P}^{n-1}]$ is defined over $\mb{R}$.

\end{proof}

\section{Bounding the capitulation discriminant} \label{explicit bounds chapter}
In this section we prove Theorem \ref{Main theorem I} and Theorem \ref{Main theorem II}. We start by proving an inequality for elliptic curves defined over $\mb{R}$.

Let $n \geq 3$ be an odd integer and let $E/\mb{R}$ an elliptic curve, with a  Weierstrass equation $W$ 
\[
y^2+a_1xy+a_3y=x^3+a_2x^2+a_4x+a_6.
\]

Let $\omega=dx/(2y + a_1x + a_3)$ be the invariant differential associated to the equation $W$. Let $\Lambda \subset \mb{C}$ be the period lattice obtained by integrating $\omega$, and let $\phi : \mb{C}/\Lambda \xrightarrow{} E$ be the corresponding complex uniformization. Consider the  $n$-diagram $B_{\Lambda}=[E\cong \mb{C}/\Lambda \xrightarrow{} \mb{P}^{n-1}]$ defined in Proposition \ref{line_bundle heisenberg diagram }. Let $\Omega_{\Lambda}$ be the $\Omega$-matrix for $B_{\Lambda}$ scaled so that for the corresponding differential $\omega_{\Lambda}$ we have $\phi^{*} \omega_{\Lambda}=\omega$, and let $D_{\Lambda}$ be the discriminant form for the diagram $B_{\Lambda}$ associated to this scaling, as explained in Section \ref{discriminant form section}.

The following theorem, together with the minimization theorem and Minkowski's theorem, implies Theorem \ref{Main theorem I}. Set $H_W=\max(|c_4|^{1/4},|c_6|^{1/6})$.
\begin{theorem} \label{compact_inequality}
	Let $K$ be a compact subset of $\mb{R}^n$. There exists a constant $c(n,K)$, depending only on $n$ and $K$, such that, for all elliptic curves $E/\mb{R}$, with a Weierstrass equation $W$, we have
	\[
	\max_{(u_0,\ldots,u_{n-1}) \in K}|D_\Lambda(u_0,\ldots,u_{n-1})| \leq c(n,K) H_W^{2n-2}
	\]
	
\end{theorem}

\begin{proof}
 By the Proposition $\ref{real_uniformization}$, we may write $\Lambda=\omega_1 \cdot (\mb{Z} \oplus \mb{Z}\tau)$, where $\mathrm{Re}(\tau) \in \{0,1/2\}$. By the Lemma \ref{uniformization}, we have $c_k(W)=(\frac{2\pi}{\omega_1})^k E_k(\tau)$, and hence
 
	  \[H_W^{2n-2}= (\frac{\omega_1}{2\pi})^{2n-2}\max(|E_4(\tau)^{1/4}|,|E_6(\tau)^{1/6}|)^{2n-2}=(\frac{\omega_1}{2\pi})^{2n-2} H(\tau)^{2n-2},
  \] 
  where $H(\tau)=\max(|E_4(\tau)|^{1/4},|E_6(\tau)^{1/6}|)$.
 	By comparing the differentials attached to $\Omega_{\Lambda}$ and $\Omega_{\tau}$, we see that $\Omega_{\Lambda}=\omega_1 \Omega_{\tau}$. As the coefficients of $D_{\Lambda}$ are homogenous polynomials of degree $2n-2$ in the entries of $\Omega_{\Lambda}$ we  see that $D_{\Lambda}(u_0,\ldots,u_{n-1})=\omega_1^{2n-2} D_{\tau}(u_0,\ldots,u_{n-1})$, where the coefficients of $D_{\tau}$ are power series in the variable $q^{1/n}$. Note also that, if we let $B_{\tau}$ denote the Heisenberg invariant diagram $B_{\langle1,\tau \rangle}$, the form $D_{\tau}$ is then the discriminant form associated to $B_{\tau}$.
 	
	We introduce the normalization $\mc{D}_{\tau}=\frac{D_{\tau}(u_0,\ldots,u_{n-1})}{H(\tau)^{2n-2}}$. We seek to prove that the expression
	 \[
	 \max_{(u_0,\ldots,u_{n-1}) \in K} \mc{D}_{\tau} (u_0,\ldots,u_{n-1})
	 \] 
	 is bounded as a function of $\tau$. As $\mc{D}_{\tau}$ is a polynomial of degree $2n$ and $K$ is compact, we can find a constant $M$ that is an upper bound on $K$ for each of the finitely many monomials that appear in $\mc{D}_{\tau}$.
	
	By Lemma \ref{real heisenberg matrix}, as $E$ is an elliptic curve defined over $\mb{R}$, we may restrict to $\tau$ with $\mathrm{Re}(\tau) \in \{0,n/2\}$. In other words, it suffices to show that the coefficients of $\mc{D}_{\tau}$ are bounded, as functions of $\tau$ on  the set $A:=\{\tau \in \mc{H} : \textrm{Re}(\tau) \in \{0,n/2\} \}$. Let us choose two arbitrary constants $C_1>C_2>0$. We split the set $A$ as $A=A_1 \cup A_2 \cup A_3$, where $A_1= \{\tau \in A : \mathrm{Im}(\tau) > C_1 \},A_2= \{\tau \in A : C_2 <  \mathrm{Im}(\tau) \le  C_1 \}$ and $A_3= \{\tau \in A : \mathrm{Im}(\tau) \le  C_2 \}$, and prove that coefficients are bounded on each set separately.

	  As $B_{\tau}$ does not vanish anywhere, since $E_4$ and $E_6$ have no common zeroes, these coefficients are continuous functions of $\tau$, and in fact, examining the $q$-expansions of the coefficients of $\Omega_{\tau}$ obtained in Lemma \ref{q-expansions and constant term lemma} holomorphic functions on the unit disc, in the variable $q^{1/(8n)}$. Thus they must be bounded in any neighbourhood of $0$. In term of the variable $\tau$, this means there exists a constant $M_1$ that is an upper bound for the absolute values of these coefficients on the set $A_1:=\{\tau \in A : \textrm{Im}(\tau) \ge C_1\}$. 
	
	Now observe that for any $g\in \mathrm{SL}_{2}(\mb{Z})$, there is an isomorphism of complex tori $\mb{C}/\langle 1, \tau \rangle \xrightarrow{} \mb{C}/\langle 1, g \cdot \tau \rangle$. Hence the Heisenberg invariant diagrams $B_{\tau}$ and $B_{g\cdot \tau}$ are isomorphic over $\mb{C}$, and in particular, there exists $M_g \in \mathrm{PGL}_{n}(\mb{C})$ such that $[C_{g \cdot \tau} \xrightarrow{} \mb{P}^{n-1}]=M_g\cdot [ C_{\tau} \xrightarrow{} \mb{P}^{n-1}]$. By Lemma \ref{discriminant form change of coord lemma}, we have $D_{\tau}(u_0,\ldots,u_{n-1})=D_{g \cdot \tau}((u_0,\ldots,u_{n-1}) \cdot M^{T}_g)$.
	
	 Thus, if the coefficients of $D_{\tau}$ are bounded on some subset of $\mc{H}$, they will be bounded, possibly by a different constant, on any $\mathrm{SL}_2(\mb{Z})$-translate of that subset. 
	 By choosing $g_1,g_2 \in \mathrm{SL}_2(\mb{Z})$ that map the cusp at $\infty$ to cusps at $0$ and $n/2$ respectively, and noting that $\textrm{Re}(\tau) \in \{0,n/2\}$, we deduce that the coefficients are bounded on the set $A_3:=\{\tau\in A : \textrm{Im}(\tau) \leq C_2 \}$
	
	As the set $A_2:=\{\tau \in \mc{H} : C_2 \leq \textrm{Im}(\tau) \leq C_1, \textrm{Re}(\tau) \in \{0,n/2\} \}$ consists of two closed and bounded intervals, it is compact, and hence the coefficients are uniformly bounded on $A_3$, concluding the proof.

\end{proof}

Recall now the setting of the Theorem \ref{Main theorem I}. Let $E/\mb{Q}$ be an elliptic curve, with a minimal Weierstrass equation $W$, and corresponding invariant differential $\omega$. The naive height $H_E$ is defined to be equal to $H_{W}$. Let $[C\xrightarrow{} \mb{P}^{n-1}]$ be an $n$-diagram that represents an element of the $n$-Selmer group $\mathrm{Sel}^{(n)}(E/\mb{Q})$. We wish to prove the existence of a constant $c(n)$, that depends only on $n$, such that $C$ admits a point defined over a $\mb{Q}$-algebra $A$, of degree $n$, and of discriminant at most $c(n) H_E^{2n-2}$.
\begin{proof}[Proof of Theorem \ref{Main theorem I}]
	By Theorem \ref{global minimization theorem}, we may assume that $[C \xrightarrow{} \mb{P}^{n-1}]$ admits a minimal integral model, i.e. a resolution model $F_{\bu}$, defined over $\mb{Z}$, with the  $c$-invariants of the corresponding $\Omega$-matrix given by $c_k(\Omega)=c_k(W)$.
	
	 Let $D_C$ be the discriminant form associated to the model $F_{\bu}$. By Lemma \ref{Z-order}, for any $n$ integers $u_0,\ldots,u_{n-1}$, not all zero, $D_C(u_0,\ldots,u_{n-1})$ is the discriminant of an order in the $n$-dimensional $\mb{Q}$-algebra $\Gamma(C\cap\{u_0x_0+\ldots+u_{n-1}x_{n-1}\})$. 
	 
	 As explained in Section \ref{analytic over real}, the diagram $[C \xrightarrow{} \mb{P}^{n-1}]$ and the diagram $B_{\Lambda}=[\mb{C}/\Lambda \xrightarrow{} \mb{P}^{n-1}]$ are isomorphic over $\mb{R}$. Hence there exists a $g \in \mathrm{SL}_n(\mb{R})$ with $g \cdot B_{\Lambda}=[C \xrightarrow{} \mb{P}^{n-1}]$. By Theorem \ref{analytic jacobian}, the matrices $\Omega_{\Lambda}$ and $\Omega_C$ have the same invariants, so we must have $\Omega_{C}=g\star \Omega_{\Lambda}$. As the discriminant form is contravariant, we have $D_{\Lambda}=D_{C}\circ g^{T}$, i.e. $D_{C}=D_{\Lambda}\circ g^{-T}$.
	 
	 Now let $K$ be the $n$-ball centered at $0 \in \mb{R}^n$ of volume $2^n$. The lattice $L=g^{-T}(\mb{Z}^n)$ has covolume one, so by Minkowski's theorem, $K$ contains a non-zero element $a$ of $L$. By Theorem \ref{compact_inequality}, we have $|D_{C}(g^{T} \cdot a )|= |D_{\Lambda} (a)| \leq c(n,K) H_E^{2n-2}$. As $g^{T} \cdot a \in \mb{Z}^n$, the value $D_{C}(g^{T} \cdot a )$ is the discriminant of an order in an  $n$-dimensional $\mb{Q}$-algebra $A$. 
\end{proof}

As a corollary, we obtain Theorem \ref{Main theorem II}. Assume that the curve $C$ has index $n$, i.e. has no positive divisors of degree less than $n$, that are defined over $\mb{Q}$. The algebra $A$ in the above proof is the ring of functions on a hyperplane section $C \cap H$. This section must consist of $n$ distinct points, as otherwise, $C$ would admit a $\mathrm{Gal}(\mb{\bar{Q}}/\mb{Q})$-invariant divisor of degree less than $n$. Hence $A$ must be étale, and can be decomposed as a product of number fields $A=K_1 \times K_2 \ldots \times K_{p}$. Using the assumption on the index again, we conclude that $A=K$ is a number field, and the theorem follows.

\bibliographystyle{amsalpha}
\bibliography{references}
\end{document}